\documentclass[12pt]{article}

\usepackage{amsmath,amsthm,latexsym,amsfonts,amssymb,epsfig,graphicx,setspace, color}
\usepackage{enumitem}

\usepackage{mathtools}
\usepackage{verbatim}
\usepackage{charter, tabularx}
\usepackage[english]{babel}
\usepackage[utf8]{inputenc}
\usepackage[T1]{fontenc}
\usepackage[table]{xcolor}
\usepackage{multirow}
\usepackage{tikz}

\usepackage[mathscr]{euscript}
\usetikzlibrary{arrows, math, shapes, positioning, shadows, trees}

 \usepackage{relsize}

\newcommand{\free}{\mathop{\Huge \mathlarger{*}}}

\usepackage{hyperref}
\hypersetup{
    colorlinks=true,
    linkcolor=blue,
    filecolor=magenta,      
    urlcolor=cyan,
}
 
\def\acts{\curvearrowright}
\newtheorem{theorem}{Theorem}[section]

\newtheorem{corollary}[theorem]{Corollary}

\newtheorem{lemma}[theorem]{Lemma}

\newtheorem{proposition}[theorem]{Proposition}
\newtheorem{notation}[theorem]{Notation}

\newtheorem{case}[theorem]{Case}

\theoremstyle{definition}
\newtheorem{definition}[theorem]{Definition}

\newtheorem{example}[theorem]{Example}

\newtheorem{remark}[theorem]{Remark}

\newtheorem{problem}{\bf Problem}

\numberwithin{equation}{section}

\newcommand{\h}[2]{\ensuremath{H_{#1}^{#2}}}
\newcommand{\la}{\ensuremath{\left\langle}}
\newcommand{\ra}{\ensuremath{\right\rangle}}

\newcommand{\Tw}[7]{

\tikzmath{\x = #1;\y = #2;}
\draw  (\x, \y + 1)--(\x + 2, \y - 1);
\draw  (\x + 2, \y + 1)--(\x , \y - 1);
\draw [fill=black] (\x , \y + 1) circle (1.5pt) node[anchor=east]{$#3$};
\draw [fill=black] (\x + 2, \y + 1) circle (1.5pt) node[anchor=west]{$#4$};
\draw [fill=black]  (\x, \y - 1) circle (1.5pt) node[anchor=east]{$#5$};
\draw [fill=black] (\x + 2, \y - 1) circle (1.5pt)node[anchor=west]{$#6$};
\node at (\x + 1, \y - 1.5) {#7};

}

\tikzstyle{block} = [rectangle, draw, fill=blue!20, 
   text width=8em, text centered, rounded corners, minimum height=4em]
\tikzstyle{line} = [draw, -latex']
\tikzstyle{cloud} = [draw, ellipse,fill=red!20, text width=7em, text centered, node distance=3cm,
    minimum height=1em]
\begin{document}

\pagenumbering{roman} \thispagestyle{empty}

\title{Thickness of $\mathsf{Out}(A_1*...*A_n)$}

\author{\footnote{The author received partial support from the National Science Foundation under Grant No.~DMS-1406376, principal investigator Lee Mosher.} Saikat Das}
\maketitle

\begin{abstract}

In this paper we have examined $\Gamma_n := \mathsf{Out}(G_n)$ from the perspective of geometric group theory, where $G_n = A_1*...*A_n$, is a finite free product and each $A_i$ is a finite group. We wanted to inspect hyperbolicity and relative hyperbolicity of such groups. We used the $\mathsf{Out}(G_n)$ action on the Guirardel-Levitt deformation space, \cite{GL}, to find a virtual generating set and prove quasi isometric embedding of a large class of subgroups. To prove non-distortion we used arguments similar to those used in \cite{HM} and \cite{Ali}. We used these subgroups to prove that $\Gamma_n$ is thick for higher complexities. Thickness implies that the groups are non relatively hyperbolic for higher complexities, \cite{BDM}.

\end{abstract}

\tableofcontents

\pagebreak


\pagenumbering{arabic} \pagestyle{myheadings} \markboth{}{}

\section{Introduction}\label{section:intro}

Our research has been motivated by trying to answer the following questions:

\begin{problem}\label{problem:main}
If each $A_i$ is a finite group, then is $\Gamma_n := \mathsf{Out}(A_1*A_2*...*A_n)$ hyperbolic? If the answer is no, then is it relatively hyperbolic?
\end{problem}

Questions similar to these have been answered for $\mathsf{Out}(F_n)$ by Behrstock-Dru\c{t}u-Mosher\cite{BDM}. In case of mapping class groups, $\mathsf{MCG}(S)$, They have been independently answered by Karlsson-Noskov \cite{KN}, Bowditch \cite{bowh}; Anderson-Aramayona-Shackelton\cite{AAS}; Behrstock-Dru\c{t}u-Mosher  \cite{BDM}. These are two of the most studied groups in geometric group theory. 

\subsection{Main theorem}
The following theorem answers the original question. 
\begin{theorem} \label{thm:main} If each $A_i$ is finite group, and $\Gamma_n := \mathsf{Out}(A_1 *...*A_n)$, then for
\begin{enumerate}[wide, labelwidth=!, labelindent=0pt]
\item $n \leq 2, \Gamma_n$ is finite.
\item $n = 3, \Gamma_n$ is infinite hyperbolic. 
\item $n > 3$, $\Gamma_n$ is a thick group of order at most one. As a consequence, $\Gamma_n$ is non relatively hyperbolic when $n >  3$. 
\end{enumerate}
\end{theorem}
\begin{remark}
Hyperbolicity for $n = 3$ was proved by Collins, \cite{col}. We will give an independent proof of hyperbolicity in lower complexities ($n \leq 3$) using theorem \ref{thm:contractible} and the topology of the deformation space. 
\end{remark}
\subsection{Methodology}\label{subsec:methods}
We have employed the following notable tools in our investigation - 
\begin{enumerate}[wide, labelwidth=!, labelindent=0pt]
\item \textbf{Deformation space of $G$-trees}, is a geodesic metric space on which $\Gamma_n$ acts by isometries such that the action is properly discontinuous. We follow the work of Guirardel-Levitt \cite{GL}, which is the most general theory of such spaces. Culler-Vogtmann spaces, see \cite{CV}, are examples of Guirardel-Levitt deformation spaces. The outer automorphisms we have used for understanding the action resemble the symmetric outer automorphisms investigated by McCullough-Miller, see \cite{MM}. 

\item \textbf{Algebraic thickness} of groups introduced by Behrstock-Dru\c{t}u-Mosher, see \cite{BDM}. Thickness is sufficient to conclude non-relative-hyperbolicity, see theorem \ref{thm:thickNRH}.
\end{enumerate}

 Guirardel-Levitt showed:
\begin{theorem}\label{thm:contractible}\cite[Theorem 6.1]{GL} $\mathcal{D}(G, \mathcal{H})$ is contractible.\end{theorem}
For $n > 3$, in addition to theorem \ref{thm:contractible} we use our understanding of $\Gamma_n$ and its action on $\mathcal{D}(G, \mathcal{H})$ to inspect their thickness. Behrstock-Dru\c{t}u-Mosher showed:
\begin{theorem}\label{thm:thickNRH}
\cite[Corollary 7.9]{BDM} If $G$ is a finitely generated group which is thick, then $G$ is not relatively hyperbolic.
\end{theorem}
 $\mathsf{Out}(F_n)$, $\mathcal{MCG}(S)$ and some other classes of geometrically interesting groups are thick for all but finitely many cases and hence non relatively hyperbolic. 

To prove thickness we have to find suitable \textit{undistorted}, zero thick subgroups of $\Gamma_n$. A subgroup is undistorted in $\Gamma_n$, if a Cayley graph of the subgroup can be quasi isometrically embedded in a Cayley graph of $\Gamma_n$. We use ideas from Handel-Mosher \cite{HM} to find a coarse Lipschitz retract from the spine of the deformation space to a sub-complex of the spine. Additionally we use ideas from Alibegovi\'c \cite{Ali} to prove non-distortion of another class of subgroups. A full justice to these ideas cannot be done in this short introduction; nonetheless, we would like to mention that one of the most innovative geometric idea in this work can be found in the definition of coarse Lipschitz-retraction map, see definition \ref{defn:retraction}. The author would like to express his gratitude towards Lee Mosher for this idea and most of the other ideas in this work.

Anthony Genevois has communicated that there is a nice argument for proving that $\Gamma_n$ is NRH for $n > 7$ which depends on \cite{GeM}.


\section{Organization}
In this section we will give a brief summary of each section in this exposition.

\begin{enumerate}[wide, labelwidth=!, labelindent=0pt]

\item[Section \ref{section:intro}: ] In the introductory section, we have stated the main question, Problem \ref{problem:main}. We have then stated our answer to the question in theorem \ref{thm:main}. We have also briefly discussed the methodologies, subsection \ref{subsec:methods}, used to investigate the question.

\item[Section \ref{section:background}: ] In this section we have discussed some of the basic definitions and results in geometric group theory, which are relevant to our research. The reader can skip this section if the reader feels comfortable about the notions of quasi isometry, Milnor-\u{S}varc lemma, hyperbolicity, relative hyperbolicity, Bass-Serre theory, undistorted subgroups and first barycentric subdivision.

\item[Section \ref{section:SPD}: ] In section \ref{section:SPD}, we have defined the deformation space, definition \ref{defn:def_sp} and in section \ref{subsec:def_sp} we have described the topology and geometry of the deformation space using collapse-expand moves (deformations). The contractibility of the deformation space in this topology is due to the work of Guirardel-Levitt, theorem \ref{thm:contractible}. We conclude the section by proving that $\Gamma_n$ acts geometrically on the spine of the deformation space, $\mathcal{SPD}(G,\mathscr{H})$ (remark \ref{rem:action}). The homotopy equivalence of the deformation space and its spine follows from lemma \ref{lemma:bsd}.

\item[Section \ref{section:G2G3}: ] In section \ref{subsec:G2}, we have proved the finiteness of $\Gamma_2$ using the triviality of $\mathcal{SPD}(G_2, \mathscr{H})$. An important consequence of this section is the uniqueness (up-to homeomorphism) of $A_i*A_j$-minimal sub-tree, discussed in remark $\ref{rem:unique_min}$. This uniqueness has been exploited in various times in sections \ref{section:fin_index}, \ref{section:undistorted}, to prove the ideas circling the most important results. In section \ref{subsec:G3} we have proved that $\Gamma_3$ is hyperbolic. This is the only result that uses the full power of the contractibility of the deformation space, theorem \ref{thm:contractible}; elsewhere we have used path connectedness of deformation space. Guirardel-Levitt has given credit to Max Forester \cite{For} for the proof of path connectedness of the deformation space. 
\item[Section \ref{section:fin_index}: ] In section \ref{section:fin_index}, we have considered a subgroup $\Gamma_n' \leq {\Omega_n} \leq \Gamma_n$, where $\Omega_n$ is the subgroup that fixes conjugacy class of each element in $\Gamma_n$ and $\Gamma_n'$ is generated by the outer automorphisms that have a representative automorphism which act by identity on at least one of the factors. We showed that $\Gamma_n'$ is finite index in $\Gamma_n$. The idea of the proof is to find a connected sub-complex of $\mathcal{SPD}(G_n,\mathscr{H})$ on which $\Gamma_n'$ and $\Gamma_n$ acts where the actions are co-compact and properly discontinuous. The essential part of the proof is to establish path connectedness of the sub-complex, which has been done in corollary \ref{cor:xy_connected}. 
\item[Section \ref{section:thick}: ] Careful inspection of the definition of $\Gamma_n'$ (definition \ref{defn:GN'}), made it clear that there is a substantial collection of subgroups, which are direct products of infinite subgroups, when $n \geq 4$. Once we observed the presence of these infinite subgroups, our motivation was to find a thickly connected network, definition \ref{defn:network}, of $\Gamma_4'$. So, a reader could start from section \ref{section:thick}, and \textit{see} that $\Gamma_n'$ has a thickly connected network, corollary \ref{cor:thickly_connected} . To prove thickness of $\Gamma_n$, definition \ref{defn:thickness}, we had to prove that $\Gamma_n'$ is finite index in $\Gamma_n$ (section \ref{section:fin_index}) and the subgroups in the network are undistorted in $\Gamma_n'$ (section \ref{section:undistorted}).

\item[Section \ref{section:undistorted}: ] In section \ref{section:undistorted}, we prove non distortion of certain classes of subgroups, corollaries \ref{cor:hij_undistorted}, \ref{cor:nij_undistorted}, \ref{cor:M4_undistorted}. The idea of the proofs of corollaries \ref{cor:hij_undistorted}, \ref{cor:M4_undistorted} are similar. We found a sub-complex of $\mathcal{SPD}$ which has a Lipschitz retraction from $\mathcal{SPD}$ and are quasi-isometric to these sub-groups. This idea draws inspiration from Handel-Mosher's paper \cite{HM}. The idea of the proof of corollary \ref{cor:nij_undistorted} has been motivated by Alibegovi\'c's work \cite{Ali}.

\item [Section \ref{section:summary}: ] In this section we have organized the our conclusions to give a complete overview of the proof of the theorem.

\end{enumerate}


\section{Definitions and Preliminaries}\label{section:background}

In this section, we will define and describe some of the fundamental concepts of geometric group theory. In geometric group theory, often the object of study is a geodesic metric space and a subgroup of its symmetry group. From another point of view the object of study is a group and its action on a geodesic metric space.

\subsection{Fundamental Observation in Geometric Group Theory}

The following fundamental observation in geometric group theory connects a group with the geodesic metric space on which it is acting.

\begin{lemma}[Milnor\cite{Mil}-\u{S}varc\cite{Sva} lemma]\label{lemma:mil_sva}
 For any group $G$ and any proper geodesic metric space $X$, if there exists a properly discontinuous, co-compact, isometric action $G \acts X$ then $G$ is finitely generated. Furthermore, for any such action and any point $x \in X$, the orbit map $g \mapsto g \cdot x$ is a quasi-isometry $\mathcal{O}: G \rightarrow X$, where $G$ is equipped with the word metric of any finite generating set.
\end{lemma}

\begin{definition}[Geodesic metric space]
    In a \textbf{geodesic metric space} we can define and measure length of any path using a function called \textbf{metric}. Additionally, any two points in the space can be connected by a shortest length path called \textbf{geodesic}.
 \end{definition}
 
 \begin{definition}
 A metric space is proper if a closed ball is compact.
 \end{definition}
 
Every finitely generated group act on its \textit{Cayley graphs} by isometries. A \textit{geometry} of a finitely generated group means the geometry of a Cayley graph of the group. Equivalently, it means the geometry of any geodesic metric space on which the group acts \textit{geometrically} (properly and co-compactly).

\begin{definition}[Properly discontinuous action]
An action of a finitely generated group $G$ on a geodesic metric space $(X, d)$ is \textbf{properly discontinuous} if $\forall x \in X$, there is a neighborhood $U_x$ of $x$ such that the set \\ $\{g \in G \vert g \cdot U_x \cap U_x \ne \phi\}$ is a finite set.
\end{definition}

\begin{definition}[Co-compact action]
An action of a finitely generated group $G$ on a geodesic metric space $(X, d)$ is \textbf{co-compact} if the quotient $G / X$ is compact.
\end{definition}

\begin{definition}[Cayley graph]
The \textbf{Cayley graph} of a group with respect to a finite generating set is a metric space on which the group acts geometrically. Given a finitely generated group $G$ and a finite generating set $S$, the Cayley graph of $G$ is a graph with vertex set labeled by group elements and two vertices labeled by group elements $g_1 \text{ and } g_2$ are connected by an edge directed from the former vertex to the latter vertex if $g_1^{-1}g_2$ is an element of $S$. If we assign length $1$ to each edge and define the distance between any two vertices on the Cayley graph by the minimum number of edges required to connect the two vertices, then the Cayley graph can be realized as a geodesic metric space. The metric on a Cayley graph is a \textbf{word metric} on $G$ with respect to the generating set $S$.

\begin{center}
\begin{tikzpicture}
\begin{scriptsize}
\draw[<->] (3,-1.5) to (7, -1.5);
      \draw[<->] (3,-1) to (7, -1);
      \draw[<->] (3,-.5) to (7, -.5);
      \draw[<->] (3, 0) to (7, 0);
      \draw[<->] (3,.5) to (7, .5);
      \draw[<->] (3,1) to (7, 1);
      \draw[<->] (3,1.5) to (7, 1.5);
       
       	\draw[<->] (3.5, -2) to (3.5, 2);
       	\draw[<->] (4, -2) to (4, 2);
       	\draw[<->] (4.5, -2) to (4.5, 2);
       	\draw[<->] (5, -2) to (5, 2);
     	\draw[<->] (5.5, -2) to (5.5, 2);
      	\draw[<->] (6, -2) to (6, 2);
	\draw[<->] (6.5, -2) to (6.5, 2);
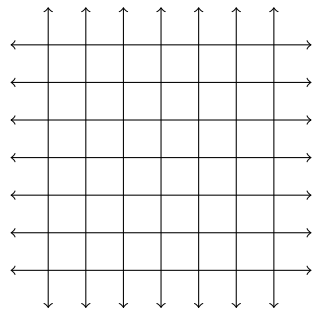
	\node at (5, -2.3) [anchor=north]{An unlabeled Cayley graph of $\mathbb{Z} \oplus \mathbb{Z}$};
\end{scriptsize}
 \end{tikzpicture} 
     \end{center}
\end{definition}

One of the main objectives in geometric group theory is to classify geodesic metric spaces up-to \textit{quasi isometry}. Quasi isometry captures large scale geometric behaviors of metric spaces.
\begin{definition}[Quasi Isometry]
A geodesic metric space $(X, d_X)$ is said to be $(K, C)$-quasi isometrically embedded for $k \geq 1, C \geq 0$ in a geodesic metric space $(Y, d_Y)$ if there is a function $f: (X, d_X) \rightarrow (Y, d_Y)$ which follows the following inequality $\forall x_1, x_2 \in X$\\
$\dfrac{1}{K}d_X(x_1, x_2) - C \leq d_Y(f(x_1), f(x_2)) \leq K d_X(x_1, x_2) + C$. 
$f$ is called a quasi isometric embedding.

Additionally, $f$ is a quasi isometry if there is a $D \geq 0$ such that $\forall y \in Y, \exists x \in X$ with $d_Y(f(x), y) \leq D$. In this case $X$ and $Y$ are said to be quasi isometric.
\end{definition} 

One of the most prominent quasi-isometry invariant is \textit{hyperbolicity}. In other words a non hyperbolic space cannot be quasi isometric to a hyperbolic space.

 \begin{definition}[Hyperbolicity]
A geodesic metric space is called \textbf{hyperbolic} if all geodesic triangles are $\delta$-thin for some fixed $\delta \geq 0$, i.e., any point on one side is within a distance $\delta$ of other two sides. A group is hyperbolic if one of its Cayley graphs is hyperbolic.
    
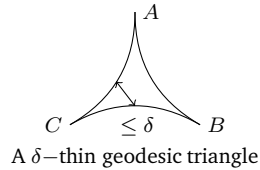
\begin{center}
    \begin{tikzpicture} \begin{scriptsize}
      \draw[-] (0,1) to [bend right](.865, -.5);
      \draw[-] (0,1) to [bend left] (-.865, -.5);
      \draw[-] (.865, -.5) to [bend right] (-.865, -.5);
      \draw[<->] (-.25, .075) to (0, -.25);
      \node at (0,1) [anchor=west]{$A$};
      \node at (.865,-.5) [anchor=west]{$B$};
      \node at (-.865,-.5) [anchor=east]{$C$};
      \node at (-.3, -.75) [anchor=south west]{$\leq \delta$};
      \node at (0, -.7) [anchor=north]{A $\delta-$thin geodesic triangle};

     \end{scriptsize} \end{tikzpicture} 
     \end{center}
  \end{definition}
  
\begin{example}
A tree with length of each edge $1$ is a $0$-hyperbolic geodesic metric space.

 
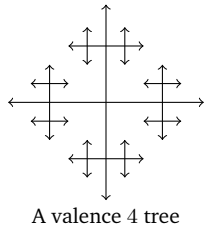
\begin{center}
    \begin{tikzpicture} \begin{scriptsize}
      \draw[<->] (0,-1.3) to (0, 1.3);
      \draw[<->] (1.3, 0) to (-1.3, 0);
      \draw[<->] (.75,.5) to (.75, -.5);
      \draw[<->] (-.75,.5) to (-.75, -.5);
      \draw[<->] (.5,-.75) to (-.5, -.75);
      \draw[<->] (.5,.75) to (-.5, .75);
       \draw[<->] (.5,.25) to (1, .25);
       \draw[<->] (.5,-.25) to (1, -.25);
       \draw[<->] (-.5,.25) to (-1, .25);
       \draw[<->] (-.5,-.25) to (-1, -.25);
        \draw[<->] (-.25,-.5) to (-.25, -1);
       \draw[<->] (.25,-.5) to (.25, -1);
        \draw[<->] (-.25,.5) to (-.25, 1);
       \draw[<->] (.25,.5) to (.25, 1);
      
      \node at (0, -1.3) [anchor=north]{A valence $4$ tree};
    \end{scriptsize} \end{tikzpicture} 
    \end{center}  
      
\end{example}

\textit{Relative hyperbolicity} serves as a quasi-isometry invariant for the groups which fail to be hyperbolic. A group is relatively hyperbolic, if we can construct a hyperbolic space which follows an additional technical condition by converting a collection of infinite diameter regions in a Cayley graph of the group to finite diameter regions using a method called \textit{coning off}.

\begin{definition}[Relative hyperbolic groups]
If $G$ denotes a finitely generated group, $\mathcal{H} = \{H_1, . . . , H_n\}$ is a finite family of subgroups of $G$ and $\mathcal{LH}$ denotes the collection of left cosets of $\{H_1, . . . , H_n\}$ in $G$. The group $G$ is \textbf{weakly hyperbolic relative to} $\mathcal{H}$ if collapsing the left cosets in $\mathcal{LH}$ to finite diameter sets, in a Cayley graph of G, yields a $\delta$-hyperbolic space. The subgroups $H_1, . . . , H_n$ are called peripheral subgroups. The group $G$ is (strongly) \textbf{hyperbolic relative to $\mathcal{H}$} if it is weakly hyperbolic relative to $\mathcal{H}$ and if it has the bounded coset penetration property (BCP). BCP property, roughly speaking, requires that in a Cayley graph of $G$ with the sets in $\mathcal{LH}$ collapsed to bounded diameter sets, a pair of quasi-geodesics with the same endpoints travels through the collapsed $\mathcal{LH}$ in approximately the same manner, see \cite{Fa, Os, Bow}. When a group contains no collection of proper subgroups with respect to which it is relatively hyperbolic, we say the group is \textbf{non relatively hyperbolic}, (NRH).
\end{definition}

\begin{example} $\mathbb{Z} \oplus \mathbb{Z}$ is weakly hyperbolic relative to $\mathbb{Z}$ but not relatively hyperbolic. In fact $\mathbb{Z} \oplus \mathbb{Z}$ is NRH. If $A$ and $B$ are finitely generated groups, $A*B$ is hyperbolic relative to subgroups $A$ and $B$. \end{example}

\begin{center}
\begin{tikzpicture} \begin{scriptsize}
      
      \draw[<->] (3,-1.5) to (7, -1.5);
      \draw[<->] (3.1,-1) to (6.9, -1);
      \draw[<->] (3.2,-.55) to (6.8, -.55);
      \draw[<->] (3.3, -.15) to (6.7, -.15);
      \draw[<->] (3.4,.2) to (6.6, .2);
      \draw[<->] (3.5,0.5) to (6.5, 0.5);
      \draw[<->] (3.6,.75) to (6.4, .75);
       
       	\draw[<->] (3.5, -2) to (4.1, 1);
       	\draw[<->] (4, -2) to (4.35, 1);
       	\draw[<->] (4.5, -2) to (4.7, 1);
       	\draw[<->] (5, -2) to (5, 1);
     	\draw[<->] (5.5, -2) to (5.3, 1);
      	\draw[<->] (6, -2) to (5.65, 1);
	\draw[<->] (6.5, -2) to (5.9, 1);
	
	\fill[gray, opacity = 0.2] (3.5, -1.5) -- (8, -.55) -- (6.5, -1.5);
	\fill[gray, opacity = 0.2] (3.6, -1) -- (2, -.15) -- (6.4, -1);
	\fill[gray, opacity = 0.15] (3.7, -.55) -- (8, .2) -- (6.3, -.55);
	\fill[gray, opacity = 0.15] (3.8, -.15) -- (2, .5) -- (6.2, -.15);
	\fill[gray, opacity = 0.1] (3.9, .2) -- (8, .75) -- (6.1, .2);
	\fill[gray, opacity = 0.1] (4, .5) -- (2, .95) -- (6, .5);
	\fill[gray, opacity = 0.05] (4.1, .75) -- (8, 1.1) -- (5.9, .75);
	
	\draw [dashed] (3.5, -1.5) -- (8, -.55) -- (4, -1.5);
	\draw [dashed] (4.5, -1.5) -- (8, -.55) -- (5, -1.5);
	\draw [dashed] (5.5, -1.5) -- (8, -.55) -- (6, -1.5);
	\draw [dashed] (6.5, -1.5) -- (8, -.55);
	
	\draw [dashed] (3.5, -1) -- (2, -.15) -- (4, -1);
	\draw [dashed] (4.5, -1) -- (2, -.15) -- (5, -1);
	\draw [dashed] (5.5, -1) -- (2, -.15) -- (6, -1);
	\draw [dashed] (6.5, -1) -- (2, -.15);
	
	\draw [dashed] (3.7, -.55) -- (8, .2) -- (6.3, -.55);
	\draw [dashed] (3.8, -.15) -- (2, .5) -- (6.2, -.15);
	\draw [dashed] (3.9, .2) -- (8, .75) -- (6.1, .2);
	\draw [dashed] (4, .5) -- (2, .95) -- (6, .5);
	\draw [dashed] (4.1, .75) -- (8, 1.1) -- (5.9, .75);

	\node at (5, -2.3) [anchor=north]{A coned-off Cayley graph of $\mathbb{Z} \oplus \mathbb{Z}$};

     \end{scriptsize} \end{tikzpicture} 
     \end{center}
     
\begin{remark}

\item If $G$ is a finitely generated subgroup and $H \leq G$ is a finite index subgroup, then $H$ is quasi-isometric to $G$, where the quasi-isometry is given by the inclusion map. 

\item
\end{remark}

\subsection{Undistorted Subgroups}
The definition of algebraic thickness requires the existence of certain \textit{undistorted} subgroups. In this section we will define an undistorted subgroup of a finitely generated group and then discuss relevant properties of a subgroup to prove non-distortion.
\begin{definition}\label{defn:undistorted}
A finitely generated subgroup $H$ of a finitely generated group $G$ is said to be \textbf{undistorted} if the inclusion map of $H \xhookrightarrow{} G$
induces a quasi-isometric embedding of Cayley graphs.
\end{definition}

To prove non-distortion we have to prove only one side of the inequality, lemma \ref{lemma:undistorted}. The relevant side of the inequality can also be stated as a \textit{coarse Lipschitz map}

\begin{definition}[Coarse Lipschitz map]
For constants $K \geq 1, C \geq 0$, a function\\ 
$f: (X,d_X) \rightarrow (Y, d_Y)$ is $(K, C)-$\textbf{coarse Lipschitz} if \\
 $d_Y(f(x_1), f(x_2)) \leq Kd_X(x_1, x_2) + C$ for all $x_1, x_2 \in X$.
\end{definition}

The following results gives us a way of proving non-distortion using the action.

\begin{lemma}\cite[Corollary 10]{HM}
If $G$ is a finitely generated group acting properly discontinuously and
co-compactly by isometries on a connected locally finite simplicial complex $X$ , if
$H < G$ is a subgroup, and if $Y \subset X$ is a nonempty connected sub-complex which is
$H-$invariant and $H-$co-compact, then:
\begin{enumerate}[wide, labelwidth=!, labelindent=0pt]
\item $H$ is finitely generated.
\item $H$ is undistorted in $G$ if and only if the inclusion $Y \xhookrightarrow{}  X$ is a quasi-isometric
embedding.
\item The following are equivalent:
\begin{enumerate}[wide, labelwidth=!, labelindent=0pt]
\item $H$ is a Lipschitz retract of $G$.
\item The $0-$skeleton of $Y$ is a Lipschitz retract of the $0-$skeleton of $X$.
\item The $1-$skeleton of $Y$ is a Lipschitz retract of the $1-$skeleton of $X$.
\item $Y$ is a coarse Lipschitz retract of $X$.
\end{enumerate}
\end{enumerate}
\end{lemma}

\begin{lemma}\cite[Lemma 11]{HM}\label{lemma:undistorted}
If $X$ is a geodesic metric space and $Y \subset X$ is a rectifiable subspace, and if $Y$ is a coarse Lipschitz retract of $X$, then the inclusion $Y \rightarrow X$ is a quasi-isometric embedding.
\end{lemma}
     
\subsection{$G$-trees and graphs of groups}

Our understanding of $\mathsf{Out}(G)$ will be related to our understanding of a space of \textit{$G$-trees}, called \textit{deformation space} \cite{GL}. Roughly a deformation space is a space of \textit{$G$-equivariance} classes of \textit{$G$-trees}. The trees considered in this exposition are simplicial trees with metric. It may be convenient at times to only consider the underlying simplicial structure.
 
\begin{definition}[$G$-tree]
A group action of $G$ on a metric (resp., simplicial) tree $T$ via isometries (resp., simplicial homeomorphisms) is called \textbf{minimal}, if there are no proper $G$-invariant subtrees of $T$. A metric (resp., simplicial) tree on which $G$ acts minimally is called a \textbf{$G$-tree}.
\end{definition}

\begin{definition}[$G$-equivariant isometry]
Consider metric $G$-trees $T_1$ and $T_2$. $T_1$ and $T_2$ are \textbf{$G$-equivariantly isometric} if there is an isometry $f: T_1 \mapsto T_2$ such that $g \in G \implies f(g.x) = g.f(x),  \forall x\in T_1$
\end{definition}

A \textit{fundamental domain} for a $G$-tree gives us much relevant information about the action and the geometry and topology of the tree. In our research, a fundamental domain gives all the necessary information about the $G$-tree we are interested in.

\begin{definition}[Fundamental domain of a $G$-tree]
A subtree $F$ of a given $G$-tree is called a fundamental domain for the action $G \acts T$, if $G \cdot F \supset T$ and no other proper subtree of $F$ has this property.
\end{definition}

This interplay between a $G$-tree and its fundamental domain is captured by the Bass-Serre theory \cite{Se}.

\begin{definition}[Fundamental group of a graph of groups]
A \textbf{graph of groups} over a graph $X$ is an assignment 
\begin{enumerate}[wide, labelwidth=!, labelindent=0pt]
\item of a group $G_x$ to each vertex $x$ of $X$,
\item of a group $G_y$ to each edge $y$ of $X$, and 
\item monomorphisms $\phi_{y_0}$ and $\phi_{y_1}$ mapping $G_y$ into the groups assigned to the vertices at its ends. 
\end{enumerate}
Denote this graph of groups by $\mathbf{X}$.
If $X$ is a tree then define the \textbf{fundamental group} of $\mathbf{X}$ is defined as
$\Gamma :=\la G_x \vert x \in vert(X), \phi_{y_0}(e) = \phi_{y_1}(e) \forall e \in G_y\ra$
\end{definition}

\begin{theorem}\label{thm:bst}[Fundamental theorem of Bass-Serre theory]
Let $G$ be a group acting on a tree $T$ without inversions. Let $\mathbf{X}$ be the \textit{quotient graph of groups}. Then $G$ is isomorphic to the \textit{fundamental group of $\mathbf{X}$} and there is an equivariant isomorphism between the tree $T$ and the Bass-Serre covering tree $T_{\mathbf{X}}$ (definition \ref{defn:BSCoveringTree}). That is, there is a group isomorphism $i: G \rightarrow \Gamma $ and a graph isomorphism $j: T \rightarrow T_{\mathbf{X}}$ such that  $ \forall g \in G$,  $\forall$ vertex $x \in T$ and $\forall$ edge $e \subset T$ we have $j(g\cdot x) = g\cdot j(x)$ and $j(g\cdot e) = g\cdot j(e)$.

\end{theorem}

\begin{remark}
In general, $\Gamma$ is defined for any graph $X$ (and not only a tree). However, assuming $X$ to be a tree simplifies the definition and is sufficient for this exposition. This will also simplify the definition of the \textit{Bass-Serre covering tree}.
\end{remark}

\begin{definition}[Bass-Serre covering tree]\label{defn:BSCoveringTree}
For a given graph of groups $\mathbf{X}$ with fundamental group $\Gamma$ and its underlying tree $X$ (we are considering special case where $X$ is a tree), let $G_x$ represent vertex group of a vertex $x$ of $X$ and $G_e$ represent edge group of an edge $e$ of $X$. Then define the \textbf{Bass-Serre covering tree} of $\mathbf{X}$, $T_{\mathbf{X}}$, as follows:
\begin{enumerate}[wide, labelwidth=!, labelindent=0pt]
\item The vertex set of $T_{\mathbf{X}}$ is a disjoint union of points labeled by the cosets:
$vert(T_{\mathbf{X}}) := \displaystyle\bigsqcup_{x \in vert(\mathbf{X})} \Gamma / G_x$
\item The edge set of $T_{\mathbf{X}}$ is a disjoint union of edges labeled by the cosets:
$edge(T_{\mathbf{X}}) := \displaystyle\bigsqcup_{e \in edge(\mathbf{X})} \Gamma / G_e$
\item An inclusion of groups $G_e \xhookrightarrow{} G_x$ induces a natural surjection map at the level of cosets $\Gamma / G_e \twoheadrightarrow \Gamma / G_x$. An edge is incident on a vertex if the edge label maps to the vertex label under this surjection.
\end{enumerate}
\end{definition}

\subsection{Simplicial complex}
Deformation space has an invariant spine on which $\mathsf{Out}(G)$ acts geometrically. The following theorem implies that the sub-complex spanned by the barycentric coordinates (spine) is 'good enough' substitute if the property of interest is a homotopy invariant.

\begin{lemma}\label{lemma:bsd}
Let, $S$ be a connected subset of a finite dimensional simplicial complex, $\Delta$, such that $S$ is the complement of a sub-complex of $\Delta$, then $S$ deformation retracts onto $S_B$; where $S_B$ is the sub-complex of the $1^{st}$ barycentric subdivision of $\Delta$ consisting of all simplices that lie entirely inside $S$.

\end{lemma}

\begin{proof}
Let us denote the $1^{st}$ barycentric subdivision of $\Delta$ by $\Delta_B$. Let us assume that $S$ intersects a $k$-dimensional simplex of $\Delta$ denoted by $\sigma$, such that the $0$-simplices of $\sigma$ are denoted by $\{\alpha_0,..., \alpha_k\}$.  Let, $\sigma_B \subset \sigma$ be a k-dimensional simplex of $\Delta_B$ whose vertices are denoted by $\{\beta_0,..., \beta_k\}$ such that
$\beta_i = \displaystyle\sum_{j = 0}^i \dfrac{\alpha_j}{i + 1}$. 
Without loss of generality, assume that some of the vertices of $\sigma_B$ are not in $S$. As $S$ is the complement of a sub-complex of $\Delta$, so $\beta_i \notin S$ for some $i$ implies $\alpha_0 = \beta_0 \notin S$. Additionally, assume that - 
$\beta_i \notin S, \text{ when } i \in \{0,..., l\}$ and
$\beta_i \in S, \text{ when } i \in \{l+1,..., k\}$.

$\beta_i \notin S 
\implies 
\sigma \vert_{\{\alpha_0,..., \alpha_{i}\}} \cap S = \phi
\implies
\sigma_B \vert_{\{\beta_{0},..., \beta_i\} } \cap S = \phi
$

Now, we will define a projection map $r_{\sigma_B}$
\begin{gather*}r_{\sigma_B}: S \cap \sigma_B \rightarrow S_B \cap \sigma_B\\
a_0\beta_0 + a_1\beta_1 +...+ a_k\beta_k 
\mapsto
 \dfrac{a_{l+1}}{1 - \sum_{j = 0}^{l}a_j}\beta_{l+1} + ... + \dfrac{a_{k}}{1 - \sum_{j = 0}^{l}a_j}\beta_{k}\end{gather*}

 If $\sigma'_B$ is another simplex of $\Delta_B$ such that $\sigma_B \cap \sigma'_B \ne \phi$, then we will show that the map 
 $r_{\sigma_B}\vert_{\sigma_B \cap \sigma'_B} 
 = 
 r_{\sigma'_B}\vert_{\sigma_B \cap \sigma'_B}$.\\
 Assume that the $0$-simplices of $\sigma'_B \cap \sigma_B$ are given by $\{\beta_p, ..., \beta_s\}$, where 
 \begin{gather*}
 \beta_i \in S \iff i \in \{r+1, ..., s-1, s\} \\
 \text{With these notations,} \\
 r_{\sigma_B}\vert_{\sigma_B \cap \sigma'_B} (a_p\beta_p +...+ a_s\beta_s )\\
 =
  \dfrac{a_{r+1}}{1 - \sum_{j = p}^{r}a_j}\beta_{r+1} + ... + \dfrac{a_{s}}{1 - \sum_{j = p}^{r}a_j}\beta_{s}\\
  = r_{\sigma'_B}\vert_{\sigma_B \cap \sigma'_B} (a_p\beta_p +...+ a_s\beta_s )
 \end{gather*}
 Hence, $r_{\sigma_B}$ can be continuously extended to a map $r: S \rightarrow S_B$. By, definition $r \vert_{S_B}$ is the identity map. So, this map is a continuous deformation retract.
\end{proof}


\section{$\mathsf{Out}(A_1*...*A_n)$ and A Geometric Action}\label{section:SPD}

\subsection{$\Omega_n$ - a finite index subgroup of $\mathsf{Out}(A_1*...*A_n)$}

We want to know the coarse geometric structure of $\mathsf{Out}(A_1*...*A_n)$. It will be convenient to consider the maximal subgroup of the outer automorphism group that preserves the conjugacy class of every element of each $A_i, i \in \{1, ..., n\}$. In lemma \ref{lemma:con_class} we will prove that this subgroup is a finite index subgroup of $\mathsf{Out}(A_1*...*A_n)$. As a consequence, $\mathsf{Out}(A_1*...*A_n)$ will be quasi isometric to this subgroup.

\begin{definition}
Let, $S_n$ be the symmetry group on first $n$ natural numbers. Consider $\phi \in \Gamma_n$. If 
$\phi([A_i]) = [A_j], \text{ where } i, j \in \{1,..., n\},$
then define $s_\phi \in S_n$ to be the element such that 
$s_\phi(i) = j, \text{ here }i, j \in \{1,..., n\}.$
\end{definition}

\begin{lemma} \label{lemma:con_class}
Consider the following map from $\Gamma_n$ to the symmetry group on first $n$ natural numbers, $S_n$:
\begin{gather*}P: \Gamma_n \rightarrow S_n\\
\phi \mapsto s_\phi\end{gather*} $P$ is a homomorphism of groups and the kernel of the map is a subgroup of $\Gamma_n$ which preserves conjugacy classes of the free factors $A_i$.
\end{lemma}

\begin{proof}
Let $\phi_1, \phi_2 \in \Gamma_n$ be such that $\phi_1([A_i]) = [A_j]$ and $\phi_2([A_j]) = [A_k]$, then 
$\phi_2\phi_1([A_i]) = [A_k] \implies P(\phi_2\phi_1) = s_{\phi_2}s_{\phi_1}$. If $\phi([A_i])= [A_j]$, then
$\phi^{-1}([A_j]) = [A_i] \implies P(\phi^{-1}) = s_\phi^{-1}$.
Now we will show that the kernel of the map $P$ is the subgroup of $\Gamma_n$ which preserves the conjugacy classes of free factors.
$\phi \in ker(P)
\iff s_{\phi} = id_{S_n}
\iff s_\phi(i) = i, \forall i
\iff \phi([A_i]) = [A_i], \forall i$.
\end{proof}

\begin{definition}\label{defn:fgamma_n}
$\Omega_n$ is the finite index subgroup of $\mathsf{Out}(A_1*...*A_n)$ that preserves the conjugacy classes of the free factors.
\end{definition}

\subsection{Deformation Space}\label{subsec:def_sp}

We will study the algebra and geometry of $\Omega_n$ by studying an action of $\Omega_n$ on a complete geodesic metric space. This geodesic metric space will be a subspace of the deformation space. The deformation space will be defined as a metric space depending on the following parameters - the group, $A_1*...*A_n$, and a class of subgroups. As a set, the deformation space is the set of equivalence classes $G$-trees with an additional condition on the vertex stabilizers, where two trees are equivalent if they are $G$-equivariantly isometric. In this section we will define the deformation space.

\subsubsection{$\mathcal{D}(G,\mathscr{H})$ as a set of $G$-trees} Deformation space has been defined in \cite{GL} more generally. In contrast, we will consider the following definition of deformation space (as a set). This definition will result in a space on which $\mathsf{Out}(A_1*...*A_n)$ will act isometrically, and properly discontinuously.

\begin{definition}[Deformation space as a set of $G$-trees]\label{defn:def_sp}
Consider a group $G$, which is a free product of finitely many finitely generated indecomposable groups. The deformation space $\mathcal{D}(G,\mathscr{H})$ of $G$ with respect to a collection of finitely generated subgroups $\mathscr{H}$ is the set of equivalence classes of minimal, metric $G$-trees with the following properties:

\begin{enumerate}[wide, labelwidth=!, labelindent=0pt]

\item \textbf{Equivalence relation}: Two trees $T_1, T_2$ are equivalent in $\mathcal{D}(G,\mathscr{H})$ if there is a $G$-equivariant isometry between the two trees.

\item \textbf{Vertex stabilizers belong to $\mathscr{H}$}: If $T$ is a $G$-tree from the equivalence class $[T] \in \mathcal{D}(G,\mathscr{H})$, and $v \in T$ is a vertex of $T$, then $Stab(v) \in \mathscr{H}$. Conversely, given a $H \in \mathscr{H}, \exists$ a vertex, $v \in T$, with $Stab(v) = H$. Moreover, valence of a vertex with trivial vertex stabilizer must be greater than $2$.

\item \textbf{Trivial edge stabilizer}: If $T$ is a $G$-tree from the equivalence class $[T] \in \mathcal{D}(G,\mathscr{H}) $, and $e$ is an edge of $T$; then Stab$(e) = \{id\}$.


\end{enumerate}

\end{definition}
\begin{remark}
The deformation space that we study in this exposition is a special case of the deformation space discussed in \cite{GL}. Here, the maximal elliptic subgroups of a tree are the conjugates of the free factors of $G$, which are also vertex stabilizers and the edge groups are trivial.

\end{remark}

\subsubsection{$\mathcal{D}(G,\mathscr{H})$ as a set of graph of groups} 
The goal of this subsection is to describe the deformation space as a set of equivalence classes of graph of groups whose fundamental group is a finitely generated group $G$. Consider a point of $\mathcal{D}(G,\mathscr{H})$, corresponding to a $G$-tree $T \in \mathcal{D}(G,\mathscr{H})$. This point also corresponds to a graph of groups $\mathbf{X}$, where $T$ is $G$-equivariantly isometric to the Bass-Serre tree of $\mathbf{X}$.

Consider a $G$-tree $T \in \mathcal{D}(G,\mathscr{H})$, $X := T\slash G$ is a finite graph. This is a result of the minimal action of the finitely generated group $G$ on $T$. If we choose a fundamental domain for the action of $G$ on $T$ we can associate a graph of groups, $\mathbf{X}$, to $T$. 

\begin{lemma}
Consider a $G$-tree $T$ such that the edge stabilizers are trivial and $X := T\slash G$ is a finite tree, then any fundamental domain of $T$ under the action of $G$ is isometric to $X$.
\end{lemma}

\begin{proof}
We will prove that any fundamental domain is isomorphic to the quotient $X := T\slash G$. 

Fix a fundamental domain of $T$ under the action of $G$ and name it $Y$. As the quotient is a tree, so no two points in $Y$ are in the same orbit. Hence, we can define a unique bijective map $f$ from $X$ to $Y$.\begin{gather*}f: X \rightarrow Y.\\
x \mapsto \text{the unique pre image of }x\end{gather*}
$f$ is an isometry as $G$ acts on $T$ by isometries.
\end{proof}

\begin{remark}\cite[Page 147]{GL}
If $T_1, T_2 \in \mathcal{D}(G,\mathscr{H})$ is a trees in a deformation space then, then rank of the quotient graphs $T_1/G$ and $T_2/G$ are same. Hence, the underlying graph of every quotient graph of groups in our case is a tree.
\end{remark}

\begin{lemma}
Consider a $G$-tree $T$ such that the edge stabilizers are trivial and $X := T\slash G$ is a finite tree. Fix a fundamental domain $Y$ of $T$, then $G$ is equal to the internal free product of vertex stabilizer subgroups of vertices in $Y$.
\end{lemma}

\begin{proof}
Consider a graph of groups $\mathbf{X}$ whose underlying tree is isometric to $X$ and the vertex group associated to a vertex of $\mathbf{X}$ under this isometry is the vertex stabilizer group of the corresponding vertex of $X$. Then, the Bass-Serre tree of $\mathbf{X}$ is equivariantly isometric to $T$. Hence, $G$ is isomorphic to internal the free product of vertex groups of $\mathbf{X}$.
\end{proof}

\subsubsection{A dictionary between two points of view of $\mathcal{D}(G, \mathscr{H})$}Consider a tree $T\in \mathcal{D}(G, \mathscr{H})$. We will now describe the graph of groups, $\mathbf{X}$, corresponding to $T$. The underlying graph of $\mathbf{X}$ is isomorphic to $X = T/G$. If $X$ is a tree then, $X$ is isomorphic to a fundamental domain of $T$ under the action of $G$. The vertex group associated to a vertex of $\mathbf{X}$ under this homeomorphism is the vertex stabilizer group of the corresponding vertex of $X$. The vertices having trivial vertex groups must have valence $3$ or more. The edge groups of $\mathbf{X}$ are trivial as the edge stabilizers of $T$ are trivial. 

Hence, a point of the deformation space can be represented by a graph of groups. Two graph of groups $\mathbf{X_1}$ and $\mathbf{X_2}$ represent the same point of $\mathcal{D}(G, \mathscr{H})$ if their Bass-Serre trees are $G$ equivariantly isometric.

\begin{remark}\label{rem:dictionary}
 We will use the following dictionary to change our viewpoint of $\mathcal{D}(G,\mathscr{H})$ from a set of trees to a set of graph of groups and vice versa.

\begin{enumerate}[wide, labelwidth=!, labelindent=0pt] \item Consider a tree $T \in \mathcal{D}(G,\mathscr{H})$, a graph of groups $\mathbf{X}_T \in \mathcal{D}(G,\mathscr{H})$ representing the point $T$ can be constructed once we fix a fundamental domain $F$ of $T$. $\mathbf{X}_T$ is isometric to $F$ as a graph and the vertex group associated to a vertex of $\mathbf{X}_T$ is the vertex stabilizer group of the corresponding vertex of $F$.

\item Consider a $\mathbf{X} \in \mathcal{D}(G,\mathscr{H})$ a tree $T_\mathbf{X} \in \mathcal{D}(G,\mathscr{H})$ representing the point $\mathbf{X}$ is the Bass-Serre tree of $\mathbf{X}$.

\end{enumerate}
\end{remark}


\subsection{Geometry of Deformation Space} 
\subsubsection{$\mathbb{O}_T$ - open cone of deformation space}
Consider an equivalence class, $[T] \in \mathcal{D}(G, \mathscr{H})$. Let, us choose a tree $T \in [T]$. If T has $k + 1$ orbits of edges then every tree in $[T]$ has $k + 1$ orbits of edges. In terms of the graph of groups, if $\mathbf{X}$ is a graph of groups corresponding to $T$; then $\mathbf{X}$ has $k+1$ edges. From this point, we will abuse notation and denote an equivalence class in $\mathcal{D}(G, \mathscr{H})$ by a tree (or a graph of groups) belonging to the equivalence class. Consider the set, $\mathbb{O}_T := \{S \in \mathcal{D}(G, \mathscr{H})\vert S \text{ is $G$-equivariantly homeomorphic to } T\}. \mathbb{O}_T$ can be naturally identified with the positive orthant of $\mathbb{R}^{k + 1}$ which induces a topology on $\mathbb{O}_T$. Consider a tree $S \in \mathbb{O}_T$ with edge lengths of distinct orbits of edges given by $e_0, e_1, ..., e_k$, then the identification map is described as follows:
\begin{gather*}\mathbb{O}_T \rightarrow \mathbb{R}^{k + 1}\\
S \mapsto (e_0, e_1, ..., e_k).\end{gather*}

Hence, $\mathbb{O}_T$ is realized geometrically as a metric space isometric to the positive orthant of $\mathbb{R}^{k + 1}$.



\subsubsection{Admissible collapse and expand moves in $ \mathcal{D}(G, \mathscr{H})$} A tree, $T \in \mathcal{D}(G, \mathscr{H})$ admits a \textbf{collapse move} if collapsing some edge orbits $G$-equivariantly, produces a tree $T' \in \mathcal{D}(G, \mathscr{H})$. \textbf{Admissible collapse move} is a relation $(T, T')$ on $\mathcal{D}(G,\mathscr{H})$. The inverse of an admissible collapse move is an \textbf{admissible expand move}. 

Admissible expand move can be defined independently. For a given fundamental domain $Y$ of $T$, a set of vertices $\{v_i\}$ of the fundamental domain; $T$ admits an expand move if all the $v_i$s follow one of the following two conditions:
\begin{enumerate}[wide, labelwidth=!, labelindent=0pt]
\item $stab(v_i) = id$ with valence of $v_i\vert_Y > 3$.
\item $stab(v_i) \ne id$ with valence of $v_i\vert_Y > 1$.
\end{enumerate}
This move produces a new tree $T' \in \mathcal{D}(G, \mathscr{H})$ by constructing a fundamental domain $Y'$ after a deformation of $Y$. We replace each $v_i$ by a finite sub-tree such that there are no vertices of valence $2$ or lower with trivial vertex stabilizer. We extend this operation equivariantly to all of $T$ to obtain $T'$.

If we apply an admissible collapse move on a tree $T$, then the resulting tree, $T'$, may not be in $\mathbb{O}_T$. In that case $T'$ may be associated to a point on the boundary of the positive orthant, with one or more $0$ coordinate, i.e., a point in one of the bounding hyperplanes of $\mathbb{O}_{T}$.

\begin{remark}
We have to administer our collapse moves cautiously, so that we do not produce a tree whose vertex stabilizer is not in the collection $\mathscr{H}$. Hence, the name admissible collapse move. Similarly, we have to exercise caution while applying expand moves to make sure that the resulting tree does not have a vertex of valence $\leq 2$ with trivial vertex stabilizer.
\end{remark}

\subsubsection{Boundary of $\mathbb{O}_T$}
After realizing $\mathcal{D}(G,\mathscr{H})$ as a collection of disjoint open orthants, our next goal is to give identification maps to the collection of open orthants. 

Let, $\mathbb{O}_{T'}$ be a $k'$ dimensional open simplex and $\mathbb{O}_T$ be a $k$ dimensional open simplex, where $k' \le k$. $\mathbb{O}_{T'}$ is a boundary of $\mathbb{O}_T$ if and only if $T'$ is isomorphic to a tree obtained collapse move applied on $T$.

\subsubsection{Topology of $ \mathcal{D}(G, \mathscr{H})$} A set in this space is open if and only if the intersection of the set with the relative interior of any simplex is a relative open subset of the simplex.

\subsubsection{$ \mathcal{PD}(G, \mathscr{H})$ - projectivized deformation space} $\mathbb{R} / \{0\}$ acts on the open cone and the quotient of the action can be identified with $\sigma_k = \{(e_0, e_1, ..., e_k)\vert \Sigma_{i = 0}^k e_i = 1 \}$, the $k$-dimensional open simplex in $\mathbb{R}^{k + 1}$. The sum of the edge lengths of different edge-orbits of a tree corresponding to a point of $\sigma_k$ is $1$.

\subsubsection{$ \mathcal{SPD}(G, \mathscr{H})$ - spine of $ \mathcal{PD}(G, \mathscr{H})$} Spine of the projectivized deformation space is a subspace of the projectivized deformation space. It is a simplicial complex spanned by the $0$-simplices corresponding to points having equal values on every coordinates. In other words $0$-simplices correspond to barycenter of every open simplex. For example, the $0$-simplex corresponding to a $k-$dimensional open simplex of the projectivized deformation space is given by \\
$\{(e_0, e_1, ..., e_k)\vert \Sigma_{i = 0}^k e_i = 1, e_0 = e_1 = ...= e_k \}.$ So, $\mathcal{SPD}(G,\mathscr{H})$ can be realized as the flag complex spanned by the barycenter of $\mathcal{PD}(G, \mathscr{H})$. Observe that $\mathcal{SPD}(G,\mathscr{H})$ is a simplicial complex, which is a deformation retract of $\mathcal{PD}(G,\mathscr{H})$ (lemma \ref{lemma:bsd}).

\begin{remark}
Contractibility of $\mathcal{D}(G,\mathscr{H}), \mathcal{PD}(G,\mathscr{H}), \mathcal{SPD}(G,\mathscr{H})$ follows from theorem \ref{thm:contractible}.
\end{remark}

\subsection{Action of $\mathsf{Out}(G)$ on $\mathcal{SPD}(G,\mathscr{H})$}

The goal of this section is to establish a geometric connection between $\mathsf{Out}(G)$ and $\mathcal{SPD}(G,\mathscr{H})$. We want to show that $\mathsf{Out}(G)$ action on $\mathcal{SPD}(G,\mathscr{H})$ is properly discontinuous and co-compact.

\begin{remark}
 For the rest of our exposition we will inspect the space $ \mathcal{D}(G, \mathscr{H})$ for 
\begin{gather*} G = G_n = A_1*...*A_n, \text{ where each } A_i \text{ is finite, and}\\ 
\mathscr{H} = \{H \leq G_n \vert H \text{ is conjugate to }A_i \text{ or the trivial subgroup}\}.
\end{gather*}
We may drop the subscript $n$ from $G_n$, if it is clear from the context.
\end{remark}

\subsubsection{Structure of graph of groups in $\mathcal{D}(G_n, \mathscr{H})$}

\begin{lemma} \label{lemma:count}
If a graph of groups has following properties:
\begin{enumerate}[wide, labelwidth=!, labelindent=0pt]
\item There are exactly $n$ non trivial vertex groups,  $n \geq 2$.
\item The vertices having trivial vertex groups have valence greater than $2$.
\item The edge groups are trivial.
\item The underlying graph is a finite tree.
\end{enumerate}
Then $n \leq \mathcal{V} \leq 2(n - 1), \text{ and }, (n - 1) \leq \mathcal{E} \leq 2n - 3$, where $\mathcal{V}, \mathcal{E}$ represent the number of vertices, edges of the underlying tree, respectively.
\end{lemma}

\begin{proof}
The second set of inequalities follow from the first set of inequalities because in a finite tree the number of vertices is $1$ more than the number of edges.

$n \leq \mathcal{V}$ follows from the fact that there are $n$ non trivial vertex groups. Additionally, the lower bound is attained by a tree isometric to $[0 , n - 1]$ with the integer points of the interval realized as the vertices.

We will prove the second half of the first inequality by induction. Let us inspect a finite tree of groups having two vertex groups. Such a tree has at most $2$ vertices of valence $1$. The underlying space is homeomorphic to the interval $[0 , 1]$, as we do not allow vertices of valence less than $3$ for trivial vertex groups. So, the only possible configuration is a tree with $2$ vertices and $1$ edge. Now, let us assume this statement is true is for $n = k$. That is, a graph of groups with $k$ non trivial vertex groups and following conditions $2, 3,\text{ and } 4$ above has at most $2(k - 1)$ vertices; and there exists a tree with $2(k - 1)$. Using this tree we will construct a tree with $2k$ vertices having $k + 1$ vertices of valence $\leq 2$. Take this tree and choose an interior point of an edge. Attach the interval $[0, 1]$ to this point by a quotient map where only the point $0$ from the interval gets identified to the chosen point. In the quotient space, define the image of $0$ and $1$ to be new vertices. This way the quotient space formed can be realized as a tree with exactly $2(k - 1) + 2 = 2k$ vertices having $k + 1$ vertices of valence $1$.

Now, if there exists a tree $T$,  with $k + 1$ vertices of valence $\leq 2$ satisfying conditions $2, 3, \text{ and}, 4$ and $e$ is an edge containing a terminal vertex (a vertex of valence $1$); then $T/\{e\}$ is homeomorphic to a tree with $k$ vertices of valence $\leq 2$. From the previous paragraph it follows that such a tree can have at most $2(k - 1)$ vertices. So, $T$ can have at most $2k$ vertices.
\end{proof}

\subsubsection{$\mathsf{Out}(G)$ action on $\mathcal{D(G, \mathscr{H})}$}

We will take the help of the following proposition to define an action of $\phi \in \mathsf{Out}(G)$ on the deformation space.

\begin{definition}
If $\Phi \in \mathsf{Aut}(G)$ is an automorphism and $T$ is a $G$-tree, then define $\Phi(T)$ to be a $G$-tree which is isometric to $T$ with a \textit{twisted} action of $G$ on $T$. The action is denoted by $\cdot_{\Phi}$ and is defined as- $$g \cdot_{\Phi} x := \Phi(g) \cdot x, \forall x \in T, g \in G.$$

here the action($\cdot$) on the right denotes the original action.
\end{definition}

\begin{proposition}\label{prop:welldefined}
If $\Phi_1, \Phi_2 \in \mathsf{Aut}(G)$ are two automorphisms representing the same outer automorphism class $\phi \in \mathsf{Out}(G)$ and $T \in \mathcal{D}(G, \mathscr{H})$ is a $G$-tree, then $\Phi_1(T)$ is $G$-equivariantly isometric to $\Phi_2(T)$.

\end{proposition}

\begin{proof}
 We will prove that if $\Phi$ is a non-identity automorphism representing the identity outer automorphism class, then $\Phi(T)$ is $G$-equivariantly isometric to $T$. 

Let, $\Phi$ represent the trivial outer automorphism then $\exists h \in G$, such that
$\Phi(g) = hgh^{-1}, \forall g \in G$. This motivates the definition of an isometry, $f$, between $T$ and $\Phi(T)$
\begin{align*}
f: T \rightarrow & \Phi(T) \\
x \mapsto & h\cdot x
\end{align*}

Next, we will verify the $G$-equivariance of the map $f$.\\
In $T$ we have, 
$f(g \cdot x) = h\cdot (g \cdot x), \forall g \in G$.

In $\Phi(T)$ we have, 
$g \cdot_{\Phi} f(x) = hgh^{-1} \cdot (h \cdot x) = (hg) \cdot x, \forall g \in G$.

Hence, $f$ is a $G$-equivariant isometry.
\end{proof}

\begin{definition}[Definition of the action] 
Given a $G$-tree, $T \in \mathcal{D}(G, \mathscr{H})$ and $\phi \in \mathsf{Out}(G)$. $\phi(T)$ is the equivalence class of G-equivariant trees represented by $\Phi(T)$, where $\Phi$ is an automorphism from the class of outer automorphism $\phi$. 
\end{definition}

\begin{remark}\label{rem:action_stab}
If $v$ is a vertex of $T$, then stab$_{\Phi(T)}(v)$ = $\Phi(\text{stab}_T(v))$, where $\Phi \in \mathsf{Aut}(G)$. If $F$ is a fundamental domain of $T$ with vertices $v_1, v_2,..., v_d$ and vertex stabilizers $G_{v_1}, G_{v_2},..., G_{v_d}$, respectively; then $F$ is a fundamental domain in $\Phi(T)$ and the vertex stabilizers of the vertices $v_1, v_2,..., v_d$ are given by $\Phi(G_{v_1}), \Phi(G_{v_2}),..., \Phi(G_{v_d})$, respectively.
\end{remark}

Following remarks \ref{rem:dictionary} and \ref{rem:action_stab}, we can give a simpler description of the action $\mathsf{Out}(F_n) \acts \mathcal{D}(G,\mathscr{H})$, when the latter is considered as a space of graph of groups.

\begin{definition}
Let $\mathbf{X} \in \mathcal{D(G, \mathscr{H})}$ be a graph of groups whose underlying graph is denoted by $X$; and $\Phi \in \mathsf{Out}(G)$ be an automorphism. Define $\Phi(\mathbf{X})$ (denoted by $\mathbf{X'}$) to be a graph of groups whose underlying graph, $X'$, is related to $X$ by an isometry $i : X \rightarrow X'$, such that if $v$ is a vertex of $\mathbf{X}$ having $G_v$ as the vertex group; then the vertex group corresponding to $i(v)$ is $\Phi(G_v)$.

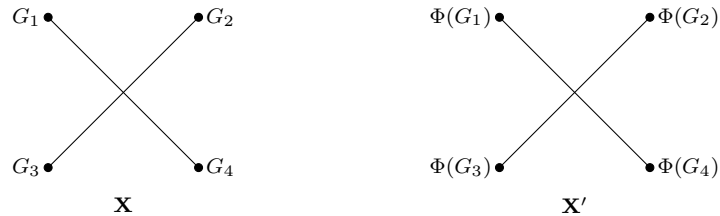
\begin{center}
\begin{tikzpicture} \begin{scriptsize}
\begin{scriptsize}

\Tw{-1}{0}{G_1}{G_2}{G_3}{G_4}{$\mathbf{X}$} 

\Tw{5}{0}{\Phi(G_1)}{\Phi(G_2)}{\Phi(G_3)}{\Phi(G_4)}{$\mathbf{X'}$}

\end{scriptsize}
\end{scriptsize} \end{tikzpicture} 
\end{center}
 
\end{definition}

\begin{proposition}
Let $\Phi \in \mathsf{Aut}(G)$ and $\mathbf{X} \in \mathcal{D}(G,\mathscr{H})$ be a graph of groups whose Bass-Serre tree is denoted by $T_{\mathbf{X}}$, then the Bass-Serre tree of $\Phi(\mathbf{X})$ is $\Phi({T_{\mathbf{X}}})$
\end{proposition}

\begin{proof}

Let the vertices and the vertex groups of $\mathbf{X}$ be labeled as $v_1, v_2, ..., v_d$, and $G_{v_1}, G_{v_2}, ..., G_{v_d}$, respectively. We can use the same vertex labeling for the vertices of $\Phi(\mathbf{X})$ and the associated vertex groups are $\Phi(G_{v_1}), \Phi(G_{v_2}),..., \Phi(G_{v_d})$, respectively.

So, there is a fundamental domain of $T_{\mathbf{X}}$ and $T_{\Phi({\mathbf{X}})}$(Bass-Serre tree of $\Phi(\mathbf{X})$) with vertex stabilizers of the vertices given by $\{G_{v_1}, G_{v_2}, ..., G_{v_d}\}$ and \\ $\{\Phi(G_{v_1}), \Phi(G_{v_2}),..., \Phi(G_{v_d})\}$, respectively. On the other hand, $\Phi(T_\mathbf{X})$ has a fundamental domain with vertex stabilizer group given by \\ $\{\Phi(G_{v_1}), \Phi(G_{v_2}),..., \Phi(G_{v_d})\}$.
 
 From the bijective correspondence between the fundamental domain and the Bass-Serre tree in $\mathcal{D}(G,\mathscr{H})$ we conclude that $\Phi(T_\mathbf{X})$ is $G$-equivariantly isometric to $T_{\Phi({\mathbf{X}})}$.
\end{proof}

\begin{corollary}
If $\Phi_1, \Phi_2 \in \mathsf{Aut}(G)$ are two automorphisms representing the same outer automorphism class $\phi \in \mathsf{Out}(G)$ and $\mathbf{X} \in \mathcal{D}(G, \mathscr{H})$ is a graph of groups, then $\Phi_1(\mathbf{X})$ is $G$-equivariantly isometric to $\Phi_2(\mathbf{X})$.
\end{corollary}

\begin{definition}[{$\mathsf{Out}(G)$ action on a graph of groups}] 
If $\phi \in \mathsf{Out}(G)$ and $\mathbf{X} \in \mathcal{D}(G,\mathscr{H})$, then $\phi \cdot \mathbf{X}:= \Phi(\mathbf{X})$, where $\Phi$ is automorphism from the outer automorphism class $\phi$.
\end{definition}


\subsection{Properties of the action}
$\mathcal{D}(G, \mathscr{H})$ is locally finite when the the elements of $\mathscr{H}$ are finite subgroups. We will later prove proper discontinuity and co-compactness (when restricted to the spine) of the action.

\begin{lemma}
$\mathcal{D}(G, \mathscr{H})$ is a locally finite topological space when $\vert H \vert < \infty, \forall H \in \mathscr{H}$.
\end{lemma}

\begin{proof}
Consider a tree $T \in \mathcal{D}(G, \mathscr{H})$. This point is part of boundary of other open simplices if we can equivariantly expand some edge-orbits of $T$. The number of edge orbits of $T$ is bounded above by $2n - 3$ and the number of vertex orbits are bounded above by $2n - 2$. Since, the vertex groups are finite, each vertex has a finite valence. Hence, the number of fundamental domains containing a vertex is bounded above. So, the number of $G$-equivariant vertex expansions is bounded above for the tree $T$.

Therefore, the relative open simplex containing $T$ can be a boundary to at most finitely many relative open simplices. As a result $\mathcal{D}(G, \mathscr{H})$ is locally finite. 
\end{proof}
 
 \begin{lemma}
Stabilizer of any point of $\mathcal{D}(G, \mathscr{H})$ under the action of \\$\mathsf{Out}(G)$ is finite.
 \end{lemma}

\begin{proof}
Consider a tree $T \in \mathcal{D}(G,\mathscr{H})$. $\phi$ is a stabilizer of the point $T$, if $\phi(T)$ is $G$-equivariantly isometric to $T$. Let us fix a fundamental domain of $T$ and name it $F$. As $\phi \in stab(T)$, $\phi(T)$ contains a fundamental domain identical to $F$ (isometric and same vertex stabilizers under the action of $G$). 

Now, let us fix a vertex $v \in F$ and choose a representative automorphism $\Phi$ from the outer class $\phi$ such that $stab(v)\vert_{\Phi(T)} = stab(v)\vert_{T}$. So, $\Phi$ permutes the fundamental domains identical to $F$ based at $v \in T$. However, there are only finitely many such fundamental domains at a given vertex and finitely many vertices $v$ of $F$. So, the vertex stabilizer subgroup is finite.
\end{proof}

A corollary of the two previous results is proper discontinuity of the action-
\begin{corollary}
The action of $\mathsf{Out}(G)$ on $\mathcal{D}(G,\mathscr{H})$ is properly discontinuous.
\end{corollary}

$\mathcal{D}(G,\mathscr{H})$ and $\mathcal{PD}(G,\mathscr{H})$ is not a simplicial complex. The spine of the $\mathcal{D}(G,\mathscr{H})$ is a simplicial complex whose $0$ skeleton is the barycenter of $\mathcal{PD}(G,\mathscr{H})$ and is denoted by  $\mathcal{SPD}(G,\mathscr{H})$. $\mathcal{PD}(G,\mathscr{H})$ deformation retracts onto $\mathcal{SPD}(G,\mathscr{H})$. The advantage of working with $\mathcal{SPD}(G,\mathscr{H})$ is that the quotient of the action $\mathsf{Out}(G) \acts \mathcal{SPD}(G,\mathscr{H})$ is compact, which is not true for the action on $\mathcal{PD}(G,\mathscr{H})$.

\begin{proposition}
The action of $\mathsf{Out}(G)$ on $\mathcal{SPD}(G,\mathscr{H})$ is co-compact.
\end{proposition}

\begin{proof}

Consider a graph of groups $\mathbf{X} \in \mathcal{D}(G,\mathscr{H})$. Each vertex group is either trivial or a conjugate of exactly one of the $A_i, i \in \{1,..., n\}$ such that the fundamental group of the graph of groups is $G$. Hence, the internal free product of the vertex groups is $G$ and we can define a $G$ automorphism $\Phi$ which maps each $A_i, i \in \{1,..., n\}$ to the vertex group of $\mathbf{X}$ conjugate to $A_i$.

If $\phi$ is the outer automorphism defined by $[\Phi]$, then $\phi^{-1} \cdot \mathbf{X}$ is a graph of groups with the set of vertex groups $\{A_1,..., A_n\}$.

So, under the action of $\mathsf{Out}(G_n)$ on $\mathcal{D}(G,\mathscr{H})$ every graph of groups is in the orbit of a graph of groups with the set of vertex groups $\{A_1,..., A_n\}$.

The underlying graph of any graph of groups from $\mathcal{SPD}(G, \mathscr{H})$ is a tree with at most $2n - 3$ edges and at least $n - 1$ edges. Hence, up-to homeomorphism there are only finitely many graphs of groups with the set of vertex groups $\{A_1,..., A_n\}$.

$\mathsf{Out}(G)$ acts by isometries on $\mathcal{SPD}(G, \mathscr{H})$, which is a simplicial complex. The quotient is a finite dimensional locally finite simplicial complex such that there are only finitely many vertices. Hence, the quotient is compact.
\end{proof}

\begin{remark}\label{rem:action}
$\mathsf{Out}(G)$ action on $\mathcal{SPD}(G,\mathscr{H})$ is properly discontinuous and co-compact. Hence, by Milnor-\u{S}varc lemma $\mathsf{Out}(G)$ is quasi isometric to $\mathcal{SPD}(G,\mathscr{H})$. We will exploit this fact to answer the original question in lower complexities and also to find a virtual generating set of $\mathsf{Out}(G)$ in general.
\end{remark}

\section{Structure of Deformation Space in Lower Complexities}\label{section:G2G3}
Recall that $\Gamma_n := \mathsf{Out}(A_1*...*A_n)$, where each $A_i$ is a finite group. $\Omega_n$ is the finite index subgroup of $\mathsf{Out}(A_1*...*A_n)$ which preserves conjugacy class of each free factor. In this section we will prove that 
 $\Omega_2$ is finite and
 , $\Omega_3$ is a hyperbolic group (virtually free). We will also inspect a finite index subgroup of $\Omega_4$ and denote it by $\Gamma_4'$. This will lay the ground work for a similar inspection for $\Omega_n, (n \geq 5)$.

\subsection{Finiteness of $\Gamma_2$}\label{subsec:G2}

\begin{lemma}\label{lemma:unique_min}
$\mathcal{SPD}(G_2,\mathscr{H})$ is a point.
\end{lemma}

\begin{proof}
Consider the graph of groups:
\begin{center}
\begin{tikzpicture} \begin{scriptsize}
\draw (1,0) -- (-1, 0);
\draw [fill=black] (1, 0) circle (1.5pt) node [anchor = west]{$A_2$};
\draw [fill=black] (-1, 0) circle (1.5pt) node [anchor = east]{$A_1$};
\end{scriptsize} \end{tikzpicture} 
\end{center}

The Bass-Serre tree of this graph of groups is a $G$-tree whose vertex stabilizers are conjugates of $A_i, i \in \{1, 2\}$. There is only one edge orbit. If we collapse an edge equivariantly in this tree, we will get a point. So, no $G$-equivariant collapses are possible. Contractibility of the deformation space due to theorem \ref{thm:contractible} implies that if there is a different tree non $G$-equivariantly homeomorphic to the given tree, then they can be connected in the deformation space by a collapse expand path. However, a collapse or expand move is not permissible due to constraint on the vertex stabilizers. So, we arrive at a contradiction.
\end{proof}

\begin{corollary}\label{cor:G2_fin}
$\Omega_2$ and $\Gamma_2$ is finite.
\end{corollary}

\begin{remark}\label{rem:unique_min}
Consider finite subgroups $H$ and $K$ of the group $G := A_1*...*A_n$ which are factors in a collection of system of subgroups whose internal free product is $G$. Lemma \ref{lemma:unique_min} implies that $(H*K)$-minimal subtrees are $G$-equivariantly homeomorphic for any two $G$-trees, $T_1, T_2 \in \mathcal{D}(G,\mathscr{H})$. 
\end{remark}

\subsection{Hyperbolicity of $\Gamma_3$}\label{subsec:G3}

\begin{proposition}
$\mathcal{SPD}(G_3,\mathscr{H})$ is a $1$ dimensional simplicial complex.
\end{proposition}

\begin{proof}
If $T \in \mathcal{SPD}(G_3,\mathscr{H})$ is a $G$-tree then the number of edge orbits of $T$ is at most $3$ and at least $2$. 
Hence, we can apply only $1$-edge orbit collapse move on $T$. So, $\mathcal{SPD}(G_3,\mathscr{H})$ does not have any $2$ dimensional simplex and is a $1$ dimensional simplicial complex.
\end{proof}

\begin{corollary}\label{cor:G3_hyp}
$\Omega_3$ and $\Gamma_3$ hyperbolic groups.
\end{corollary}

\begin{proof}
By Guirardel-Levitt's work (theorem \ref{thm:contractible}) $\mathcal{SPD}(G_3,\mathscr{H})$ is contractible. Also, $\mathcal{SPD}(G_3,\mathscr{H})$ is a $1$ dimensional simplicial complex. So, $\mathcal{SPD}(G,\mathscr{H})$ is a tree. 

$\Omega_3$ and $\Gamma_3$ act geometrically on $\mathcal{SPD}(G_3,\mathscr{H})$. Using lemma \ref{lemma:mil_sva} we can say that $\Omega_3$ and $\Gamma_3$ are hyperbolic.
\end{proof}

\section{A finite index subgroup of $\Omega_n$}\label{section:fin_index}

In this section we will investigate a subgroup generated by some elements of $\Omega_n$ and prove that the subgroup is finite index. In our subsequent discussions we have referred to a unique graph of groups of $\mathcal{SPD}$, frequently. The element is denoted by $\mathbf{X}$. We have also used another class of graph of groups $\mathbf{Y_i}$ for the proofs in section \ref{subsec:fin_index}.
\begin{notation}\label{not:X}

\begin{enumerate}[wide, labelwidth=!, labelindent=0pt]
\item Let, $\mathbf{X} \in \mathcal{SPD}(G,\mathscr{H})$ be the vertex of $\mathcal{SPD}$ given by the following graph of groups.

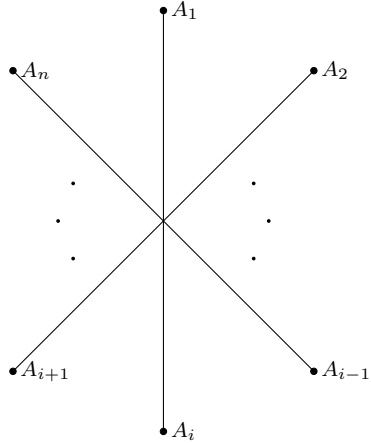
\begin{center}

\begin{tikzpicture} \begin{scriptsize}
\draw (2, 2)  -- (-2,-2);
\draw (-2, 2)  -- (2,-2);
\draw (0, 2.8) -- (0, -2.8);

\draw [fill=black] (2,-2) circle (1.2pt) node[anchor=west]{$A_{i-1}$};
\draw [fill=black] (0,2.8) circle (1.2pt) node[anchor=west]{$A_{1}$};
\draw [fill=black] (2,2) circle (1.2pt) node[anchor=west]{$A_{2}$};
\draw [fill=black] (0,-2.8) circle (1.2pt) node[anchor=west]{$A_{i}$};
\draw [fill=black] (-2,-2) circle (1.2pt) node[anchor=west]{$A_{i + 1}$};
\draw [fill=black] (-2,2) circle (1.2pt) node[anchor=west]{$A_{n}$};
\draw [fill=black] (1.4, 0) circle (.5pt);
\draw [fill=black] (1.2, .5) circle (.5pt);
\draw [fill=black] (1.2, -.5) circle (.5pt);
\draw [fill=black] (-1.4, 0) circle (.5pt);
\draw [fill=black] (-1.2, -.5) circle (.5pt);
\draw [fill=black] (-1.2, .5) circle (.5pt);
\end{scriptsize} \end{tikzpicture} 
\end{center}

\item Any graph of groups whose underlying graph is isomorphic (simplicially) to the underlying graph of $\mathbf{X}$ will be called a graph of groups of type $X$. Similarly, any tree $G$-equivariantly homeomorphic to the the Bass-Serre tree of a type $X$ graph of groups will be called a tree of type $X$.

\item Let, $\mathbf{Y_i} \in \mathcal{SPD}(G,\mathscr{H})$ be the vertex of $\mathcal{SPD}$ given by the following graph of groups. The subscript $i$ signifies that the vertex associated to the vertex group $A_i$ has valence $n - 1$ and the rest of the vertices have valence $1$.

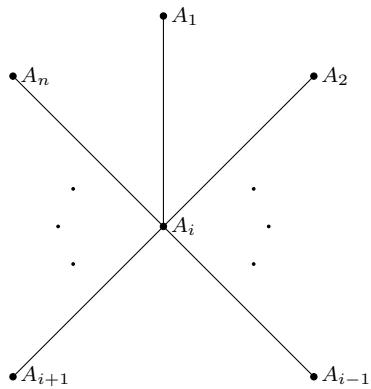
\begin{center}
\begin{tikzpicture} \begin{scriptsize}
\draw (2, 2)  -- (-2,-2);
\draw (-2, 2)  -- (2,-2);
\draw (0, 2.8) -- (0, 0);

\draw [fill=black] (2,-2) circle (1.2pt) node[anchor=west]{$A_{i-1}$};
\draw [fill=black] (0,2.8) circle (1.2pt) node[anchor=west]{$A_{1}$};
\draw [fill=black] (2,2) circle (1.2pt) node[anchor=west]{$A_{2}$};
\draw [fill=black] (0, 0) circle (1.2pt) node[anchor=west]{$A_{i}$};
\draw [fill=black] (-2,-2) circle (1.2pt) node[anchor=west]{$A_{i + 1}$};
\draw [fill=black] (-2,2) circle (1.2pt) node[anchor=west]{$A_{n}$};
\draw [fill=black] (1.4, 0) circle (.5pt);
\draw [fill=black] (1.2, .5) circle (.5pt);
\draw [fill=black] (1.2, -.5) circle (.5pt);
\draw [fill=black] (-1.4, 0) circle (.5pt);
\draw [fill=black] (-1.2, -.5) circle (.5pt);
\draw [fill=black] (-1.2, .5) circle (.5pt);
\end{scriptsize} \end{tikzpicture}
\end{center}

\item Any graph of groups whose underlying graph is isomorphic to the underlying graph of $\mathbf{Y_i}$ will be called a graph of groups of type $Y$. In other words, a graph of groups with $1$ vertex of valence $n - 1$ and $n - 1$ vertices of valence $1$ is a graph of groups of type ${Y}$. Similarly, any tree $G$-equivariantly homeomorphic to the the Bass-Serre tree of a type $Y$ graph of groups will be called a tree of type $Y$.

\end{enumerate}
\end{notation}

\subsection{$\mathcal{D}(G,\mathscr{H})$ and a finite index subgroup of $\Omega_n$}\label{subsec:fin_index}

\subsubsection{Deformation Space of $G_n$-trees}
\begin{definition}\label{defn:FIJ} 
Given $w \in \displaystyle\bigsqcup_{j = 1}^n A_j$ and a fixed integer, $i \in \{1,..., n\}$, define a map $F_{A_i}^w: \displaystyle \bigsqcup_{j = 1}^{n} A_j \rightarrow \displaystyle \free_{j = 1}^n A_j$ as follows:
$$F_{A_i}^w(a) = \begin{cases}
waw^{-1} & , \text{if }a \in A_i \\
a & , \text{otherwise}\\
\end{cases}$$
By the universal property, this map can uniquely be extended to an automorphism $F_{A_i}^w: 
\displaystyle \free_{j = 1}^n A_j \rightarrow \free_{j = 1}^n A_j$. 
In general, for $w \in \displaystyle \free_{j = 1}^n A_i$ we define $F_{A_i}^w$ 
inductively as follows. If $w = uv$, 
then define $F_{A_i}^w := F_{A_i}^vF_{A_i}^u$.
\end{definition}

\begin{definition}\label{defn:F^w_H}
Define $f_{H}^w$ to be the outer automorphism defined by the automorphism $F_{H}^w$, 
where $H \in \{A_1,..., A_n\}$ and $w \in \displaystyle \free_{j = 1}^n A_j$.
\end{definition}

\begin{lemma}\label{lemma:diff_base_commute}
If $k, m \in \{1, ..., n\}$ are distinct integers, then for any $u, v \in \displaystyle \free_{\substack{i \ne k, m\\i = 1}}^n A_i$, $f_{A_k}^u$ commutes with $f_{A_m}^v$.
\end{lemma}

\begin{proof}
Definition \ref{defn:FIJ} implies $F_{A_k}^uF_{A_m}^v = F_{A_m}^vF_{A_k}^u$, when $m$ and $k$ are distinct integers and $u, v \in \displaystyle \free_{\substack{i \ne k, m\\i = 1}}^n A_i$. So,  $f_{A_k}^uf_{A_m}^v = f_{A_m}^vf_{A_k}^u$, when $m$ and $k$ are distinct integers and $u, v \in \displaystyle \free_{\substack{i \ne k, m\\i = 1}}^n A_i$.
\end{proof}

\begin{definition}\label{defn:hij}
For a fixed $i \in \{1,..., n\}$,
define the following subgroups
\begin{align*}
 \overline{H_i^j} := &
 \{ F_{A_i}^w \vert w \in A_j\} <
 \mathsf{Aut}(\displaystyle \free_{j = 1}^n A_j) \\
 H_i^j := &
\{f_{A_i}^w \vert w \in A_j\} <
\mathsf{Out}(\displaystyle \free_{j = 1}^n A_j)
\end{align*}
\end{definition}

\begin{proposition}\label{prop:hij_prod}
For a fixed $i \in \{1, ..., n\}$, we have the following isomorphisms
\begin{align*}
\displaystyle \left\langle \overline{H_i^j} \vert j \in \{1,..., n\} - \{i\}\right\rangle 
= & \free_{\substack{j \ne i\\j =1}} \overline{H_i^j} \\
 \left\langle {H_i^j} \vert j \in \{1,..., n\} - \{i\}\right\rangle
 =  & \free_{\substack{j \ne i\\j =1}} {H_i^j} \\
 \free_{\substack{j \ne i\\j =1}} \overline{H_i^j} \cong \free_{\substack{j \ne i\\j =1}} {H_i^j}  \cong & \free_{\substack{j \ne i\\j =1}}^n A_j
\end{align*}
\end{proposition}

\begin{proof}
Let $\Phi \in \left\langle \overline{H_i^j} \vert j \in \{1,..., n\} - \{i\}\right\rangle$ be an element such that it can be expressed as a composition of $F_{A_i}^w$s. That is, $\Phi = F_{A_i}^{u_1} \circ... \circ F_{A_i}^{u_k}$, where each $u_l \in \displaystyle\bigsqcup_{\substack{j \ne i\\j= 1}}^n A_j$. Then 
\begin{gather*}\Phi(A_t) = \begin{cases} A_t & \text{ if } t \ne i\\ 
(u_k...u_1)A_t(u_k...u_1)^{-1} & \text{ if } t = i\end{cases}\end{gather*}
Hence, 
\begin{align*}
\Phi = 
id \hspace{.3in} (\text{in }{ \left\langle \overline{H_i^j} \vert j \in \{1,..., n\} - \{i\}\right\rangle})
\iff &
u_k...u_1 = id\hspace{.3in}(\text{in }{\displaystyle\free_{\substack{j \ne i\\j = 1}}^n \overline{A_j}})\\
\iff &
F_{A_i}^{u_1}... F_{A_i}^{u_k}
=
id \hspace{.1in}(\text{in }{\displaystyle\free_{\substack{j \ne i\\j = 1}}^n \overline{H_i^j}})
\end{align*}
So, the following maps are well defined isomorphisms
\begin{align*}
\left\langle \overline{H_i^j} \vert j \in \{1,..., n\} - \{i\}\right\rangle
& \rightarrow &
\displaystyle\free_{\substack{j \ne i\\j =1}}^n \overline{H_i^j}
& \rightarrow &
\displaystyle \free_{\substack{j \ne i\\j =1}}^n A_j  \\
\Phi
& \mapsto &
F_{A_i}^{u_1}... F_{A_i}^{u_k}
& \mapsto &
u_k...u_1
\end{align*}

Similarly, let $\phi \in \left\langle {H_i^j} \vert j \in \{1,..., n\} - \{i\}\right\rangle$ be an element such that it can be expressed as a product of $f_{A_i}^w$s. That is, $\phi = f_{A_i}^{v_1}... f_{A_i}^{v_r}$, where each $v_l \in \displaystyle\bigsqcup_{j=1, j \ne i}^n A_j$. Consider the graph of groups $\mathbf{X} \in \mathcal{SPD}(G, \mathscr{H})$ (notation \ref{not:X}). Then the underlying graph of $\phi(\mathbf{X})$ is isomorphic to the underlying graph of $\mathbf{X}$ and the corresponding vertex groups are \\ $\{A_1,..., A_{i-1}, (v_r...v_1)A_i(v_r...v_1)^{-1}, A_{i+1},...,A_n\}$. If $\phi(\mathbf{X})$ is $G$-equivariantly isometric to $\mathbf{X}$ (denoted by $\phi(\mathbf{X}) \cong_G \mathbf{X}$), then $\Phi(A_i) = A_i, \forall i$.
\begin{align*}
\phi = 
id ( \text{in }{ \left\langle {H_i^j} \vert j \in \{1,..., n\} - \{i\}\right\rangle})
\iff &
\phi(\mathbf{X}) \cong_G \mathbf{X} \\
\iff & 
\phi(T_{\mathbf{X}}) \cong_G T_{\mathbf{X}} \\
\iff &
\text{the vertices labeled by }\\
& A_i, wA_iw^{-1} (w \in A_j, j \ne i) \text{ are } \\
&
\text{adjacent to } A_j (j \in \{1,..., n\}-\{i\}) \\
& \text{ in }\phi(T_{\mathbf{X}}) \\
\iff &
v_r...v_1 \in \displaystyle \bigcap_{\substack{j \ne i\\j = 1}}^n A_j \\
\iff &
v_r...v_1 = id \hspace{1in} (\text{in }{\displaystyle\free_{\substack{j \ne i\\j = 1}}^n A_j}) \\
\iff &
f_{A_i}^{v_1}... f_{A_i}^{v_r}
=
id  \hspace{.85in} (\text{in }{\displaystyle\free_{\substack{j \ne i\\j = 1}}^n {H_i^j}})
\end{align*}
So, the following maps are well defined isomorphisms
\begin{align*}
\left\langle {H_i^j} \vert j \in \{1,..., n\} - \{i\}\right\rangle
\rightarrow &
\displaystyle\free_{\substack{j \ne i\\j = 1}}^n {H_i^j}
\rightarrow &
\displaystyle \free_{\substack{j \ne i\\j = 1}}^n A_j  \\
\phi
\mapsto &
f_{A_i}^{v_1}... f_{A_i}^{v_r}
\mapsto &
v_r...v_1
\end{align*}
\end{proof}

\begin{proposition}\label{prop:hij_sum}
Consider two distinct, fixed integers $j_1, j_2 \in \{1, ..., n\}$, then\\
$\langle H_i^{j_1}, H_i^{j_2} \vert i \in \{1, ..., n\} - \{j_1, j_2\} \rangle = 
\displaystyle \bigoplus_{\substack{i \ne j_1, j_2\\ i = 1}}^n H_i^{j_1} \free H_i^{j_2} \cong
\bigoplus_{\substack{i \ne j_1, j_2\\ i = 1}}^n A_{j_1} \free A_{j_2}$
\end{proposition}

\begin{proof}
Let $k,  l \in \{1,..., n\} - \{j_1, j_2\}$ be distinct integers, then 
$\langle \overline{H_k^{j_1}}, \overline{H_k^{j_2}} \rangle$ commutes with 
$\langle \overline{H_l^{j_1}}, \overline{H_l^{j_2}} \rangle$. Hence, 
$\langle {H_k^{j_1}}, {H_k^{j_2}} \rangle$ commutes with 
$\langle {H_l^{j_1}}, {H_l^{j_2}} \rangle$.

Consider $\phi \in 
\langle H_i^{j_1}, H_i^{j_2} \vert i \in \{1, ..., n\} - \{j_1, j_2\} \rangle$.
Due to the commutativity stated above $\phi$ can be expressed as a product, 
$\phi = f_{A_{i_1}}^{w_{i_1}}...f_{A_{i_s}}^{w_{i_s}}$, 
where $i_1,..., i_s \in \{1,..., n\} - \{j_1, j_2\}$ 
are distinct integers and 
$w_{i_1},..., w_{i_s} \in A_{j_1} \free A_{j_2}$. We want to show that $\phi$ 
is the identity outer automorphism if and only if $w_1,..., w_s$ are all identity elements. 
To prove this we will consider the action of $\phi$ on the Bass-Serre tree ($T_{\mathbf{X}}$) of the graph of groups $\mathbf{X} \in \mathcal{SPD}(G,\mathscr{H})$ (notation \ref{not:X}). The underlying graph of $\mathbf{X}$ has $n$ vertices of valence $1$ and $1$ vertex of valence $n$. The non-trivial vertex groups of $\mathbf{X}$ are $\{A_1,...,  A_n\}$ (the groups assigned to the valence $1$ vertices). Then the underlying graph of $\phi(\mathbf{X})$ is isomorphic to the underlying graph of $\mathbf{X}$ and the corresponding vertex groups are \\
$\{w_1A_1w_1^{-1},
..., 
w_{j_1-1}A_{j_1-1}w_{j_1-1}^{-1},
A_{j_1},
w_{j_1+1}A_{j_1+1}w_{j_1+1}^{-1},
..., \\
w_{j_2-1}A_{j_2-1}w_{j_2-1}^{-1},
A_{j_2},
w_{j_2+1}A_{j_2+1}w_{j_2+1}^{-1},
..., 
w_nA_nw_n^{-1}
\}$. Without loss of generality, we have assumed $1 \ne j_1 < j_2 \ne n$.
\begin{align*}
\phi = id 
\iff & 
T_\mathbf{X} \text{ is $G$-equivariantly isometric to } \phi({T_\mathbf{X}}) \\
\iff &
w_i \in A_{j_1} \cap A_{j_2}, \forall i \in \{1,..., n\}-\{j_1, j_2\} \\
\iff & 
w_i = id,  \forall i \in \{1,..., n\}-\{j_1, j_2\} \\
\iff &
f_{A_i}^{w_i} = id, \forall i \in \{1,..., n\}-\{j_1, j_2\} \\
\iff & 
\langle H_k^{j_1}, H_k^{j_2}\rangle \cap \langle H_l^{j_1}, H_l^{j_2}\rangle = \{\}, \forall k \ne l \in \{1,..., n\}-\{j_1, j_2\}
\end{align*}
If we combine this with the commutativity of the subgroups 
$\langle {H_k^{j_1}}, {H_k^{j_2}} \rangle$ and 
$\langle {H_l^{j_1}}, {H_l^{j_2}} \rangle$ for distinct $k, l \in \{1,..., n\}-\{j_1, j_2\}$, 
then we get decomposition into direct products as follows - 
\begin{align*}
 \langle H_i^{j_1}, H_i^{j_2} \vert i \in \{1, ..., n\} - \{j_1, j_2\} \rangle 
 = &
\displaystyle \bigoplus_{i \ne j_1, j_2, i = 1}^n \langle H_i^{j_1}, H_i^{j_2} \rangle \\
(\text{proposition } \ref{prop:hij_prod} \implies)
= &
\displaystyle \bigoplus_{i \ne j_1, j_2, i = 1}^n H_i^{j_1} \free H_i^{j_2} \\
(\text{proposition } \ref{prop:hij_prod} \implies)
\cong & 
\bigoplus_{i \ne j_1, j_2, i = 1}^n A_{j_1} \free A_{j_2}
\end{align*}
\end{proof}
\begin{corollary}
$\mathsf{Out}(\displaystyle\free_{i = 1}^n A_i)$ is not hyperbolic, when $n \geq 4$.
\end{corollary}
\begin{proof}
When $n \geq 4$, the cardinality of the set $\{1,..., n\} - \{j_1, j_2\}$ is greater than $1.$ Hence, $\displaystyle \bigoplus_{\substack{i \ne j_1, j_2\\ i = 1}}^n H_i^{j_1} \free H_i^{j_2}$ is a direct sum of more than one infinite groups, which violates the hyperbolicity of $\displaystyle \bigoplus_{\substack{i \ne j_1, j_2\\ i = 1}}^n H_i^{j_1} \free H_i^{j_2}$. As a result, $\mathsf{Out}(\displaystyle\free_{i = 1}^n A_i)$ is not hyperbolic.
\end{proof}

\subsubsection{$\Gamma_n'$ - a finite index subgroup of $\Omega_n$}\label{subsec:fin_index}

\begin{definition}\label{defn:GN'}

Consider the subgroup $\Gamma_n' \leq$ $ \Omega_n$, generated by outer automorphisms of the form $f_H^w$ (definition \ref{defn:F^w_H}), where $w \in  \displaystyle\bigcup_{i = 1}^n A_i - H$.
\end{definition}
\begin{remark} From definitions \ref{defn:hij}, \ref{defn:GN'} we get, 
$\Gamma_n' = \la H_i^j \vert i, j \in \{1,..., n\}, i \ne j \ra$. We will prove that $\Gamma_n'$ is a finite index subgroup of $\Omega_n$. We will refer to graph of groups $\mathbf{X}, \mathbf{Y_i}$, graph of groups (and $G$-trees) of type $X$ and type $Y$ from notation \ref{not:X} for our discussion in this section.
\end{remark}


\begin{remark}
Our strategy for the proof of finite index of $\Gamma_n'$ in $\Omega_n$ is described below.

\begin{enumerate}[wide, labelwidth=!, labelindent=0pt]
\item In lemma \ref{lemma:conjugacy_relation} we will give a relation between the vertex stabilizers of two vertices of a $G$-tree which are in the same $G$-orbit and are part of two fundamental domains with non trivial intersection(s). 

\item Corollary \ref{cor:1_edge_exp} will follow from lemma \ref{lemma:conjugacy_relation}. In corollary \ref{cor:1_edge_exp} we will establish a relation between the vertex stabilizer subgroups in a selected fundamental domain of two different trees when the trees differ by $1$-edge orbit expansion. 

\item In lemma \ref{lemma:1_edge_exp_X} we will see how to construct a path between two trees of type $X$ using trees of type $X$ and $Y$, when there is a fundamental domain in the respective trees that share a relation similar to that as in corollary \ref{cor:1_edge_exp}. 

\item In lemma \ref{lemma:same_fund} we will prove that any two trees that have the same non trivial vertex stabilizer subgroups in a fundamental domain can be connected. 

\item In lemma \ref{lemma:edge_exp} we will connect any two trees of type $X$ by trees of type $X$ and $Y$ under certain restriction. 

\item Corollary \ref{cor:xy_connected} will follow from lemma \ref{lemma:edge_exp}, where we will prove that the sub-complex of $\mathcal{SPD}$ spanned by trees of type $X$ and $Y$ is connected in $\mathcal{SPD}$.

\item In lemma \ref{lemma:same_orbit_x} we will prove that two trees of type $X$ which are distance $2$ apart are in the same $\Gamma_n'$ orbit.

\end{enumerate}
\end{remark}

\begin{lemma}\label{lemma:conjugacy_relation}
Consider $T \in \mathcal{SPD}^{0}(G,\mathscr{H})$ and two fundamental domains $F_1, F_2$ of $T$ such that the vertex stabilizer groups of $F_1$ are given by $W_1, ..., W_{n }$ and the vertex stabilizer groups of $F_2$ are given by $V_1, ..., V_{n}$ where $V_k$ is conjugate to $W_k, \forall k \in \{1,..., n\}$. Assume that $V_i = W_i$ for a fixed $i$ and for $j \ne i$ the vertices with non trivial vertex stabilizers in the shortest path between $W_i$, and $W_j$ (excluding $W_i$ and $W_j$) are labeled as $W_{j_1},...,W_{j_{n_j}}$ in increasing order of distance from $W_i$, then the conjugacy relations are given by - 
\begin{equation}\label{eqn:1_edge_exp}
V_{j_p} = (w_iw_{j_1}...w_{j_{p - 1}})W_{j_p}(w_iw_{j_1}...w_{j_{p - 1}})^{-1}
\end{equation}
where, $w_r \in W_r$, for $r \in \{i, j_1,..., j_p,.., j_{n_j}\}$
\end{lemma}

\begin{proof}
The choice of vertices in the respective conjugacy classes of subgroups for the fundamental domain based at the vertex $W_i$ is outlined below:

\begin{enumerate}[wide, labelwidth=!, labelindent=0pt]
\item[$W_{j_1}$] In the $W_i*W_{j_1}$ minimal subtree of $T_1$ we choose the vertex labeled by $w_iW_{j_1}w_i^{-1}$ (for the conjugacy class of $W_{j_1}$ in the fundamental domain). 

\item[$W_{j_2}$] In the $w_iW_{j_1}w_i^{-1} * w_iW_{j_2}w_i^{-1}$ minimal subtree of $T_1$ we choose the vertex labeled by $w_iw_{j_1}W_{j_2}w_{j_1}^{-1}w_i^{-1}$ (for the conjugacy class of $W_{j_2}$ in the fundamental domain). 

\vdots

\item[$W_{j_p}$] In the $(w_iw_{j_1}...w_{j_{p - 2}})W_{j_{p - 1}}(w_iw_{j_1}...w_{j_{p - 2}})^{-1} * (w_iw_{j_1}...w_{j_{p - 2}})W_{j_p}(w_iw_{j_1}...w_{j_{p - 2}})^{-1}$ minimal subtree of $T_1$ we choose the vertex labeled by\\ $(w_iw_{j_1}...w_{j_{p - 1}})W_{j_p}(w_iw_{j_1}...w_{j_{p - 1}})^{-1}$ (in the conjugacy class of $W_{j_p}$ for the fundamental domain).

\end{enumerate}
\end{proof}

\begin{corollary}\label{cor:1_edge_exp}
Consider $T_1 \in \mathcal{SPD}^{0}(G,\mathscr{H})$ and fix a fundamental domain of $T_1$ whose nontrivial vertex stabilizers are given by $W_1,..., W_n$. Let $T_2 \in \mathcal{SPD}^{0}(G,\mathscr{H})$ be obtained from $T_1$ by equivariantly expanding the vertex labeled by $W_i$. Then there exists a fundamental domain of $T_2$ whose non trivial vertex stabilizers are labeled by $V_1,..., V_n$, where each $W_k$ is conjugate to $V_k$ for $k \in \{1,..., n\}$, and the conjugacy relations are given by-
\begin{enumerate}[wide, labelwidth=!, labelindent=0pt]
\item $V_i = W_i$ for some $i \in \{1,...,n\}$

\item If $j \ne i$ and the vertices with non trivial vertex stabilizers in the shortest path between $W_i$, and $W_j$ (excluding $W_i$ and $W_j$) are labeled as $W_{j_1},...,W_{j_{n_j}}$ in increasing order of distance from $W_i$, then
\begin{equation}\label{eqn:1_edge_exp}
V_{j_p} = (w_iw_{j_1}...w_{j_{p - 1}})W_{j_p}(w_iw_{j_1}...w_{j_{p - 1}})^{-1}
\end{equation}
where, $w_r \in W_r$, for $r \in \{i, j_1,..., j_p,.., j_{n_j}\}$
\end{enumerate}   
\end{corollary}

\begin{proof}
$V_i = W_i$ implies that $T_2$ has two fundamental domains which satisfies the conditions of lemma $\ref{lemma:conjugacy_relation}$. Hence, we have the result.
\end{proof}

\begin{lemma}\label{lemma:1_edge_exp_X}

\begin{enumerate}[wide, labelwidth=!, labelindent=0pt] Consider two trees $T_1$ and $T_2$ with the following properties
\item $T_1 \in \mathcal{SPD}^{0}(G,\mathscr{H})$ is a tree of type $X$ with a fundamental domain whose non trivial vertex stabilizers are labeled by $W_1,..., W_n$. 
\item$T_2 \in \mathcal{SPD}^{0}(G,\mathscr{H})$ is also a tree of type $X$ with a fundamental domain whose non trivial vertex stabilizers are labeled by $V_1,..., V_n$ 
\item For each $k \in \{1,..., n\}, W_k, \text{is related to } V_k$ by the equation \ref{eqn:1_edge_exp} given in the corollary \ref{cor:1_edge_exp}. That is
\begin{enumerate}[wide, labelwidth=!, labelindent=0pt]
\item $V_i = W_i$ for a fixed $i \in \{1,...,n\}$

\item If $j \ne i$, then 
$V_{j_p} = (w_iw_{j_1}...w_{j_{p - 1}})W_{j_p}(w_iw_{j_1}...w_{j_{p - 1}})^{-1}$
where, $w_r \in W_r$ for $r \in \{i, j_1,..., j_p,.., j_{n_j}\}$
\end{enumerate}
\end{enumerate} 
Then we can connect $T_1$ and $T_2$ by a path in $\mathcal{SPD}^{1}(G,\mathscr{H})$ using trees of type $X$ and ${Y}$.

\end{lemma}

\begin{proof}
We will import the notations from corollary \ref{cor:1_edge_exp} and show that if 
\begin{itemize}[wide, labelwidth=!, labelindent=0pt]
\item $V_i = W_i$
\item $V_{j_1} = w_iW_{j_{1}}w_i^{-1}$
\item $V_{j_2} = w_iw_{j_1}W_{j_2}w_i^{-1}w_{j_1}^{-1}$

\vdots
\item $V_{j_p} = (w_iw_{j_1}...w_{j_{p - 1}})W_{j_p}(w_iw_{j_1}...w_{j_{p - 1}})^{-1}$

\end{itemize}
then to get to $T_2$ from $T_1$ we have to apply some number of carefully chosen collapse-expand moves. Our choice will be motivated by the previous lemma so that all the intermediate trees are only of type $X$ and $Y$.

The steps of the expand and collapse moves are underlined below:
\begin{itemize}[wide, labelwidth=!, labelindent=0pt]

\item On $T_1$ apply the following moves:
\begin{enumerate}[wide, labelwidth=!, labelindent=0pt]
\item Collapse the edges of $T_1$ adjacent to vertex labeled by $W_i$, equivariantly. The resulting tree is of type $Y$.
\item Choose a fundamental domain replacing each $W_r$ by $w_iW_rw_i^{-1}$, $\forall$ $r \in \{j_1,..., j_p\}$.
\item Expand (equivariantly) the vertex labeled by $W_i$. This tree is of type $X$.
\end{enumerate}

\item On the resulting tree we apply the following moves:
\begin{enumerate}[wide, labelwidth=!, labelindent=0pt]
\item Collapse the edges adjacent to vertex labeled by $w_iW_{j_1}u^{-1}$, equivariantly. The resulting tree is of type $Y$.
\item Choose a fundamental domain replacing $w_iW_rw_i^{-1}$ by $w_iw_{j_1}W_rw_{j_1}^{-1}w_i^{-1}$, 
\\ $\forall$ $r \in \{j_2,..., j_p\}$.
\item Expand (equivariantly) the vertex labeled by $w_iW_{j_1}{w_i}^{-1}$. This tree is of type $X$.
\end{enumerate}

\vdots

\item This is the final step:
\begin{enumerate}[wide, labelwidth=!, labelindent=0pt]
\item Collapse the edges adjacent to vertex labeled by \\$(w_iw_{j_1}...w_{j_{p - 2}})W_{j_{p - 1}}(w_iw_{j_1}...w_{j_{p - 2}})^{-1}$, equivariantly. The resulting tree is of type $Y$.
\item Choose a fundamental domain replacing \\$(w_iw_{j_1}...w_{j_{p - 2}})W_{j_{r}}(w_iw_{j_1}...w_{j_{p - 2}})^{-1}$ by $(w_iw_{j_1}...w_{j_{p - 1}})W_{j_{r}}(w_iw_{j_1}...w_{j_{p - 1}})^{-1}$,\\
for $r =j_p$.
\item Expand (equivariantly) the vertex labeled by \\$(w_iw_{j_1}...w_{j_{p - 2}})W_{j_{p - 1}}(w_iw_{j_1}...w_{j_{p - 2}})^{-1}$. This tree is of type $X$.
\end{enumerate}
\end{itemize}
\end{proof}

\begin{lemma}\label{lemma:same_fund}
Consider $T \in \mathcal{SPD}(G,\mathscr{H})$ with a fundamental domain having non-trivial vertex stabilizer subgroups labeled by $W_1, ..., W_n$. If $T' \in \mathcal{SPD}(G,\mathscr{H})$ has a fundamental domain with the non trivial vertex stabilizer subgroups labeled by $W_1,..., W_n$, then $T$ and $T'$ are connected by a expand-collapse path in $\mathcal{SPD}(G,\mathscr{H})$ such that every intermediate tree in that path has a fundamental domain with the non trivial vertex stabilizer subgroups labeled by $W_1,..., W_n$.
\end{lemma}

\begin{proof}
We will show the existence of such a path in a few steps. 

\begin{enumerate}[wide, labelwidth=!, labelindent=0pt]
\item $T$ is connected to a tree with maximum number of edge orbits having a fundamental domain such that the non trivial vertex stabilizers are labeled by $W_1,..., W_n$.

\item Any two trees with maximum number of edge orbits having a fundamental domain with non trivial vertex stabilizer subgroups labels $W_1,..., W_n$ are connected. This is because both of them are connected to the tree of type $X$ with a fundamental domain labeled by $W_1,..., W_n$ via collapse moves.
\end{enumerate}

So, $T$ and $T'$ are connected to the same tree of type $X$.
\end{proof}

\begin{lemma}\label{lemma:edge_exp}
Consider $T, T' \in \mathcal{SPD}^{0}(G,\mathscr{H})$ such that the distance between them is $1$ in $\mathcal{SPD}$. If $S, S' \in \mathcal{SPD}^{0}(G,\mathscr{H})$ are trees of type $X$ whose nontrivial vertex stabilizer subgroups in a fundamental domain are same as that of in a fundamental domain of $T, T'$, respectively. Then $S$ and $S'$ can be connected by a path in $\mathcal{SPD}$ consisting only of trees of type $X$ and $Y$.
\end{lemma}

\begin{proof}
Since $T$ and $T'$ are at a distance of $1$. So, without loss of generality let us assume $T'$ is
obtained by expanding $p$ edge orbits of $T$, equivariantly. 

We can find trees $T = T_0, T_1,..., T_{p - 1}, T_p = T'$ such that $T_{i + 1}$ is obtained from $T_i$ by one edge orbit expansion. That is, we find trees so that the $p$-edge orbit expansions are factored into $p$ $1$-edge orbit expansions.

For each $T_i$ let $S_i$ denote the tree of type $X$ with a fundamental domain whose non trivial vertex stabilizer subgroups are same as that of a fundamental domain of $T_i$.

From lemma \ref{lemma:1_edge_exp_X} of this subsection we know that $S_i$ and $S_{i + 1}$ can be connected by an expand-collapse path consisting only of trees of type $X$ and $Y$.
\end{proof}

\begin{corollary}\label{cor:xy_connected}
The sub-complex of the $1$-skeleton of $\mathcal{SPD}(G,\mathscr{H})$ spanned by vertices corresponding to graph of groups of type $X$ and $Y$ is connected.

\end{corollary}

\begin{proof}
If $S, S' \in \mathcal{SPD}^0(G,\mathscr{H})$ are trees of type $X$. Consider a path of length $q$ connecting them. Starting from $S$ let the trees in this path be given by $T_0 = S, T_1, ..., T_{q - 1}, T_q = S'$.

For a given $T_i$, let $S_i$ represent the tree of type $X$ having a fundamental domain with non trivial vertex stabilizer subgroups identical to that of a fundamental domain of $T_i$. 

Following the previous lemma, lemma \ref{lemma:edge_exp}, we see that $S_i$ and $S_{i + 1}$ can be connected by a path containing only of trees of type $X$ and $Y$. So the alternative path would consist of trees $S_0 = S, S_1,..., S_{q-1}, S_q = S'$ and all the trees between each $S_i$ and $S_{i + 1}$.
\end{proof}

\begin{lemma}\label{lemma:same_orbit_x}
If $S, S' \in \mathcal{SPD}(G, \mathscr{H})$ are two trees of type $X$ which are distance $2$ apart, such that the non trivial vertex subgroups of a fundamental domain of the tree $S$ are $A_1,..., A_n$. Then there is an outer automorphism $\phi \in \Gamma_n'$ such that $\phi(S) = S'$
\end{lemma}

\begin{proof}
If the distance between $S$ and $S'$ is $2$, then there is a tree $T$ such that distance of $T$ from $S$ is $1$ and the distance of $T$ from $S'$ is also $1$. We will prove that $T$ must be a tree of type $Y$. We will show that to get to $S'$ from $S$ via another tree $T$ we must collapse an orbit of edge $G$-equivariantly of $S$ and then expand an orbit of vertex $G$-equivariantly of $T$.

Complete list of vertex stabilizer subgroups in a fundamental domain uniquely (up to equivariant homeomorphism) determines a tree of type $X$. So, $S$ and $S'$ do not have a fundamental domain whose non trivial vertex stabilizer subgroup match. To move to a different tree in $\mathcal{SPD}$ from $S$ we must apply either a collapse move or an expand move. 

Expand move must be applied to the orbit of vertices with trivial vertex stabilizer subgroup, as the other orbits of vertices have valence $1$, when restricted to any fundamental domain. The vertex groups of any fundamental domain for a tree of type $X$ is at the extremity of the fundamental domain. Every tree obtained from applying only expand move to a tree of type $X$ must also have a fundamental domain that has all the non trivial vertex groups in the extremities of the fundamental domain.

Any tree with a fundamental domain that has all the vertex groups at the extremities of the fundamental domain does not have edge overlap from two distinct fundamental domains. As a result, one expand move followed by one collapse move on a tree of type $X$ does not give rise to a different tree due to inability to choose a different fundamental domain. So, to get to a different tree of type $X$, we need to apply collapse move first and then expand move.

There can only be $1$-edge orbit collapse move. Let the edge adjacent to the vertex group labeled by $A_i$ be collapsed, equivariantly. Let us denote this tree by $T_i$. Then we have to apply expand move on the vertex labeled by $A_i$ of the tree $T_i$ to get $S'$.

The choices of vertices for a fundamental domain of $T$ are as follows:
\begin{itemize}[wide, labelwidth=!, labelindent=0pt]
\item In the $A_1*A_i$ minimal subtree we can choose vertex labeled by $a_{i1}A_1a_{i1}^{-1}$.
\item In the $A_2*A_i$ minimal subtree we can choose vertex labeled by $a_{i2}A_2a_{i2}^{-1} $.

\vdots
\item In the $A_{i - 1}*A_i$ minimal subtree we can choose vertex labeled by $a_{i{i-1}}A_{i - 1}a_{i{i - 1}}^{-1} $.
\item $A_i$.
\item In the $A_{i + 1}*A_i$ minimal subtree we can choose vertex labeled by $a_{i{i + 1}}A_{i + 1}a_{i{i + 1}}^{-1} $.

\vdots
\item In the $A_n*A_i$ minimal subtree we can choose vertex labeled by $a_{in}A_na_{in}^{-1} $
Here, $a_{ik} \in A_i, \text{for } k \in \{1,..., n\}.$

\end{itemize}
For such a choice, the vertex stabilizer subgroup of a fundamental domain of $S'$ are given by 

$a_{i1}A_1a_{i1}^{-1}$
, $a_{i2}A_2a_{i2}^{-1} ,\hdots , a_{i{i-1}}A_{i - 1}a_{i{i - 1}}^{-1} $, $A_i$, $a_{i{i + 1}}A_{i + 1}a_{i{i + 1}}^{-1} , \hdots , a_{in}A_na_{in}^{-1} $.

In this case, $\phi := f_{A_n}^{a_{in}} \hdots f_{A_{i + 1}}^{a_{i{i + 1}}} f_{A_{i - 1}}^{a_{i{i - 1}}} \hdots f_{A_1}^{a_{i1}} $.
\end{proof}

\begin{corollary}\label{cor:fin_ind}
$\Gamma_n'$ is a finite index subgroup of $\Omega_n$
\end{corollary}

\begin{proof}
Consider the $1$-skeleton of the sub-complex of $\mathcal{SPD}(G,\mathscr{H})$ spanned by trees of type $X$ and $Y$. By corollary \ref{cor:xy_connected}, it is connected. By lemma \ref{lemma:same_orbit_x}, any two trees of type $X$ which are distance $2$ apart are in the same $\Gamma_n'$ orbit. Hence,  all trees of type $X$ corresponding to a vertex of $\mathcal{SPD}$ are in the same $\Gamma_n'$ orbit. $\Gamma_n'$ acts co-compactly as there is only $1$ orbit of $0$-cells of trees of type $X$. The action is properly discontinuous as $\Gamma_n' \leq {\Omega_n}$.

So, by Milnor-\u{S}varc lemma, lemma \ref{lemma:mil_sva}, $\Gamma_n' $ is finite index in $\Omega_n$.
\end{proof}

\begin{remark}
The elements considered by McCullough-Miller, see \cite{MM}, were of the form $f_H^w$, as well. A key difference is we are restricting further by requiring $w \notin H$. So, $\Gamma_n'$ is a proper subgroup of the symmetric outer automorphisms considered by them.
\end{remark}
\section{Algebraically Thick Groups}\label{section:thick}

A major tool used in the investigation of the original question (relative hyperbolicity of $\mathsf{Out}(A_1*...*A_n)$) for higher complexities is \textit{algebraic thickness}. Theorem \ref{thm:thickNRH} by Behrstock-Dru\c{t}u-Mosher underscores the relevance of the study of algebraic thickness. According to theorem \ref{thm:thickNRH}, thickness of a finitely generated group implies non-relative hyperbolicity of the group. Thickness has been developed in full generality by Behrstock-D\c{r}utu-Mosher in \cite{BDM}. 

In section \ref{section:thick_defn}, we will briefly describe the terms related to the definition of algebraic thickness. Our exposition closely follow the exposition in \cite{BDM}. We will start by defining a \textit{non-principal ultrafilter} in definition \ref{defn:ultrafilter}. Then, we will define \textit{ultralimit} in definition \ref{defn:ultralimit}. Using the concept of ultralimit of a family of metric spaces we will define the \textit{asymptotic cone} of a metric space $(X, \text{dist})$ in definition \ref{defn:asymptotic_cone}. We will use the concepts of ultrafilter and asymptotic cone to define an \textit{unconstricted metric space} in definition \ref{defn:unconstricted_space}. Algebraic thickness of a group is an inductive property, where the base case or algebraically thick group of order at most zero are groups which are unconstricted. We will use the notion of \textit{algebraic network of subgroups}, definition \ref{defn:network}, to define algebraic thickness in higher order, definition \ref{defn:thickness}.

\subsection{Definition}\label{section:thick_defn}
\begin{definition}\label{defn:ultrafilter}
A \textit{non-principal ultrafilter} on the positive integers, denoted by $\omega$, is a non-empty collection of sets of positive integers with the following properties:
\begin{enumerate}[wide, labelwidth=!, labelindent=0pt]
\item If $S_1 \in \omega, \text{and } S_2 \in \omega$, then $S_1 \cap S_2 \in \omega$.
\item If $S_1 \subset S_2$ and $S_1 \in \omega$, then $S_2 \in \omega$.
\item For each $S \subset \mathbb{N}$ exactly one of the following must occur: $S \in \omega$ or $\mathbb{N}\\S \in \omega$.
\item $\omega$ does not contain any finite set.
\end{enumerate}
\end{definition}

\begin{definition}\label{defn:ultralimit}
For a non-principal ultrafilter $\omega$, a topological space $X$, and a sequence of points $(x_i)_{i \in \mathbb{N}}$ in $X$, we define $x$ to be the \textit{ultralimit} of $(x_i )_{i \in \mathbb{N}}$ with respect to $\omega$, and we write $x = \text{lim}_{\omega} x_i$, if and only if for any neighborhood $\mathcal{N}$ of $x$ in $X$ the set $\{i \in \mathbb{N} : x_i \in \mathcal{N}\}$ is in $\omega$.
\end{definition}

\begin{remark}
\begin{enumerate}[wide, labelwidth=!, labelindent=0pt]
\item When $X$ is compact any sequence in $X$ has an ultralimit.
\item If moreover $X$ is Hausdorff then the ultralimit of any sequence is unique.
\end{enumerate}
\end{remark}

Fix a non-principal ultrafilter $\omega$ and a family of based metric spaces $(X_i, x_i, \text{dist}_i)$. Using the ultrafilter, a pseudo distance on $\displaystyle\prod_{i \in \mathbb{N}}X_i$ is provided by:
$$\text{dist}_{\omega}((a_i ), (b_i )) = \text{lim}_\omega \text{dist}_i (a_i , b_i ) \in [0,\infty].$$

One can eliminate the possibility of the previous pseudo-distance taking the value $\infty$
by restricting to sequences $y = (y_i )$ such that $\text{dist}_\omega(y, x) < \infty$, where $x = (x_i )$. A
metric space can be then defined, called the \textit{ultralimit} of $(X_i , x_i , \text{dist}_i )$, by:
$$\text{lim}_\omega (X_i , x_i , \text{dist}_i ) = \left\{ y \in \displaystyle\prod_{i \in \mathbb{N}} X_i : \text{dist}_\omega(y, x) < \infty \right\} / \sim,$$
where two points $y, z \in \displaystyle\prod_{i\in \mathbb{N}}X_i$ we define $y \sim z$ if and only if $\text{dist}_\omega(y, z) = 0.$ The pseudo-distance on $\displaystyle\prod_{i\in \mathbb{N}}X_i$ induces a complete metric on $\text{lim}_\omega (X_i , x_i , \text{dist}_i )$.

\begin{definition} \label{defn:asymptotic_cone}
For a metric space $(X, \text{dist})$, consider $x = (x_n)$ a sequence of points in
$X$, called \textit{observation points}, and $d = (d_n)$ a sequence of positive numbers such that
$\text{lim}_\omega d_n = \infty$, called scaling constants. The \textit{asymptotic cone
of} $(X, \text{dist})$ \textit{relative to the non-principal ultrafilter} $\omega$ \textit{and the sequences} $x$ \textit{and} $d$ is given by:
$\text{Cone}_\omega(X, x, d) = \text{lim}_\omega\left( X, x_n, \dfrac{1}{d_n} \text{dist} \right).$
\end{definition}

\begin{remark}
\textit{Convention}: If $X$ is a connected metric space and $X - \{x\}$ is not connected then $x$ is a cut point of $X$. By cut-points we always mean global cut-points. We consider a singleton
to have a cut-point.
\end{remark}

\begin{definition}\label{defn:unconstricted_space}[Definition 3.1 (Unconstricted space/ group)]\cite{BDM}
A path connected metric space $B$ is \textit{unconstricted} if the following two properties hold:
\begin{enumerate}[wide, labelwidth=!, labelindent=0pt]
\item there exists a non-principal ultrafilter $\omega$ and a sequence $d$ such that for every sequence of
observation points $b$, $\text{Cone}_\omega(B, b, d)$ does not have cut-points;
\item for some constant $c$, every point in $B$ is at distance at most $c$ from a bi-infinite
geodesic in $B$.
\end{enumerate}
An infinite finitely generated \textit{group is unconstricted} if at least one of its asymptotic cones does not have cut-points.
\end{definition}

\begin{definition}\label{defn:network}[Definition 5.2(Algebraic network of subgroups)]\cite{BDM} Let $G$ be a finitely generated group, let $\mathcal{H}$ be a finite collection of subgroups of $G$ and let $M > 0$. The group $G$ is an $M-$\textit{algebraic network with respect to }$\mathcal{H}$ if:
\begin{enumerate}[wide, labelwidth=!, labelindent=0pt]
\item[$\mathbf{AN_0}$] All subgroups in $\mathcal{H}$ are finitely generated and undistorted in $G$.

\item[$\mathbf{AN_1}$] There is a finite index subgroup $G_1$ of $G$ such that $G \subset \mathcal{N}_M(G_1)$, such that a finite generating set of $G_1$ is contained in $\displaystyle\bigcup_{H \in \mathcal{H}}H$.

\item[$\mathbf{AN_2}$] Any two subgroups $H, H'$ in $\mathcal{H}$ can be \textit{thickly connected} in $\mathcal{H}$: there exists a finite sequence $H = H_1,..., H_n = H'$ of subgroups in $\mathcal{H}$ such that for all $1 \leq i < n, H_i \cap H_{i + 1}$ is infinite.

\end{enumerate}
\end{definition}

\begin{definition}\label{defn:thickness}[Definition 7.3(Algebraic thickness)]\cite{BDM} Consider a finitely generated group $G$.
\begin{enumerate}[wide, labelwidth=!, labelindent=0pt]
\item[$\mathbf{A_1}$] $G$ is called \textit{algebraically thick of order zero} if it is unconstricted.

\item[$\mathbf{A_2}$] $G$ is called $M$-\textit{algebraically thick of order at most} $n + 1$ with respect to $\mathcal{H}$, where $\mathcal{H}$ is a finite collection of subgroups of $G$ and $M > 0$, if:
\begin{enumerate}[wide, labelwidth=!, labelindent=0pt]
\item[-] $G$ is an $M$-algebraic network with respect to $\mathcal{H}$;
\item[-] all subgroups in $\mathcal{H}$ are algebraically thick of order at most $n$.
\end{enumerate}

\end{enumerate}
$G$ is said to be \textit{algebraically thick of order at most} $n + 1$ \textit{with respect to} $\mathcal{H}$ if there is a $M>0$, such that $G$ is $M$-\textit{algebraically thick of order at most} $n + 1$ with respect to $\mathcal{H}$. $G$ is said to be \textit{algebraically thick of order} $n + 1$ \textit{with respect to} $\mathcal{H}$, when $G$ is algebraically thick of order at most $n + 1$ and $G$ is not algebraically thick of order at most $n$.
 \end{definition}
 
\begin{remark}
We will show that in higher complexities, $\Gamma_n'$ is algebraically thick of order at most $1$. In order to show thickness, we will start from a collection of an algebraic network of undistorted, zero thick subgroups. A subgroup is zero thick if it is unconstricted \cite[Definition 3.4]{BDM}. Examples inspired by the following class of unconstricted spaces will be our base case. \end{remark}

\begin{example}\cite[Definition 3.4]{BDM}
A cartesian product of two geodesic metric spaces of infinite diameter is an example of an unconstricted space.
\end{example}

\begin{remark}
A Cayley graph of direct product of groups which have infinite diameter is an example of an unconstricted space.
\end{remark}

\subsection{In search for thickly connected subgroups} In this section we will define subgroups generated by carefully selected elements from the set of generators defined in section \ref{subsec:fin_index}. These subgroups will serve as building blocks for potential thickly connected network of $0$-thick subgroups. These subgroups are $H_i^j$ from definition \ref{defn:hij}.


\subsubsection{Some thickly connected subgroups of $\Gamma_4'$} 

We will consider two separate cases to investigate thickly connected subgroups of $\Gamma_4$
\begin{case}\label{case:G4_abelian}
Each $A_i$ is abelian. We choose to portray this separately as the subgroups used for this case is similar to the subgroups used for $\Gamma_n', (n > 4)$.
\end{case}

\begin{case}\label{case:G4_nonabelian} In general we will not assume that $A_i$s are abelian  and investigate thickly connected subgroups $\Omega_4$ (Definition \ref{defn:fgamma_n}).
\end{case}

Case \ref{case:G4_abelian}: In this case we will consider $H_i^j$ from definition \ref{defn:hij}, such that $i \ne j$. We will organize the generating subgroups, $H_i^j$, of $\Gamma_4'$ (definition \ref{defn:GN'}) into the following table. A subgroup generated by any two subgroups in a row is a direct product of those two subgroups by proposition \ref{prop:hij_sum}. Subgroup generated by any two subgroups in a column is a free product of those two subgroups by proposition \ref{prop:hij_prod}. 

\begin{center}
\renewcommand{\arraystretch}{1.2}
\begin{tabular}{c|c|c|c}
		& $\h{2}{1}$ 	& $\h{3}{1}$ \cellcolor{blue!25}	& $\h{4}{1}$ \cellcolor{blue!25}	\\ \hline

$\h{1}{2}$	& 			& $\h{3}{2}$ \cellcolor{blue!25}	& $\h{4}{2}$ \cellcolor{blue!25} \\ \hline
  
$\h{1}{3}$ \cellcolor{red!25}	& $\h{2}{3}$ \cellcolor{red!25}	&	  		& $\h{4}{3}$ \\ \hline
  
$\h{1}{4}$ \cellcolor{red!25}	& $\h{2}{4}$ \cellcolor{red!25}	& $\h{3}{4}$	& 		 \\ 
\end{tabular}
\end{center}

\begin{lemma}\label{lemma:gen_G4}
If each $A_i$ is an abelian group, then the subgroups in the shaded region of the table generate $\Gamma_4'$
\end{lemma}
\begin{proof}
Fix $a_1 \in A_1$, then 
\begin{gather*}f_{A_2}^{a_1}f_{A_3}^{a_1}f_{A_4}^{a_1}(a) = 
\begin{cases} 
a_1aa_1^{-1}, \text{when }a \in A_2 \cup A_3 \cup A_4 \\
a, \text{when }a \in A_1
\end{cases}\\
\text{If, $A_1$ is abelian}
\implies  a =  a_1aa_1^{-1}, 
\text{when }a\in A_1\\
\text{Hence, }
f_{A_2}^{a_1}f_{A_3}^{a_1}f_{A_4}^{a_1}  =  id \text{,    (Conjugation by $a_1$)}\\
\implies f_{A_2}^{a_1} = \left(f_{A_3}^{a_1}\right)^{-1} \left(f_{A_4}^{a_1}\right)^{-1}
\implies H_2^1 \subset \la H_3^1, H_4^1 \ra \\
\text{Similarly, }
H_1^2 \subset \la H_3^2, H_4^2 \ra,
H_4^3 \subset \la H_1^3, H_2^3 \ra \text{ and }
H_3^4 \subset \la H_1^4, H_2^4 \ra
\end{gather*}
\end{proof}

Now we will define some $0$-thick subgroups of $\Gamma_4'$ such that together they can be potential candidates for proving thickness of $\Gamma_4'$. 
\begin{definition}\label{defn:n1234}
Fix non identity elements $a_i \in A_i$.
\begin{gather*}g_{12} := f_{A_3}^{a_1}f_{A_3}^{a_2}f_{A_4}^{a_1}f_{A_4}^{a_2} \in  \left(H_3^1 * H_3^2\right) \oplus \left(H_4^1 * H_4^2\right)\\
g_{34} := f_{A_1}^{a_3}f_{A_1}^{a_4}f_{A_2}^{a_3}f_{A_2}^{a_4} \in \left(H_1^3 * H_1^4\right) \oplus \left(H_2^3 * H_2^4\right)\\
\text{We will use the following notations ($N^{ij}$) in subsection \ref{subsec:undistorted2} in a more general}\\\text{capacity. For this section let us define $N^{12}, N^{34}$.}\\
N^{12} := \la g_{12} \ra  \cong \mathbb{Z}\\
N^{34} := \la g_{34} \ra \cong \mathbb{Z}\\
\text{Define a subgroup $H_3$ as follows. The last equality will be proved in corollary \ref{cor:nij_zerothick}}\\
H_3 := \la g_{12}, g_{34} \ra = \la N^{12}, N^{34} \ra \cong \mathbb{Z} \oplus  \mathbb{Z}\end{gather*}

\end{definition}

\begin{notation}\label{not:h1234n1234}
List of the $0$-thick subgroups we will consider using notations used in sections \ref{section:undistorted}, \ref{subsec:undistorted2} for further discussions are as follows
\begin{enumerate}[wide, labelwidth=!, labelindent=0pt]
\item $H^{12} := \left(H_3^1 * H_3^2\right) \oplus \left(H_4^1 * H_4^2\right) $
\item $H^{34} := \left(H_1^3 * H_1^4\right) \oplus \left(H_2^3 * H_2^4\right) $
\item $H_3$

\end{enumerate}
\end{notation}
If we combine all the information from this section. We get
\begin{gather*}\la H^{12}, H^{34}, H_3 \ra = \Gamma_4'\\
H^{12} \cap H_3 \cong H^{34} \cap H_3 \cong \mathbb{Z}\\
H^{12} \cong \left(A_1 * A_2\right) \oplus \left(A_1 * A_2\right)\\
H^{34} \cong \left(A_3 * A_4\right) \oplus \left(A_3 * A_4\right)\\
H_3 \cong \mathbb{Z} \oplus  \mathbb{Z}\end{gather*}
Hence, $\Gamma_4'$ will be algebraically thick of order at most $1$ when each $A_i$ is abelian, if we can prove that $H^{12}, H^{34}$ and $H_3$ are undistorted subgroups in $\Gamma_4'$. We will do this in sections \ref{section:undistorted}, \ref{subsec:undistorted2}.

Case \ref{case:G4_nonabelian}: Now we will not assume that $A_i$s are abelian. Here, we will investigate a finitely generated subgroup $M_4 \leq {\Omega_4}$ (Definition \ref{defn:fgamma_n}) for thickly connected subgroups. The definition will imply $\Gamma_4' \leq M_4 \leq \Omega_4$. So, $M_4$ will be a finite index subgroup of $\Gamma_4 (= \mathsf{Out(A_1*A_2*A_3*A_4)})$.

\begin{definition} \label{defn:M4}
$M_4 : = \la H_i^j \vert i, j \in \{1, 2, 3, 4\} \ra$. Recall the definition of $H_i^j$ from definition \ref{defn:hij}.
\end{definition}
With the notations described in definition \ref{defn:M4}, subgroups $H_i^j \leq M_4$ can be organized in a table similar to the previous case - 

\begin{center}
\renewcommand{\arraystretch}{1.2}
\begin{tabular}{c|c|c|c}\label{tab:G4}
$\h{1}{1}$ \cellcolor{green!25}	& $\h{2}{1}$  \cellcolor{green!25}	& $\h{3}{1}$ \cellcolor{blue!25}	& $\h{4}{1}$ \cellcolor{blue!25}	\\ \hline

$\h{1}{2}$	  \cellcolor{green!25}& 	$\h{2}{2}$ \cellcolor{green!25}		& $\h{3}{2}$ \cellcolor{blue!25}	& $\h{4}{2}$ \cellcolor{blue!25} \\ \hline
  
$\h{1}{3}$ \cellcolor{red!25}	& $\h{2}{3}$ \cellcolor{red!25}	&	  $\h{3}{3}$ \cellcolor{green!25}		& $\h{4}{3}$  \cellcolor{green!25}\\ \hline
  
$\h{1}{4}$ \cellcolor{red!25}	& $\h{2}{4}$ \cellcolor{red!25}	& $\h{3}{4}$	 \cellcolor{green!25}& $\h{4}{4}$ \cellcolor{green!25}		 \\ 
\end{tabular}
\end{center}

For a thickly connected network of $M_4$, we have to consider the following subgroups in addition to the subgroups $H^{12}, H^{34}$ (see notation \ref{not:h1234n1234}) considered in the previous case.

\begin{definition} \label{defn:M1234}
$M^{12} := \la \h{1}{1}, \h{2}{1}, \h{1}{2}, \h{2}{2}\ra; M^{34} := \la \h{3}{3}, \h{4}{3}, \h{3}{4}, \h{4}{4}\ra$

\end{definition}

\begin{lemma}\label{lemma:M1234thick}
$\la M^{12}, M^{34} \ra =  M^{12} \oplus M^{34}$
\end{lemma}

\begin{proof}
\begin{enumerate}[wide, labelwidth=!, labelindent=0pt]
\item The generating subgroups of $M^{12}$ commute with the generating subgroups of $M^{34}$. Hence, $M^{12}, M^{34} \trianglelefteq \la M^{12}, M^{34} \ra$

\item Now we will show that $M^{12} \cap M^{34} = \{id\}$. We will show this by considering the action of a generic element of $M^{12}$ and a generic element of $M^{34}$ on the graph of groups $\mathbf{X}$ described below

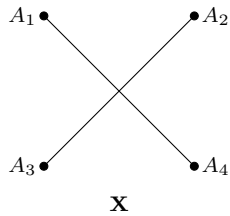
\begin{center}
\begin{tikzpicture} \begin{scriptsize}
\begin{scriptsize}
\Tw{5}{0}{A_1}{A_2}{A_3}{A_4}{$\mathbf{X}$} 
\end{scriptsize}
\end{scriptsize} \end{tikzpicture} 
\end{center}

Let, $m_{12} \in M^{12}$ and $m_{34} \in M^{34}$. Then, $m_{12}m_{34}(\mathbf{X}) = $

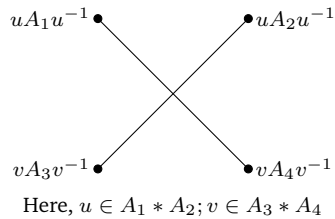
\begin{center}
\begin{tikzpicture} \begin{scriptsize}
\begin{scriptsize}
\Tw{5}{0}{uA_1u^{-1}}{uA_2u^{-1}}{vA_3v^{-1}}{vA_4v^{-1}}{ Here, $u \in A_1*A_2; v \in A_3*A_4$} 
\end{scriptsize}
\end{scriptsize} \end{tikzpicture} 
\end{center}

\begin{gather*} M^{12} \cap M^{34} \ne \{id\} \implies \exists m_{12}, m_{34}\\
 \text{such that } m_{12}m_{34}(\mathbf{X}) = \mathbf{X}\\
 \iff m_{12}m_{34}(T_\mathbf{X}) = T_\mathbf{X}\end{gather*}
 By uniqueness of $A_1*A_2$-minimal subtree and $A_3*A_4$-minimal subtree in every tree of $\mathcal{SPD}$, 
 $m_{12}m_{34}(\mathbf{X}) = \mathbf{X} \implies u = v \implies u = v = id$
\end{enumerate}

Hence, $\la M^{12}, M^{34} \ra =  M^{12} \oplus M^{34}$.
\end{proof}

\begin{remark}\label{rem:M4_connection}
\begin{enumerate}[wide, labelwidth=!, labelindent=0pt]
\item $M^{12}, M^{34}$ each contain an element of infinite order. Fix $a_i \in A_i / \{id_{A_i}\} (i \in \{1, 2, 3, 4\})$, then the elements of infinite order are 
$ f_{A_1}^{a_1}f_{A_2}^{a_1}f_{A_1}^{a_2}f_{A_2}^{a_2} \in M^{12},
f_{A_3}^{a_3}f_{A_4}^{a_3}f_{A_3}^{a_4}f_{A_4}^{a_4}\in M^{34}$.

\item $ f_{A_1}^{a_1}f_{A_2}^{a_1}f_{A_3}^{a_1}f_{A_4}^{a_1} = id_{\Gamma_4}
\implies  f_{A_1}^{a_1}f_{A_2}^{a_1} = \left(f_{A_4}^{a_1}\right)^{-1}\left(f_{A_3}^{a_1}\right)^{-1} \in \la \h{3}{1}, \h{4}{1}\ra $. 
Similarly, $f_{A_1}^{a_2}f_{A_2}^{a_2} = \left(f_{A_4}^{a_2}\right)^{-1}\left(f_{A_3}^{a_2}\right)^{-1} \in \la \h{3}{2}, \h{4}{2}\ra
\implies  f_{A_1}^{a_1}f_{A_2}^{a_1}f_{A_1}^{a_2}f_{A_2}^{a_2} \in  \la \h{1}{3}, \h{1}{4}, \h{2}{3}, \h{2}{4}\ra = H^{34}$.
\item Similarly, $f_{A_3}^{a_3}f_{A_4}^{a_3}f_{A_3}^{a_4}f_{A_4}^{a_4} \in  \la \h{3}{1}, \h{3}{2}, \h{4}{1}, \h{4}{2}\ra = H^{12}$.
\end{enumerate}
$\left(M^{12} \oplus M^{34}\right) \cap H^{12}, \left(M^{12} \oplus M^{34}\right) \cap H^{34} $ is infinite. Hence the thickly connected subgroups of $M_4$ are $H^{12}, H^{34}, \left(M^{12} \oplus M^{34}\right)$. To prove thickness of $M_4$ we will show in section \ref{subsec:M4_nondistortion} that all of the above subgroups are undistorted.
\end{remark}


\subsection{Some thickly connected subgroups of $\Gamma_n'$, when $n \geq 5$} In this subsection we will generalize the analysis of $\Gamma_4'$ in case \ref{case:G4_abelian} to $\Gamma_n', n\geq 5$. A major difference when $n \geq 5$ is that potential algebraic networks can be found in $\Gamma_n'$ without any assumptions on the free factors, $A_i$ (For $\Gamma_4'$ in case \ref{case:G4_abelian}, we assumed each $A_i$ is abelian).

In accordance to our discussion of $\Gamma_4'$ in case \ref{case:G4_abelian}, we have organized the subgroups $H_i^j, (i \ne j,\text{ and } i, j\in \{1,..., n\})$ in the following table. $H_i^j, (i \ne j,\text{ and } i, j\in \{1,..., n\})$ will be the building blocks for the $0$-thick subgroups, which can form algebraic network if the $0$-thick subgroups are quasi isometrically embedded in $\Gamma_n'$. In contrast to the case \ref{case:G4_nonabelian}, the diagonal groups $H_i^i$s have not been considered.

\begin{center}
\renewcommand{\arraystretch}{1.2}
\begin{tabular}{c|c| c |c|c}\label{tab:GN}
		& $\h{2}{1}$ 	& $\hdots$	& $\h{n-1}{1}$ 	& $\h{n}{1}$	\\ \hline
				
$\h{1}{2}$	& 			& $\hdots$ 	& $\h{n-1}{2}$	& $\h{n}{2}$	\\ \hline

$\vdots$	& $\vdots$	& $\ddots$ 	& $\vdots$	& $\vdots$	\\  \hline
  
$\h{1}{n-1}$& $\h{2}{n-1}$	& $\hdots$  	& 			& $\h{n}{n-1}$ 	\\ \hline
  
$\h{1}{n}$	& $\h{2}{n}$	& $\hdots$	& $\h{n-1}{n}$	& 		 	\\ 
\end{tabular}
\end{center}

The notation for the two different classes of subgroups that we will consider are $H^{ij} ( i \ne j \in \{1,...,n\})$ and $\la N^{i_1i_2}, N^{i_3i_4}\ra (i_1, i_2, i_3, i_4$ are distinct integers from the set $\{1,..., n\})$.
\begin{definition}\label{defn:H^ij}
$H^{ij} :=\displaystyle \bigoplus_{\substack{k \ne i, j\\ k =1}}^{k = n} H_{k}^i * H_{k}^j$
\end{definition}
We observe that, $\Gamma_n' \subset \la \displaystyle\bigcup_{i \ne j}^n H^{ij} \ra$. Now we will define an infinite order element of $H^{ij}$, and call the group generated by that element as $N^{ij}$

\begin{definition}
Fix distinct integers $i, j \in \{1,..., n\}$ and $x_i \in A_i - \{id_{A_i}\}, x_j \in A_j - \{id_{A_j}\}. 
 \text{ Define an outer automorphism, } f^{ij} := \displaystyle\prod_{\substack{k \ne i, j\\k =1}}^n \left(f_{A_k}^{x_i} f_{A_k}^{x_j}\right) \in H^{ij}$ and a subgroup of 
$\Gamma_n' \geq N^{ij} := \la f^{ij} \ra.$

\end{definition}
 
 In section \ref{subsec:undistorted2}, we will prove the following results
 
 \begin{enumerate}[wide, labelwidth=!, labelindent=0pt]

\item $N^{ij} \cong \mathbb{Z}$

\item $\la N^{i_1i_2}, N^{i_3i_4} \ra \cong \mathbb{Z} \oplus \mathbb{Z}$,
	where $i_1, i_2, i_3, i_4$ are all different integers.

\item $\la N^{i_1i_2}, N^{i_3i_4} \ra$ is undistorted in $\Gamma_n'$

\end{enumerate}

The  following corollary follows from definition of $H^{ij}$ and $N^{ij}$
\begin{corollary}\label{cor:thickly_connected}
If $i_1, i_2, i_3, i_4$ are distinct integers from the set $\{1,..., n\}$, then the collection of subgroups of the form $\{H^{i_1i_2}, \la N^{i_1i_2}, N^{i_3i_4}\ra, H^{i_3i_4}\}$ constitute a thickly connected collection of subgroup of $\Gamma_n'$, where $n > 4$.
\end{corollary}

\section{Some Undistorted Subgroups of $\Gamma_n'$}\label{section:undistorted}
In this section we will prove that $H^{ij}, N^{ij}$ and $\left( M^{12} \oplus M^{34} \right)$ discussed in section \ref{section:thick} are quasi isometrically embedded in $\Gamma_n'$. The idea of the proof of non-distortion of $H^{ij}, \left( M^{12} \oplus M^{34} \right)$ is inspired by work of Handel-Mosher \cite{HM}. Proof of non-distortion of $N^{ij}$ is inspired by work of Alibegovi\'c \cite{Ali}.
\subsection{An important class of undistorted subgroups of $\Gamma_n'$}
In this section we will prove that $H^{ij}$ is quasi isometrically embedded in $\Gamma_n'$. The strategy of the proof is to find a sub-complex $\mathscr{K}^{ij}$ of $\mathcal{SPD}$ on which $H^{ij}$ acts geometrically and there is a Lipschitz retraction from $\mathcal{SPD}$ to $\mathscr{K}^{ij}$ implying quasi isometric embedding of $\mathscr{K}^{ij}$ into $\mathcal{SPD}$.

\begin{definition}\label{defn:Kij}

\begin{enumerate}[wide, labelwidth=!, labelindent=0pt]Define $\mathscr{K}^{ij}$ to be the sub-complex of $\mathcal{SPD}(G,\mathscr{H})$
\item Spanned by vertices of $\mathcal{SPD}$ which are trees with a fundamental domain containing vertices stabilized by $A_i$ and $A_j$. 
\item The other vertices in the fundamental domain are stabilized by conjugates of $A_k, k \ne i, j$ and the conjugating elements are from the subgroup $A_i * A_j$. 

\end{enumerate}

\end{definition}

\begin{example}
A graph of groups representing a vertex of $\mathscr{K}^{{ij}}$ from definition \ref{defn:Kij} is given below: 
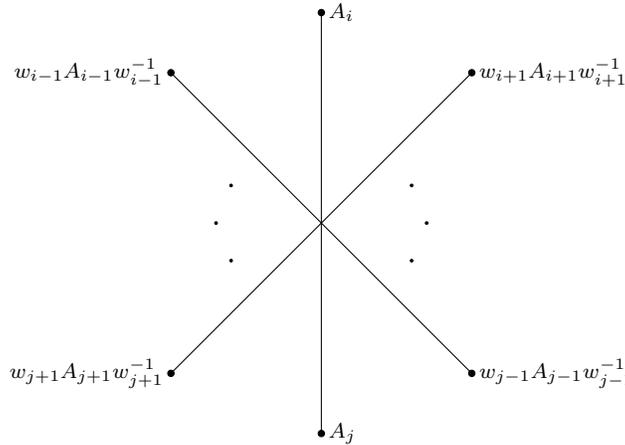
\begin{center}
\begin{tikzpicture} \begin{scriptsize}
\draw (2, 2)  -- (-2,-2);
\draw (-2, 2)  -- (2,-2);
\draw (0, 2.8) -- (0, -2.8);

\draw [fill=black] (2,-2) circle (1.2pt) node[anchor=west]{$w_{j-1}A_{j-1}w_{j-1}^{-1}$};
\draw [fill=black] (0,2.8) circle (1.2pt) node[anchor=west]{$A_{i}$};
\draw [fill=black] (2,2) circle (1.2pt) node[anchor=west]{$w_{i+1}A_{i+1}w_{i+1}^{-1}$};
\draw [fill=black] (0,-2.8) circle (1.2pt) node[anchor=west]{$A_{j}$};
\draw [fill=black] (-2,-2) circle (1.2pt) node[anchor=east]{$w_{j+1}A_{j+1}w_{j+1}^{-1}$};
\draw [fill=black] (-2,2) circle (1.2pt) node[anchor=east]{$w_{i-1}A_{i-1}w_{i-1}^{-1}$};
\draw [fill=black] (1.4, 0) circle (.5pt);
\draw [fill=black] (1.2, .5) circle (.5pt);
\draw [fill=black] (1.2, -.5) circle (.5pt);
\draw [fill=black] (-1.4, 0) circle (.5pt);
\draw [fill=black] (-1.2, -.5) circle (.5pt);
\draw [fill=black] (-1.2, .5) circle (.5pt);
\end{scriptsize} \end{tikzpicture} 
\end{center}

where, $w_k \in A_i * A_j$ for all $k $. Recall that, a graph of groups whose underlying graph is isomorphic to the underlying graph of this graph of groups is called a graph of groups of type $X$ (in accordance with notation \ref{not:X}). 

\end{example}

\begin{lemma} \label{lemma:kconnected}

Let $\mathbf{X}\in \mathscr{K}^{{ij}^0}$ denote a graph of groups of type $X$ with non trivial vertex groups $A_1,..., A_n$. Consider a graph of groups $\mathbf{X'} \in \mathscr{K}^{{ij}^0}$ of type $X$, then $\mathbf{X}$ and $\mathbf{X'}$ can be connected by a path in $\mathscr{K}^{ij}$.
\end{lemma}

\begin{proof}
The proof will be broken down into two parts: In part 1. We will assume that $\mathbf{X}$ and $\mathbf{X'}$ only differ at one vertex (The vertex labeled by the conjugate of the group $A_p, \text{ where $p$ is an arbitrary fixed integer from the set } \{1,.., n\} - \{i, j\}$). In the second part we will consider more general $\mathbf{X'}$.

\begin{enumerate}[wide, labelwidth=!, labelindent=0pt]
\item

Consider, graph of groups $\mathbf{X_1}, \mathbf{X_2} \in \mathscr{K}^{{ij}^0}$ of type $X$. Assume that $\mathbf{X_1}$ and $\mathbf{X_2}$ are identical except for the vertex corresponding to vertex group congruent to $A_p$, where $p \in \{1,.., n\} - \{i, j\}$ is an arbitrary fixed integer. The vertex group congruent to $A_p$ in $\mathbf{X_1}$ is $A_p$; whereas in $\mathbf{X_2}$ the vertex group congruent to $A_p$ is $wA_pw^{-1}$ (where, $w \in A_i*A_j)$. In this proof we will show that in such a situation $\mathbf{X_1}$ and $\mathbf{X_2}$ can be connected by a path in $\mathscr{K}^{{ij}}$. 

First let us assume, that $w = vu$ is a word of length $2$, such that $u \in  A_i$ and $v \in A_j$. So, $\mathbf{X_2}=f_{A_p}^w(\mathbf{X_1})$, where $f_{A_p}^w \in \Gamma_n$ has been defined in definition \ref{defn:F^w_H}. In $f_{A_p}^w(\mathbf{X_1}) (=\mathbf{X_2})$, the non-trivial vertex group conjugate to $A_p$ is $uvA_pv^{-1}u^{-1}$.

We will give a collapse-expand route from $\mathbf{X_1}$ to $f_p(\mathbf{X_2})$ lying in $\mathscr{K}^{ij}$. Observe that our argument is inductive and we have started with the base case where $w$ is a word of length $2$ (instead of $1$). However, the description of the collapse-expand path when $w$ has word length $1$ is contained in part a of the base case.
\begin{enumerate}[wide, labelwidth=!, labelindent=0pt]
\item 
\begin{enumerate}[wide, labelwidth=!, labelindent=0pt]
\item \textbf{Collapse:} Starting from $T_{\mathbf{X_1}}$ we collapse the edges adjacent to the vertex labeled by $A_j$, equivariantly. 
\item \textbf{Expand:} Expand the edges adjacent to $A_j$ after choosing the vertex labeled by $vA_pv^{-1}$ in the $A_j*A_p$ minimal subtree as a replacement vertex for the vertex label $A_p$ of the fundamental domain. \end{enumerate}

\item Starting from this tree we follow a similar procedure as described above to obtain $f_{A_p}^w(\mathbf{X_1})$. 
\begin{enumerate}[wide, labelwidth=!, labelindent=0pt]
\item \textbf{Collapse:} This time we collapse the edges adjacent to the vertex labeled by $A_i$, equivariantly, 

\item \textbf{Expand:} Expand the edges adjacent to $A_i$ after choosing the vertex labeled by $uvA_pv^{-1}u^{-1}$ in the $A_i*vA_pv^{-1}$ minimal subtree as a replacement vertex for the vertex label $vA_pv^{-1}$ of the fundamental domain. The resulting tree is equivariantly homeomorphic to $f_{A_p}^w(\mathbf{X_1})$. 
\end{enumerate}

\end{enumerate}

Notice - $w$ is a word of length $2$. More generally, for any word $w \in A_i* A_j$ this proof can be extended by induction on the length of the word $w$, when $w$ is expressed as an alternating product of elements of $A_i$ and $A_j$. So, that concludes the proof of the part 1, where $\mathbf{X}$ and $\mathbf{X'}$ only differ at the vertex labeled by conjugate of $A_p$.

\item $\mathbf{X'}$ can be expressed as $f(\mathbf{X})$, where 
$
f = \displaystyle \prod_{\substack{w_p \in A_i*A_j\\p\ne i, j\\ p= p_1}}^{p_l}f_{A_p}^{w_p}
$ (definition \ref{defn:F^w_H}), such that
 $p_i \in \{1,.., n\} - \{i, j\}$. Hence, we can connect
\begin{itemize}[wide, labelwidth=!, labelindent=0pt] 
\item $\mathbf{X}$ to $f_{A_{p_1}}^{w_{p_1}}(\mathbf{X})$ via a path in $\mathscr{K}^{ij}$.
\item $f_{A_{p_1}}^{w_{p_1}}(\mathbf{X})$ to $f_{A_{p_2}}^{w_{p_2}}(\mathbf{X})$ via a path in $\mathscr{K}^{ij}$.

\vdots

\item $\displaystyle \prod_{\substack{w_p \in A_i*A_j\\p\ne i, j\\ p= p_1}}^{p_{l - 1}}f_{A_p}^{w_p}(\mathbf{X}$) to $\displaystyle \displaystyle \prod_{\substack{w_p \in A_i*A_j\\p\ne i, j\\ p= p_1}}^{p_l}f_{A_p}^{w_p}(\mathbf{X})= \mathbf{X'}$ via a path in $\mathscr{K}^{ij}$.

\end{itemize}

\end{enumerate}
\end{proof}

\begin{remark}
We will use the lemma \ref{lemma:same_fund} in our following discussion, which states that two different graphs of groups having same vertex groups can be connected by a path consisting of graphs of groups having same vertex groups in $\mathcal{SPD}$

\end{remark}

\begin{corollary}
$\mathscr{K}^{ij}$ is connected.
\end{corollary}
\begin{proof}

\begin{enumerate}[wide, labelwidth=!, labelindent=0pt]

\item By lemma \ref{lemma:same_fund}, we can connect any graph of groups in $\mathscr{K}^{ij}$  to a graph of groups of type $X$ via a path contained inside $\mathscr{K}^{ij}$. 

\item By lemma \ref{lemma:kconnected} we can connect any graph of groups of type $X$ inside $\mathscr{K}^{ij}$ to a graph of group of type $X$ whose non trivial vertex groups are $A_1,..., \text{ and }A_n$, via a path contained inside $\mathscr{K}^{ij}$. Both paths can be constructed so that they are entirely contained inside $\mathscr{K}^{ij}$.
\end{enumerate}\end{proof}

\begin{remark}
Recall definition \ref{defn:H^ij} from section \ref{section:thick},
$H^{ij} := \displaystyle \bigoplus_{\substack{k\ne i, j\\ k=1}}^n H_k^i * H_k^j $

\end{remark}

\begin{lemma}
$\mathscr{K}^{ij}$ is invariant under the action of the subgroup $H^{ij}$.
\end{lemma}

\begin{proof}
Let, $T \in \mathscr{K}^{{ij}^0}$ and $\phi \in H^{ij}$. Assume, that $ \Phi \in \mathsf{Aut}(G_n)$ be such that $\phi = [\Phi]$ and
\begin{gather*} 
\Phi(A_k) =  A_k \text{, when } k = i, j \text{, and } \\
\Phi(A_k) =  u_kA_ku_k^{-1}  (\text{where }k\ne i,j \text{ and } u_k \in A_i*A_j)
\end{gather*}

There is a fundamental domain of $T$, such that the non-trivial vertex stabilizers are given by $A_i, A_j$, and  $ w_kA_kw_k^{-1} (\text{where }k\ne i,j \text{ and } w_k \in A_i*A_j)$.\begin{gather*}\Phi(A_i) = A_i; \Phi(A_j) = A_j; \text{and } \\
\Phi(w_kA_kw_k^{-1}) = \Phi(w_k)\Phi(A_k)\Phi(w_k^{-1}) = w_k\Phi(A_k)w_k^{-1} = w_ku_kA_ku_k^{-1}w_k^{-1}\end{gather*}

So, $\phi(T) \in\mathscr{K}^{{ij}^0}$. If $e$ is any edge of length $1$ connecting two vertices of $\mathscr{K}^{ij}$, then $\phi(e)$ is also an edge of length $1$ as the $\Gamma_n$ action is isometric. Hence, it is in $\mathscr{K}^{ij}$.
\end{proof}

\begin{lemma}
$H^{ij} \acts \mathscr{K}^{ij}$ is properly discontinuous and co-compact.
\end{lemma}

\begin{proof}
There are only finitely many graphs of groups in $\mathscr{K}^{ij}$ (up-to homeomorphism) such that vertex groups are either trivial or $A_k$. We will show that any other graph of groups in $\mathscr{K}^{ij}$ is in the $H^{ij}$-orbit of a graph of groups described in the first line. This will prove co-compactness.

Let, $T \in \mathscr{K}^{ij} \cap \mathcal{SPD}^{0}(G,\mathscr{H})$ be a tree whose graph of groups is represented by $\mathbf{X}$ and the non trivial vertex groups are given by $A_i, A_j$ and $w_kA_kw_k^{-1}$ (where $k \ne i, j$ and $w_k \in A_i * A_j$). Consider $f_{A_k}^{w_k} \in \Gamma_n'$ (definition \ref{defn:F^w_H}), then $(f_{A_k}^{w_k})^{-1}(\mathbf{X})$ is a graph of groups described in the first line of the proof. So, there are finitely many orbits (up-to homeomorphism) of graphs of groups in $\mathscr{K}^{ij}$ under the action of $H^{ij}$. The sub-complex is locally finite. If we look at the Bass-Serre tree of any graph of groups, then there are finitely many fundamental domains containing the vertices labeled by $A_i$ and $A_j$. So, point stabilizer is finite. Hence, the action is properly discontinuous.
\end{proof}

\begin{lemma}\label{lemma:def_TIJ}
Fix distinct integers $i, j \in \{1,..., n\}$ and $w_k \in A_i*A_j$, where $k \in \{1, ..., n\}-\{i, j\}$, then the fundamental group of a graph of groups having non trivial vertex groups $A_i, A_j, w_kA_kw_k^{-1}, k \in \{1, ..., n\} - \{i, j\}$ is $\displaystyle \free_{l = 1}^n A_l $.
\end{lemma}
\begin{proof}
Consider the map

$
 \mathbb{A}: \displaystyle\bigcup_{l = 1}^n A_l \rightarrow \free_{l = 1}^n A_l \\
$
$
a \mapsto
\begin{cases}
a , & \text{if } a \in A_i \cup A_j \\
w_kaw_k^{-1}, & \text{if } a \in \displaystyle\bigcup_{\substack{l \ne i, j\\l = 1}}^n A_l
\end{cases}
$

By the universal property of the free products, this map can be uniquely extended to a homomorphism denoted by $\mathbb{A}: \displaystyle\free_{l = 1}^n A_l \rightarrow \free_{l = 1}^n A_l$ (abusing notation). We will define a homomorphism $\mathbb{A'}: \displaystyle\free_{l = 1}^n A_l \rightarrow \free_{l = 1}^n A_l$, such that $\mathbb{A} \circ \mathbb{A'} = \mathbb{A'} \circ \mathbb{A} = id$. 

The map $
 \mathbb{A'}: \displaystyle\bigcup_{l = 1}^n A_l \rightarrow \free_{l = 1}^n A_l \\
$
$
a \mapsto
\begin{cases}
a , & \text{if } a \in A_i \cup A_j \\
w_k^{-1}aw_k, & \text{if } a \in \displaystyle\bigcup_{\substack{l \ne i, j\\l=1}}^n A_l
\end{cases}
$

By the universal property of the free products, this map can be uniquely extended to a homomorphism denoted by $\mathbb{A'}: \displaystyle\free_{l = 1}^n A_l \rightarrow \free_{l = 1}^n A_l$ (abusing notation). Hence, $\mathbb{A}$ is an automorphism and the fundamental group of a graph of groups having non trivial vertex groups $A_i, A_j, w_kA_kw_k^{-1}, k \in \{1, ..., n\} - \{i, j\}$ is $\displaystyle \free_{l = 1}^n A_l $.
\end{proof}

The goal of our next definition is to assign a tree in $\mathscr{K}^{ij}$ for a given tree in $\mathcal{SPD}(G,\mathscr{H})$

\begin{definition}\label{defn:retraction}
Consider a tree $T \in \mathcal{SPD}$ and fix two distinct integers $i, j \in \{1, ..., n\}$. We will build a metric tree, $\overline{T}^{ij}$, using $T$ tree as follows: 
\begin{enumerate}[wide, labelwidth=!, labelindent=0pt]
\item Start with the $A_i*A_j$-minimal subtree in $T$ and call it $T^{ij}$.
\item If the nearest point projection to $T^{ij}$ of the vertex stabilized by $A_k, k \ne i, j$  is contained in the fundamental domain of $A_i * A_j \acts T^{ij}$ whose extremities are stabilized by the subgroups $w_kA_iw_k^{-1}$ and $w_kA_jw_k^{-1}$, then the nearest point projection of the vertex labeled by the subgroup $w_k^{-1}A_kw_k$ to $T^{ij}$ is contained in the fundamental domain labeled by the subgroups $A_i$ and $A_j$. If the nearest point projection of $A_k, k \ne i, j$ is part of more than one fundamental domains of $A_i*A_j \acts {T^{ij}}$, then choose the fundamental domain closest to the fundamental domain whose vertices are labeled by $A_i$ and $A_j$.
\item Construct a graph of groups, such that the underlying geometry is isometric to the geometry of the smallest subtree of $T$ containing the vertices labeled by the groups from the following set - $\{A_i, A_j, w_k^{-1}A_kw_k \vert k \in \{1,..., n\} - \{i, j\}\}$ and the corresponding non trivial vertex groups are $\{A_i, A_j, w_k^{-1}A_kw_k \vert k \in \{1,..., n\} - \{i, j\}\}$. $\overline{\mathbf{X}}^{ij}_{\mathcal{PD}}$ is the graph of groups homothetic to the above graph of groups such that the sum of edge lengths is 1. By lemma \ref{lemma:def_TIJ} it follows that $\overline{\mathbf{X}}^{ij}_\mathcal{PD}$ is an  element of $\mathcal{PD}$. Define $\overline{\mathbf{X}}^{ij}$ to be the image of $\overline{\mathbf{X}}^{ij}_\mathcal{PD}$ in $\mathcal{SPD}$ under the retraction stated in lemma \ref{lemma:bsd} and $\overline{{T^{ij}}}$ is the Bass-Serre tree of $\overline{\mathbf{X}}^{ij}$. 

\end{enumerate}
\end{definition}

Our next goal is to define a map which can be extended to a Lipschitz retraction.

\begin{definition} \label{defn:retraction_map}

Define a map \begin{gather*}L_{ij}: \mathcal{SPD}^0 (G,\mathscr{H}) \rightarrow \mathscr{K}^{ij} \\
T\mapsto \overline{T}^{ij}\end{gather*}

\end{definition}

\begin{lemma} \label{lemma:retract}
If $T_1, T_2 \in \mathcal{SPD} (G,\mathscr{H})$ are two distinct vertices which are distance $1$ apart, then $L_{ij}(T_1)$ and $L_{ij}(T_2)$ in $\mathscr{K}^{ij}$ is at most $2$. 
\end{lemma}

\begin{proof}
If $T_1$ and $T_2$ are distance $1$ apart, then without loss of generality we can assume that $T_2$ is obtained by collapsing one or more edge orbits of $T_1$, equivariantly. We will show that the simplex containing $L_{ij}(T_2)$ is a boundary to the simplex containing $L_{ij}(T_1)$. Hence, the distance is at most $2$ in $\mathcal{SPD}$. That is, a tree in the $G$-equivariant homeomorphism class of $L_{ij}(T_2)$ and a tree in the $G$-equivariant homeomorphism class of $L_{ij}(T_1)$ are related by a collapse move in $\mathcal{SPD}$. So, a tree in their corresponding $G$-equivariant homeomorphism classes are related by a collapse move in $\mathscr{K}^{ij}$ and hence the distance between  $L_{ij}(T_2)$ and $L_{ij}(T_1)$ in $\mathscr{K}^{ij}$ is at most $2$.

One of the following situations can occur when collapsing edges of $T_1$ to produce $T_2$- 
\begin{enumerate}[wide, labelwidth=!, labelindent=0pt]
\item None of the edges undergoing collapses are contained in the $A_i*A_j$-minimal subtree. In this case $L_{ij}(T_2)$ and  $L_{ij}(T_1)$ have a fundamental domain with identical non-trivial vertex stabilizer subgroups. Hence, the distance between them is at most $2$, by lemma \ref{lemma:same_fund}.
\item One or more collapsed edge-orbits are contained in the $A_i*A_j$-minimal subtree. Then $L_{ij}(T_2)$ and $L_{ij}(T_1)$ may have different fundamental domain (explained in the picture below). For example consider the case where, the nearest point projection of $A_k, (k \ne i, j)$ onto the $A_i*A_j$-minimal subtree in $T_1$ is contained in the fundamental domain of $A_i*A_j \acts T^{ij}$ labeled by $w_kA_iw_k^{-1}$ and $w_kA_jw_k^{-1}$; and the nearest point projection of $A_k$ onto the $A_i*A_j$-minimal subtree in $T_2$ is contained in the fundamental domain of $A_i*A_j \acts T^{ij}$ labeled by $w_k'A_iw_k'^{-1}$ and $w_k'A_jw_k'^{-1}$. Now, $L_{ij}(T_2)$, contains a fundamental domain some of whose vertices are labeled by the subgroups $A_i, A_j, w_k'^{-1}A_kw_k'$ using the definition of the map $L_{ij}$. On the other hand, a collapse move on $L_{ij}(T_1)$ produces a tree with a fundamental domain some of whose vertices are labeled by  the subgroups $A_i, A_j, w_k^{-1}A_kw_k$. However, due to the rigidity (up-to $G$ equivariant homeomorphism) of the $A_i*A_j$ minimal subtree we have $w_k'^{-1}w_k \in A_j$. Hence, these two trees are $G$-equivariantly isometric. So, a tree in the $G$-equivariant homeomorphism class of $L_{ij}(T_2)$ and a tree in the $G$-equivariant homeomorphism class of $L_{ij}(T_1)$ are related by a collapse move in $\mathcal{SPD}$.
\end{enumerate}

\begin{center}
\begin{tikzpicture} \begin{scriptsize}
\draw (0, 0) -- (0, -2);
\draw (2, 0) -- (2, -2);
\draw (7, 0) -- (7, -2);
\draw (5, 0) -- (5, -2);

\draw [fill=black] (-1, 0) circle (1.2pt) node[anchor = south]{$w_kA_iw_k^{-1}$};
\draw [fill=black] (1, 0) circle (1.2pt)node[anchor = south, rotate = 90]{$w_kA_jw_k{-1}$};
\draw [fill=black] (1, 0) circle (1.2pt)node[anchor = north, rotate  = 90]{$ = w_k'A_jw_k'^{-1}$};

\draw [fill=black] (3, 0) circle (1.2pt)node[anchor = north]{$w_k'A_iw_k'^{-1}$};
\draw [fill=black] (4, 0) circle (1.2pt)node[anchor = south]{$a_jA_ia_j^{-1}$};
\draw [fill=black] (6, 0) circle (1.2pt)node[anchor = south]{$A_j = a_jA_ja_j^{-1}$};
\draw [fill=black] (8, 0) circle (1.2pt)node[anchor = south]{$A_i$};

\draw [fill=black] (0, -2) circle (1.2pt) node[anchor = north]{$A_k$};
\draw [fill=black] (2, -2) circle (1.2pt) node[anchor = north]{$a_j^{-1}A_ka_j$};
\draw [fill=black] (5, -2) circle (1.2pt) node[anchor = north]{$w_k'^{-1}A_kw_k'$};
\draw [fill=black] (7, -2) circle (1.2pt) node[anchor = north]{$w_k^{-1}A_kw_k$};

\draw (-1,0) -- node[fill=white,inner sep=-1.25pt,outer sep=0,anchor=center]{$\hdots \hdots$} (8,0);
\draw (0,-4) -- node[fill=white,inner sep=-1.25pt,outer sep=0,anchor=center]{$\hdots \hdots$} (6,-4);

\draw [fill=black] (0, -4) circle (1.2pt) node[anchor = north]{$w_kA_iw_k^{-1}$};
\draw [fill=black] (1, -4) circle (1.2pt)node[anchor = south]{$w_k'A_jw_k'{-1}$};
\draw [fill=black] (2, -4) circle (1.2pt)node[anchor = north west]{$w_k'A_iw_k'^{-1}$};

\draw [fill=black] (4, -4) circle (1.2pt)node[anchor = south]{$a_jA_ia_j^{-1}$};
\draw [fill=black] (5, -4) circle (1.2pt)node[anchor = south]{$A_j$};
\draw [fill=black] (6, -4) circle (1.2pt)node[anchor = south]{$A_i$};

\draw [fill=black] (0, -5) circle (1.2pt) node[anchor = north]{$A_k$};
\draw [fill=black] (2, -5) circle (1.2pt) node[anchor = north]{$a_j^{-1}A_ka_j$};
\draw [fill=black] (4, -5) circle (1.2pt) node[anchor = north]{$w_k'^{-1}A_kw_k'$};
\draw [fill=black] (6, -5) circle (1.2pt) node[anchor = north]{$w_k^{-1}A_kw_k$};

\draw (1, -4) -- (0, -5);
\draw (1, -4) -- (2, -5);
\draw (5, -4) -- (4, -5);
\draw (5, -4) -- (6, -5);

\draw node at (3, -3) [anchor=south]{$T_1$};
\draw node at (3, -6) [anchor=south]{$T_2$};

\end{scriptsize} \end{tikzpicture} 
\end{center}
\end{proof}

\begin{corollary}
The map $L_{ij}$ from definition \ref{defn:retraction_map} can be extended to a continuous Lipschitz retraction $L_{ij}: \mathcal{SPD}^1 (G,\mathscr{H}) \rightarrow \mathscr{K}^{ij}$ \end{corollary}

\begin{proof}
We will extend the map linearly on each edge of $\mathcal{SPD}^{1}$. Lemma \ref{lemma:retract} implies the map is $1$-Lipschitz. Definition \ref{defn:retraction_map} implies the map is a retract. 
\end{proof}

\begin{corollary}\label{cor:hij_undistorted}
$H^{ij}$ is an undistorted subgroup of $\Gamma_n$
\end{corollary}


\subsection{A second class of undistorted subgroups of $\Gamma_n'$} \label{subsec:undistorted2}

In this section we will find a class of subgroups $N^{ij}$ (here, $i \ne j \in \{1, ..., n\}$) of $\Gamma_n'$ which satisfy the following properties:

\begin{enumerate}[wide, labelwidth=!, labelindent=0pt]

\item $N^{ij} < H^{ij}$
\item $N^{ij} \cong \mathbb{Z}$
\item $\langle N^{i_1i_2}, N^{i_3i_4} \rangle \cong \mathbb{Z} \oplus \mathbb{Z}$,
	where $i_1, i_2, i_3, i_4$ are distinct integers from the set $\{1,..., n\}$.
\item $\langle N^{i_1i_2}, N^{i_3i_4} \rangle$ is undistorted in $\Gamma_n$
\end{enumerate}

We will define $N^{ij}$ next. 

\begin{definition}\label{defn:nij}
Fix two distinct integers $i, j \in \{1,..., n\}$. For $p \in \{i, j\} $ fix $x_p \in A_p/\{id_{A_p}\}.\\
f^{ij} := \displaystyle\prod_{\substack{p \ne i, j\\p=1}}^n f_{A_p}^{x_i} f_{A_p}^{x_j} = \prod_{\substack{p \ne i, j\\p=1}}^n f_{A_p}^{x_jx_i}  \in H^{ij}, N^{ij} := \left\langle f^{ij} \right\rangle$

\end{definition}

\begin{lemma}\label{lemma:NijZ}
Consider distinct integers $i, j \in \{1,..., n\}$, then $N^{ij} \cong \mathbb{Z}$.
\end{lemma}

\begin{proof}
We will prove that $\left(f^{ij}\right)^m = id \implies m = 0$. Consider a graph of groups $\mathbf{X}$ such that the underlying graph has $1$ vertex of valence $n$ and $n$ vertices of valence $1$; and the vertex groups are $\{A_1, A_2, ..., A_n\}$.

By definition of $f^{ij}$, there is a representation of $(f^{ij})^m(\mathbf{X})$, such that the vertex groups are given by $\{A_i, A_j, (x_jx_i)^mA_k(x_jx_i)^{-m}\vert k \ne i, j \}$. 

Let us fix a $k \ne i, j$. In the Bass-Serre tree of $\mathbf{X}$, the distance between the vertex labeled by $A_i$ and $A_k$ is $2$. However, in the Bass-Serre tree of $(f^{ij})^m(\mathbf{X})$ the distance between the vertex labeled by $A_i$ and $A_k$ is $4m + 2$. So, $m \ne 0 \implies (f^{ij})^m \ne id$. Hence, $\la f^{ij} \ra = N^{ij} = \mathbb{Z}$.
\end{proof}
\begin{lemma}
If $i_1, i_2, i_3, i_4 \in \{1, ..., n\}$ are distinct integers, 
then $f^{i_1i_2}$ commutes with $f^{i_3i_4}$.
\end{lemma}

\begin{proof} 
If an automorphism conjugates every element of the group $G_n$ by a fixed element, then the automorphism represents the outer class of the identity automorphism. So,
\begin{gather*}\prod_{k = 1}^n f_{A_k}^{x_i} = id_{\Gamma_n} 
\implies \prod_{\substack{k \ne i, j\\k=1}}^n f_{A_k}^{x_i} = 
\left(f_{A_i}^{x_i}\right)^{-1}\left(f_{A_j}^{x_i}\right)^{-1}\\
\implies f^{ij} =\prod_{\substack{k \ne i, j\\k = 1}}^n f_{A_k}^{x_i} f_{A_k}^{x_j}
=  \left(\prod_{\substack{k \ne i, j\\k = 1}}^n f_{A_k}^{x_i} \right) \left(\prod_{\substack{k \ne i, j\\k = 1}}^n f_{A_k}^{x_j}\right)\\
=  \left(f_{A_i}^{x_i}\right)^{-1}  \left(f_{A_j}^{x_{i}}\right)^{-1}  \left(f_{A_i}^{x_j}\right)^{-1}  \left(f_{A_j}^{x_j}\right)^{-1}
\end{gather*}

If $i_1, i_2, i_3, \text{ and } i_4$ are all distinct numbers, then using an argument similar to the one used in proving lemma \ref{lemma:diff_base_commute} we see that, $f^{i_1i_2}$ and $f^{i_3i_4}$ commute.

\begin{notation}
Consider the graph of groups, $\mathbf{X}$, from notation \ref{not:X}, then $\left(f^{i_1i_2}\right)^m \left(f^{i_3i_4}\right)^l(\mathbf{X})$ can be represented by the following graph of groups:

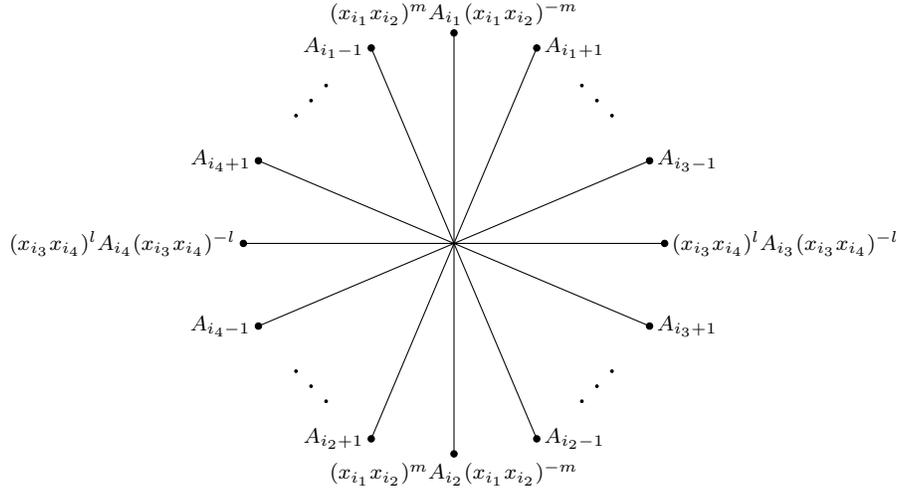
\begin{center}

\begin{tikzpicture} \begin{scriptsize}
\draw (0, 2.8) -- (0, -2.8);
\draw(-2.8, 0) -- (2.8, 0);
\draw(-2.6, 1.1) -- (2.6, -1.1);
\draw(2.6, 1.1) -- (-2.6, -1.1);
\draw(1.1, -2.6) -- (-1.1, 2.6);
\draw(1.1, 2.6) -- (-1.1, -2.6);

\draw [fill=black] (-1.1, 2.6) circle (1.2pt) node[anchor=east]{$A_{{i_1} - 1}$};
\draw [fill=black] (0,2.8) circle (1.2pt) node[anchor=south]{$(x_{i_1}x_{i_2})^mA_{i_1}(x_{i_1}x_{i_2})^{-m}$};
\draw [fill=black] (1.1, 2.6) circle (1.2pt) node[anchor=west]{$A_{{i_1} + 1}$};

\draw [fill=black] (2.6, 1.1) circle (1.2pt) node[anchor=west]{$A_{{i_3} - 1}$};
\draw [fill=black] (2.8, 0) circle (1.2pt) node[anchor=west]{$(x_{i_3}x_{i_4})^lA_{i_3}(x_{i_3}x_{i_4})^{-l}$};
\draw [fill=black] (2.6, -1.1) circle (1.2pt) node[anchor=west]{$A_{{i_3} + 1}$};

\draw [fill=black] (1.1, -2.6) circle (1.2pt) node[anchor=west]{$A_{{i_2} - 1}$};
\draw [fill=black] (0, -2.8) circle (1.2pt) node[anchor=north]{$(x_{i_1}x_{i_2})^mA_{i_2}(x_{i_1}x_{i_2})^{-m}$};
\draw [fill=black] (-1.1, -2.6) circle (1.2pt) node[anchor=east]{$A_{{i_2} + 1}$};

\draw [fill=black] (-2.6, -1.1) circle (1.2pt) node[anchor=east]{$A_{{i_4} - 1}$};
\draw [fill=black] (-2.8, 0) circle (1.2pt) node[anchor=east]{$(x_{i_3}x_{i_4})^lA_{i_4}(x_{i_3}x_{i_4})^{-l}$};
\draw [fill=black] (-2.6, 1.1) circle (1.2pt) node[anchor=east]{$A_{{i_4} + 1}$};

\draw [fill=black] (1.7, 2.1) circle (.5pt);
\draw [fill=black] (1.9, 1.9) circle (.5pt);
\draw [fill=black] (2.1, 1.7) circle (.5pt);
\draw [fill=black] (1.7, -2.1) circle (.5pt);
\draw [fill=black] (1.9, -1.9) circle (.5pt);
\draw [fill=black] (2.1, -1.7) circle (.5pt);
\draw [fill=black] (-1.7, -2.1) circle (.5pt);
\draw [fill=black] (-1.9, -1.9) circle (.5pt);
\draw [fill=black] (-2.1, -1.7) circle (.5pt);
\draw [fill=black] (-1.7, 2.1) circle (.5pt);
\draw [fill=black] (-1.9, 1.9) circle (.5pt);
\draw [fill=black] (-2.1, 1.7) circle (.5pt);
\end{scriptsize} \end{tikzpicture}

\end{center}
\end{notation}
\end{proof}

\begin{lemma}
If $i_1, i_2, i_3, i_4 \in \{1, ..., n\}$ are distinct numbers and 
$\left(f^{i_1i_2}\right)^m = \left(f^{i_3i_4}\right)^l$, then $m = l = 0$.
\end{lemma}

\begin{proof}
Recall that we have denoted a graph of groups with underlying graph having $1$ vertex of valence $n$ and $n$ vertices of valence $1$ as a graph of groups of type $X$. Let, $\mathbf{X}$ be a graph of groups of type $X$ such that the non trivial vertex groups are $\{A_1,..., A_n\}$.
Since, $i_1, i_2, i_3, i_4 \in \{1, ..., n\}$ are distinct numbers, without loss of generality assume that $1<i_1<i_2<i_3<i_4<n$. So, $x_{i_1}, x_{i_2}, x_{i_3}$, and $x_{i_4}$ are all distinct elements of the group $G_n$. Also, recall the identity $f^{ij} = \left(f_{A_i}^{x_i}\right)^{-1}  \left(f_{A_j}^{x_{i}}\right)^{-1}  \left(f_{A_i}^{x_j}\right)^{-1}  \left(f_{A_j}^{x_j}\right)^{-1}$.

Then
\begin{enumerate}[wide, labelwidth=!, labelindent=0pt]
\item $(f^{i_1i_2})^m(\mathbf{X})$ is a graph of groups of type $X$ with vertex groups \\
$\{A_1,...,\
A_{i_1-1}, 
(x_{i_2}x_{i_1})^{-m}A_{i_1}(x_{i_2}x_{i_1})^m,\ 
A_{i_1+1},...,\
A_{i_2-1},\\
(x_{i_2}x_{i_1})^{-m}A_{i_2}(x_{i_2}x_{i_1})^m,
A_{i_2+1},...,\
A_n\}$

\item $(f^{i_3i_4})^l(\mathbf{X})$ is a graph of groups of type $X$ with vertex groups \\
$\{A_1,...,\
A_{i_3-1}, 
(x_{i_4}x_{i_3})^{-l}A_{i_3}(x_{i_4}x_{i_3})^l,\ 
A_{i_3+1},...,\
A_{i_4-1},\\
(x_{i_4}x_{i_3})^{-l}A_{i_4}(x_{i_4}x_{i_3})^l,
A_{i_4+1},...,\
A_n\}$
\end{enumerate}
 We will show that $(f^{i_1i_2})^m \ne (f^{i_3i_4})^l$ in $\Gamma_n$ by showing that  $(f^{i_1i_2})^m(T_\mathbf{X}) \ne (f^{i_3i_4})^l(T_\mathbf{X})$ in $\mathcal{SPD}$, where $T_\mathbf{X}$ is the Bass-Serre tree of $\mathbf{X}$.  
 
The vertex labeled by $A_{i_1}$ is at a distance of $2$ from the vertex labeled by $A_1$ in $ (f^{i_3i_4})^l(T_\mathbf{X})$; whereas the vertex labeled by $(x_{i_2}x_{i_1})^{-m}A_{i_1}(x_{i_2}x_{i_1})^m$ is at a distance $2$ from the vertex labeled by $A_1$ in $(f^{i_1i_2})^m(T_\mathbf{X})$. By uniqueness of $A_{i_1}*A_{i_2}$-minimal subtree  the vertex labeled by $A_{i_1}$ cannot be at a distance $2$ from the vertex labeled by $A_1$ in $(f^{i_1i_2})^l(T_\mathbf{X})$. Hence, $(f^{i_1i_2})^m = (f^{i_3i_4})^l \implies m = 0$. Similarly, The vertex labeled by $A_{i_3}$ is at a distance of $2$ from the vertex labeled by $A_3$ in $ (f^{i_1i_2})^m(T_\mathbf{X})$; whereas the vertex labeled by $(x_{i_4}x_{i_3})^{-l}A_{i_3}(x_{i_4}x_{i_3})^l$ is at a distance $2$ from the vertex labeled by $A_3$ in $(f^{i_3i_4})^l(T_\mathbf{X})$. By uniqueness of $A_{i_3}*A_{i_4}$-minimal subtree the vertex labeled by $A_{i_3}$ cannot be at a distance $2$ from the vertex labeled by $A_3$ in $(f^{i_3i_4})^l(T_\mathbf{X})$. Hence, $(f^{i_1i_2})^m = (f^{i_3i_4})^l \implies l = 0$
 \end{proof}

\begin{corollary} \label{cor:nij_zerothick}
$\langle N^{i_1i_2}, N^{i_3i_4} \rangle \cong \mathbb{Z} \oplus \mathbb{Z}$,
	where $i_1, i_2, i_3, i_4$ are all different integers.
\end{corollary}

Our next goal is to prove that the distance between $\mathbf{X}$ and $\left(f^{i_1i_2}\right)^m \left(f^{i_3i_4}\right)^l(\mathbf{X})$ is at least $2(m + l)$ in $\mathcal{SPD}$.

Consider a non trivial vertex stabilizer subgroup $H \in \mathscr{H}$. We will define a function $g_{H}^{i_1i_2}$ from the $0$-skeleton of $\mathcal{SPD}$ to the real numbers. For a given tree $T \in \mathcal{SPD}$, the function will count the number of vertices labeled by conjugates of $A_{i_1}$ and $A_{i_2}$ on the $x_{i_1}x_{i_2}$-axis between two points on $T$ as described in the following definition.

\begin{definition}
Consider $T \in \mathcal{SPD}$. For $H \in \mathscr{H}$ let $g_H^{i_1i_2}(T)$ be the number of vertices labeled by subgroups of $A_{i_1}\free A_{i_2}$ which are conjugates of $A_{i_1}$ and $A_{i_2}$ on the $x_{i_1}x_{i_2}$-axis of $T$ between the following two points. \begin{enumerate}[wide, labelwidth=!, labelindent=0pt]
\item The closest point to the $x_{i_1}x_{i_2}$-axis in $T$ from a vertex labeled by the subgroup $H$. 
\item The vertex labeled by $(x_{i_2}x_{i_1})^mA_{i_2}(x_{i_2}x_{i_1})^{-m} $ on $T$.
\end{enumerate}\begin{gather*} g_{H}^{i_1i_2}: \mathcal{SPD}^{0}(G, \mathscr{H}) \rightarrow \mathbb{R}\\
 T \mapsto g_H^{i_1i_2}(T) \end{gather*} 
\end{definition}

\begin{lemma}
If $d_{\mathcal{SPD}}(T_1, T_2) = 1 \text{ and } k \ne i_1, i_2 \text{ then } \vert g_{A_k}^{i_1i_2}(T_1) - g_{A_k}^{i_1i_2}(T_2)\vert \leq 1$.
\end{lemma}

\begin{proof}

The idea of the proof is derived from the knowledge of uniqueness (up to $A_{i_1}*A_{i_2}$ equivariant homeomorphism of topological space) of $A_{i_1}*A_{i_2}$-minimal subtree inside every tree of $\mathcal{SPD}$, lemma \ref{lemma:unique_min}.

Without loss of generality, let us assume that 

\begin{enumerate}[wide, labelwidth=!, labelindent=0pt]

\item $T_2$ is obtained from $T_1$ by applying a series of collapse moves on its edge orbits.
\item $v_{T_j}$ is the vertex on the $x_{i_1}x_{i_2}$-axis of $T_j$ closest to the vertex labeled by $A_k$, where $j \in \{1, 2\}$.
\item The vertex whose stabilizer subgroup is a conjugate of $A_{i_j}$ and is closest to $v_{T_1}$ on the $x_{i_1}x_{i_2}$-axis of $T_1$ is labeled by $(x_{i_2}x_{i_1})^{s}A_{i_j}(x_{i_2}x_{i_1})^{-{s}}$ (or $(x_{i_1}x_{i_2})^{s}A_{i_j}(x_{i_1}x_{i_2})^{-{s}}$). Here $j \in \{1, 2\}$. 

\end{enumerate}

 If $v_{T_1}$ is part of two different fundamental domains of the $x_{i_1}x_{i_2}$-axis, then we choose the fundamental domain closer to $(x_{i_2}x_{i_1})^mA_{i_2}(x_{i_2}x_{i_1})^{-m}$ and its vertex labeling.

Label the vertex whose stabilizer subgroup is a conjugate of $A_{i_j}$ and is closest to $v_{T_2}$ on the $x_{i_1}x_{i_2}$-axis of $T_2$ by $(x_{i_2}x_{i_1})^{r}A_{i_j}(x_{i_2}x_{i_1})^{-r}$. Here $j \in \{1, 2\}$.

If we get $T_2$ by equivariantly collapsing edges of $T_1$, then
$\vert r - s\vert \leq 1$

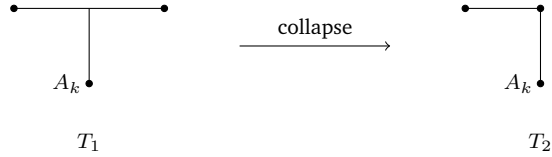
\begin{center}
\begin{tikzpicture} \begin{scriptsize}
\draw (-1,0) -- (1, 0);
\draw (0, 0) -- (0, -1);

\draw (5, 0) -- (6, 0);
\draw (6, 0) -- (6, -1);

\draw [fill=black] (-1, 0) circle (1.2pt);
\draw [fill=black] (1, 0) circle (1.2pt);
\draw [fill=black] (0, -1) circle (1.2pt) node[anchor = east]{$A_k$};

\draw [fill=black] (6, -1) circle (1.2pt) node[anchor=east]{$A_k$};
\draw [fill=black] (6, 0) circle (1.2pt);
\draw [fill=black] (5, 0) circle (1.2pt);

\draw [<-] (4, -.5) to (2, -.5);
\draw node at (3, -.5) [anchor=south]{collapse};
\draw node at (0, -2) [anchor=south]{$T_1$};
\draw node at (6, -2) [anchor=south]{$T_2$};

\end{scriptsize} \end{tikzpicture} 
\end{center}

Hence, $ \vert g_{A_k}^{i_1i_2}(T_1) - g_{A_k}^{i_1i_2}(T_2)\vert  \leq 1$.
\end{proof}

\begin{corollary}
$g_{A_k}^{i_1i_2}$ can be continuously extended to all of $\mathcal{SPD}$, so that it is a Lipschitz map.
\end{corollary}

\begin{proof}
This is a result of the definition of a simplicial complex. Any point in a simplicial complex, which is not in the $0$-skeleton, is in the interior of a unique simplex. Any point in $\mathcal{SPD}$, which is not in the $0$-skeleton can be expressed as a linear combination of the points in the $0$-skeleton of the simplex containing them. Hence, we can extend $g_{A_k}^{i_1i_2}$ linearly, and the resulting extension is Lipschitz.
\end{proof}

\begin{definition}
We will abuse notation to denote the extension of $g_{H}^{i_1i_2}$ to all of $\mathcal{SPD}$ by
$ g_{H}^{i_1i_2}: \mathcal{SPD}(G, \mathscr{H}) \rightarrow \mathbb{R}$.
\end{definition}

\begin{lemma}
$g_{A_k}^{i_1i_2}(T_\mathbf{X}) = 2m$ and $g_{A_k}^{i_1i_2}(\left(f^{i_1i_2}\right)^m(T_\mathbf{X})) = 0$.
\end{lemma}

\begin{proof}
When $k \notin\{i_1, i_2\}$, the vertices labeled by conjugates of $A_{i_1}$ and $A_{i_2}$ on the $x_{i_1}x_{i_2}$-axis in $T_\mathbf{X}$ between the vertex labeled by $A_k$ and $(x_{i_2}x_{i_1})^mA_{i_2}(x_{i_2}x_{i_1})^{-m}$ are listed below in order of increasing distance:
\begin{enumerate}[wide, labelwidth=!, labelindent=0pt]
\item[(1)] $A_2$
\item[(2)] $x_2A_1x_2^{-1} (= (x_2x_1)A_1(x_2x_1)^{-1})$
\item[(3)] $(x_2x_1)A_2(x_2x_1)^{-1} (= (x_2x_1)A_2(x_2x_1)^{-1})$

\vdots
\item[(2m)] $(x_2x_1)^mA_1(x_2x_1)^{-m}$
\end{enumerate}

So, $g_{A_k}^{i_1i_2}(T_\mathbf{X}) = 2m$. When $k \notin\{i_1, i_2\}$, there are no vertices on the tree $(\left(f^{i_1i_2}\right)^m(T_\mathbf{X}))$, with non trivial stabilizer between the vertex labeled by $A_k$ and  $(x_{i_2}x_{i_1})^mA_{i_2}(x_{i_2}x_{i_1})^{-m}$ . So, 
$(\left(f^{i_1i_2}\right)^m(T_\mathbf{X})) = 0$.
\end{proof}

\begin{lemma}
$d_{\mathcal{SPD}}(T_{\mathbf{X}}, \left(f^{i_1i_2}\right)^m(T_{\mathbf{X}})) \geq 2m$.
\end{lemma}

\begin{proof}
By intermediate value theorem for metric spaces, the image of the path from $T_{\mathbf{X}}$ to $\left(f^{i_1i_2}\right)^m(T_{\mathbf{X}})$ under the $1$-Lipschitz map $g_{A_k}^{i_1i_2}$ contains the interval $[0, 2m].$ Hence, $d_{\mathcal{SPD}}(T_{\mathbf{X}}, \left(f^{i_1i_2}\right)^m(T_{\mathbf{X}})) \geq 2m.$
\end{proof}

Next we want to prove a similar result about $d_{\mathcal{SPD}}(T_{\mathbf{X}}, \left(f^{i_1i_2}\right)^m\left(f^{i_3i_4}\right)^l(T_{\mathbf{X}}))$.

\begin{lemma}

$d_{\mathcal{SPD}}(T_{\mathbf{X}}, \left(f^{i_1i_2}\right)^m\left(f^{i_3i_4}\right)^l(T_{\mathbf{X}})) \geq 2\max(m , l) \geq m + l$.
\end{lemma}

\begin{proof}

Consider $k \notin \{i_1, i_2, i_3, i_4 \}$. Then, 
\begin{gather*}g_{A_k}^{i_1i_2}(T_\mathbf{X}) = 0\\
g_{A_k}^{i_1i_2}( \left(f^{i_1i_2}\right)^m\left(f^{i_3i_4}\right)^l(T_{\mathbf{X}})) = m\\
\implies d_{\mathcal{SPD}}(T_{\mathbf{X}}, \left(f^{i_1i_2}\right)^m\left(f^{i_3i_4}\right)^l(T_{\mathbf{X}})) \geq 2m.\\
\text{Similarly, we can show that }\\
d_{\mathcal{SPD}}(T_{\mathbf{X}}, \left(f^{i_1i_2}\right)^m\left(f^{i_3i_4}\right)^l(T_{\mathbf{X}})) \geq 2l.\\
d_{\mathcal{SPD}}(T_{\mathbf{X}}, \left(f^{i_1i_2}\right)^m\left(f^{i_3i_4}\right)^l(T_{\mathbf{X}})) \geq 2\max(m , l).
\end{gather*}
\end{proof}

\begin{corollary} \label{cor:nij_undistorted}
$\langle N^{i_1i_2}, N^{i_3i_4} \rangle := \langle f^{i_1i_2}, f^{i_3i_4} \rangle $ is quasi isometrically embedded in $\Gamma_n'$, when $i_1, i_2, i_3, i_4 \in \{1,..., n\}$ are distinct integers.

\end{corollary}

\begin{proof}
$\langle N^{i_1i_2}, N^{i_3i_4} \rangle \cong \mathbb{Z} \oplus \mathbb{Z}$, is generated by $f^{i_1i_2}, \text{and } f^{i_3i_4}$.

Consider, $g:= \left(f^{i_1i_2}\right)^m \left(f^{i_3i_4}\right)^l \in \langle N^{i_1i_2}, N^{i_3i_4} \rangle.$
As $\mathcal{SPD}$ acts geometrically on $\Gamma_n'$, we have, $\Vert g \Vert_{\Gamma_n'} \approx d_{\mathcal{SPD}}(T_{\mathbf{X}}, g(T_{\mathbf{X}})) \geq m + l = \Vert g \Vert_{\langle N^{i_1i_2}, N^{i_3i_4} \rangle}.$So, $\langle N^{i_1i_2}, N^{i_3i_4} \rangle := \langle f^{i_1i_2}, f^{i_3i_4} \rangle $ is quasi isometrically embedded in $\Gamma_n'$
\end{proof}

\subsection{An undistorted subgroup of $\mathsf{Out}(A_1*A_2*A_3*A_4)$}\label{subsec:M4_nondistortion}

Let us recall definition \ref{defn:M1234}, of $M^{12} \oplus M^{34} \leq {\Omega_4}$, where\\
 $M^{12} := \la \h{1}{1}, \h{2}{1}, \h{1}{2}, \h{2}{2}\ra; M^{34} := \la \h{3}{3}, \h{4}{3}, \h{3}{4}, \h{4}{4}\ra$. To show that $M^{12} \oplus M^{34}$ is undistorted in $\mathsf{Out}(A_1*A_2*A_3*A_4)$ we will
\begin{enumerate}[wide, labelwidth=!, labelindent=0pt]
\item Define , $\mathcal{M}_4$, an $M^{12} \oplus M^{34}$ invariant, connected sub-complex of $\mathcal{SPD}(G_4,\mathscr{H})$ such that the action of $M^{12} \oplus M^{34}$ on $\mathcal{M}_4$ is co-compact; and
\item Show that there is a Lipschitz retraction from $\mathcal{SPD}(G_4,\mathscr{H}) \mapsto \mathcal{M}_4$.
\end{enumerate}

\begin{definition}\label{defn:subcomplexM4}
$\mathcal{M}_4$ is the sub-complex of $\mathcal{SPD}(G_4,\mathscr{H})$ spanned by $0$-simplices ($G_4$-trees) of following type- $T \in \mathcal{M}_4^0 \iff \exists$ a fundamental domain, $F$, of $T$ such that the vertices corresponding to the non-trivial vertex stabilizers of $F$ are part of $A_1*A_2$-minimal subtree and $A_3*A_4$-minimal subtree. 
\end{definition}

\begin{example}
An example of a graph of groups corresponding to a vertex of $\mathcal{M}_4$ is:

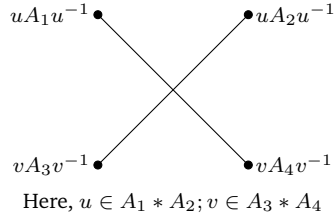
\begin{center}
\begin{tikzpicture} \begin{scriptsize}
\begin{scriptsize}
\Tw{5}{0}{uA_1u^{-1}}{uA_2u^{-1}}{vA_3v^{-1}}{vA_4v^{-1}}{ Here, $u \in A_1*A_2; v \in A_3*A_4$} 
\end{scriptsize}
\end{scriptsize} \end{tikzpicture} 
\end{center}

\end{example}

\begin{lemma}
$\mathcal{M}_4$ is an $M^{12} \oplus M^{34}$ invariant sub-complex of $\mathcal{SPD}(G_4,\mathscr{H})$.
\end{lemma}

\begin{proof}
Let $\mathbf{X'} \in \mathcal{M}_4$ be a graph of groups. Without loss of generality, assume that $\mathbf{X'}$ is given by -
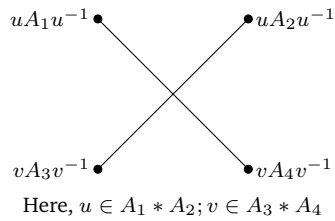
\begin{center}
\begin{tikzpicture} \begin{scriptsize}
\begin{scriptsize}
\Tw{5}{0}{uA_1u^{-1}}{uA_2u^{-1}}{vA_3v^{-1}}{vA_4v^{-1}}{ Here, $u \in A_1*A_2; v \in A_3*A_4$} 
\end{scriptsize}
\end{scriptsize} \end{tikzpicture} 
\end{center}

So, the vertex group
\begin{enumerate}[wide, labelwidth=!, labelindent=0pt]
\item conjugate to $A_1$ is given by $uA_1u^{-1}$
\item conjugate to $A_2$ is given by $uA_2u^{-1}$
\item conjugate to $A_3$ is given by $vA_3v^{-1}$
\item conjugate to $A_4$ is given by $vA_4v^{-1}$ 
\end{enumerate}
Here, $u \in A_1*A_2$ and $v \in A_3*A_4$. If $f \in M^{12} \oplus M^{34}$, then
$ f(u) \in A_1*A_2; f(v) \in A_3*A_4$. So, $f$ maps the $A_1*A_2$-minimal subtree of $T_{\mathbf{X'}}$ to the $A_1*A_2$-minimal subtree of $f(T_{\mathbf{X'}})$. Similarly,  $f$ maps the $A_3*A_4$-minimal subtree of $T_{\mathbf{X'}}$ to the $A_3*A_4$-minimal subtree of $f(T_{\mathbf{X'}})$. By, uniqueness of $A_1*A_2$-minimal subtree and $A_3*A_4$-minimal subtree inside $f(T_\mathbf{X'})$, we can represent $f(\mathbf{X'})$ by the following graph of groups -
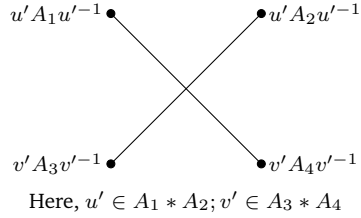
\begin{center}
\begin{tikzpicture} \begin{scriptsize}
\begin{scriptsize}
\Tw{5}{0}{u'A_1u'^{-1}}{u'A_2u'^{-1}}{v'A_3v'^{-1}}{v'A_4v'^{-1}}{ Here, $u' \in A_1*A_2; v' \in A_3*A_4$} 
\end{scriptsize}
\end{scriptsize} \end{tikzpicture} 
\end{center}

Hence, $f(\mathbf{X'}) \in \mathcal{M}_4$ and $\mathcal{M}_4$ is $M^{12} \oplus M^{34}$ invariant sub-complex of $\mathcal{SPD}(G_4,\mathscr{H})$.
\end{proof}

\begin{lemma}
$\mathcal{M}_4$ is a connected sub-complex of $\mathcal{SPD}(G_4,\mathscr{H})$.
\end{lemma}
\begin{proof}
We will prove this by induction. Let, $\mathbf{X'} \in \mathcal{M}_4$ be given by -
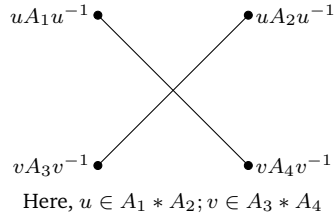
\begin{center}
\begin{tikzpicture} \begin{scriptsize}
\begin{scriptsize}
\Tw{5}{0}{uA_1u^{-1}}{uA_2u^{-1}}{vA_3v^{-1}}{vA_4v^{-1}}{ Here, $u \in A_1*A_2; v \in A_3*A_4$} 
\end{scriptsize}
\end{scriptsize} \end{tikzpicture} 
\end{center}
We will show that $\mathbf{X'}$ is connected in $\mathcal{M}_4$ to the graph of groups, $\mathbf{X''}$, given by - 
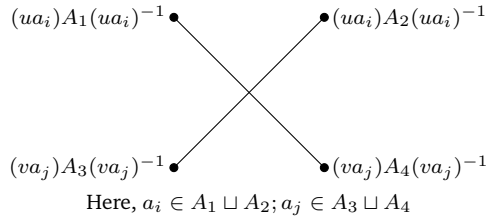
\begin{center}
\begin{tikzpicture} \begin{scriptsize}
\begin{scriptsize}
\Tw{5}{0}{(ua_i)A_1(ua_i)^{-1}}{(ua_i)A_2(ua_i)^{-1}}{(va_j)A_3(va_j)^{-1}}{(va_j)A_4(va_j)^{-1}}{ Here, $a_i \in A_1 \sqcup A_2; a_j \in A_3 \sqcup A_4$} 
\end{scriptsize}
\end{scriptsize} \end{tikzpicture} 
\end{center}

Without loss of generality assume that $a_i = a_1 \in A_1$. Hence, $\mathbf{X''}$ can be given by -
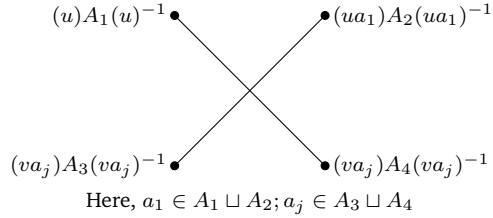
\begin{center}
\begin{tikzpicture} \begin{scriptsize}
\begin{scriptsize}
\Tw{5}{0}{(u)A_1(u)^{-1}}{(ua_1)A_2(ua_1)^{-1}}{(va_j)A_3(va_j)^{-1}}{(va_j)A_4(va_j)^{-1}}{ Here, $a_1 \in A_1 \sqcup A_2; a_j \in A_3 \sqcup A_4$} 
\end{scriptsize}
\end{scriptsize} \end{tikzpicture} 
\end{center}

Now we will give a collapse-expand path from $T_{\mathbf{X''}}$ to $T_{\mathbf{X'}}$ contained inside $\mathcal{M}_4$.
\begin{enumerate}[wide, labelwidth=!, labelindent=0pt]
\item \textbf{Collapse} the edge adjacent to the vertex labeled by $uA_1u^{-1}$ of $T_{\mathbf{X''}}$, $G_4$-equivariantly.
\item In the resulting tree choose the vertex labeled by $uA_2u^{-1}$ (instead of the vertex labeled by $(ua_1)A_2(ua_1)^{-1}$) from the $(u)A_1(u)^{-1}*(ua_1)A_2(ua_1)^{-1}$-minimal subtree to observe a fundamental domain where the non-trivial vertices are labeled by \\ $\{uA_1u^{-1}, uA_2u^{-1}, (va_j)A_3(va_j)^{-1}, (va_j)A_4(va_j)^{-1}\}$. Observe that this is a tree in $\mathcal{M}_4$
\item \textbf{Expand} the vertex labeled by $uA_1u^{-1}$, so that in the resulting tree there is a fundamental domain containing the vertices labeled by \\ $\{uA_1u^{-1}, uA_2u^{-1}, (va_j)A_3(va_j)^{-1}, (va_j)A_4(va_j)^{-1}\}$.

\item Starting from the above tree we will follow similar collapse-expand path (described above) to connect it to a tree containing a fundamental domain in which the vertices are labeled by \\ $\{uA_1u^{-1}, uA_2u^{-1}, vA_3v^{-1}, vA_4v^{-1}\}$. This tree is $G_4$-equivariantly isometric to $T_{\mathbf{X'}}$.

\end{enumerate}

Now consider the graph of groups $\mathbf{X}$ given by - 
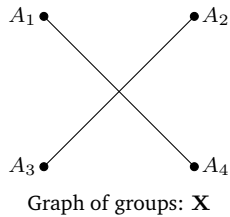
\begin{center}
\begin{tikzpicture} \begin{scriptsize}
\begin{scriptsize}
\Tw{5}{0}{A_1}{A_2}{A_3}{A_4}{ Graph of groups: $\mathbf{X}$} 
\end{scriptsize}
\end{scriptsize} \end{tikzpicture} 
\end{center}

By an induction on the word length of $u\in A_1*A_2$ and $w \in A_3*A_4$, and repeatedly following the collapse-expand moves described above, we can connect $\mathbf{X'}$ to $\mathbf{X}$ in $\mathcal{M}_4$. 

Now, consider any graph of groups $\mathbf{Z}\in \mathcal{M}_4$. Using lemma \ref{lemma:same_fund} we can find a collapse-expand path in $\mathcal{M}_4$ from $\mathbf{Z}$ to a graph of groups with same non trivial vertex groups as that of $\mathbf{Z}$, whose underlying graph is isomorphic to the underlying graph of $\mathbf{X}$. Hence, there is a collapse-expand path in $\mathcal{M}_4$ from $\mathbf{Z}$ to $\mathbf{X}$
\end{proof}

\begin{lemma}
The action $M^{12}\oplus M^{34} \acts \mathcal{M}_4$ is co-compact.
\end{lemma}

\begin{proof}
We will show that there is an outer automorphism in $M^{12}\oplus M^{34}$, which maps a graph of groups whose non-trivial vertex groups are given by $\{A_1, A_2, A_3, A_4\}$ to any $\mathbf{Z} \in \mathcal{M}_4$, where underlying graph of both graphs of groups are isomorphic.

Assume that the non trivial vertex groups of $\mathbf{Z}$ are given by \\ $\{uA_1u^{-1}, uA_2u^{-1}, vA_3v^{-1}, vA_4v^{-1}\}$, where $u\in A_1*A_2$ and $v \in A_3*A_4$ are reduced words given by \\
$u=a_{1i_1}^{\epsilon_1}a_{2i_1}a_{1i_2}a_{2i_2}....a_{1i_k}a_{2i_k}^{\epsilon_2};
v=a_{3i_1}^{\epsilon_3}a_{4i_1}a_{3i_2}a_{4i_2}....a_{3i_l}a_{4i_l}^{\epsilon_4}.$
Here, $a_{ji_m} \in A_j$ and $\epsilon_n \in \{0, 1\}$. Then the outer automorphism represented by the following automorphism is the required outer automorphism \\
 $(f_{A_2}^{a_{2i_k}})^{\epsilon_2}(f_{A_1}^{a_{1i_k}})
 ...
 (f_{A_2}^{a_{2i_1}})(f_{A_1}^{a_{1i_1}})^{\epsilon_1}
(f_{A_4}^{a_{4i_l}})^{\epsilon_4}(f_{A_3}^{a_{3i_l}})
 ...
 (f_{A_4}^{a_{4i_1}})(f_{A_3}^{a_{3i_1}})^{\epsilon_3}
 \in M^{12}\oplus M^{34}.
$
\end{proof}

\begin{lemma}
Let $T \in \mathcal{SPD}$ be a tree. Let, $T^{ij}$ denote the $A_i*A_j$-minimal subtree of $T$. If $i_1, i_2, i_3, i_4 \in \{1,...,n\}$ are distinct integers, then $T^{i_1i_2} \cap T^{i_3i_4}$ is homeomorphic to a connected subset of a line segment.
\end{lemma}

\begin{proof}
We will prove this by contradiction. Let, $i_1, i_2, i_3, i_4 \in \{1,...,n\}$ be distinct integers such that $I :=T^{i_1i_2} \cap T^{i_3i_4} \ne \phi$. Let, $v \in I$ be a vertex of valence greater than $2$. 

\begin{enumerate}[wide, labelwidth=!, labelindent=0pt]
\item Since, $v$ is a part of $T^{i_1i_2}$, the uniqueness of the minimal subtree $T^{i_1i_2}$ forces  the stabilizer subgroup of $v$ to be either a conjugate of $A_{i_1}$ or a conjugate of $A_{i_2}$, where the conjugating element belongs to the $A_{i_1}*A_{i_2}$. 
\item Similarly, $v$ is a part of $T^{i_3i_4}$. So, the stabilizer subgroup of $v$ is either a conjugate of $A_{i_3}$ or a conjugate of $A_{i_4}$, where the conjugating element belongs to the $A_{i_3}*A_{i_4}$.
\
\end{enumerate} 
We arrive at a contradiction. The intersection, $I$, cannot have a vertex of valence greater than $2$.
\end{proof}

Now we will construct a map from $\mathcal{SPD}(G_4,\mathscr{H})$ to $\mathcal{M}_4$.

\begin{definition}
Consider the map
\begin{gather*}L_4: \mathcal{SPD}^0(G_4,\mathscr{H}) \rightarrow \mathcal{M}^0_4\\
T \mapsto \overline{T},\end{gather*}
where $\overline{T}$ is described as follows:
\begin{enumerate}[wide, labelwidth=!, labelindent=0pt]
\item If $T \in \mathcal{M}_4$, then $\overline{T}:=T$

\item If $T \notin \mathcal{M}_4$, then $T^{12} \cap T^{34} $ can be empty, a point or a line segment.
\begin{enumerate}[wide, labelwidth=!, labelindent=0pt]
\item If $T^{12} \cap T^{34} = \phi$, and $J$ is the segment that realizes the shortest distance between $T^{12}$ and $T^{34}$ in $T$, then let us denote the vertex $ T^{12} \cap J$ by $v_{12}$ and the vertex $T^{34} \cap J$ by $v_{34}$. With these notations $\overline{T}$ will be the Bass-Serre tree of a graph of groups $\overline{\mathbf{X}}$ , where the vertex groups of $\overline{\mathbf{X}}$ are associated to specific vertex labels of $T^{12}$ and $T^{34}$. It is worth noting that the fundamental group of $\overline{X}$ is $A_1*A_2*A_3*A_4$, as $\overline{X}$ is composed of two subgraphs having fundamental groups $A_1*A_2$ and $A_3*A_4$, connected by an edge. $\overline{X}$ and $\overline{T}$ is described below-

Select the vertex labelling of the fundamental domain of $A_1 \free A_2 \acts T^{12}$ ( resp., $A_3 \free A_4 \acts T^{34}$) containing $v_{12} (\text{resp., }v_{34})$. If $v_{12} (\text{resp., }v_{34})$ is a part of two distinct fundamental domains of $T^{12} (\text{resp., }T^{34})$, then select the fundamental domain closer to the vertices labeled by $A_1 \text{ and } A_2 (\text{resp., }A_3 \text{ and } A_4)$. If $\mathbf{X}$ is the graph of groups corresponding to $T$, then $\mathbf{\overline{X}}$ is the graph of groups whose underlying graph is isomorphic to $\mathbf{X}$ and the vertex groups are replaced by the selected vertex groups without varying the conjugacy classes. $\overline{T}$ is the Bass-Serre tree of $\overline{X}$.

\begin{remark}
If we choose the other fundamental domain instead of the one described above we will get a $G_4$-equivariantly isometric graph of groups.
\end{remark}

\item If $T^{12} \cap T^{34}$ is a single point, then we will use strategy similar to the previous case to construct $\overline{T}$. In this case $v_{12} = v_{34}$.

\item If $I := T^{12} \cap T^{34}$ is a line segment, then take $v_{12} = v_{34}$ to be the midpoint of the segment $I$. We follow case 1 for the rest of the construction.
\end{enumerate}
\end{enumerate}
\end{definition}

\begin{lemma}
$L_4$ can be extended to a continuous, Lipschitz map - \\$L_4: \mathcal{SPD}(G_4,\mathscr{H}) \rightarrow \mathcal{M}_4$
\end{lemma}

\begin{proof}
If $T_1, T_2 \in \mathcal{SPD}^{0}$ are two trees such that $d_{\mathcal{SPD}}(T_1, T_2) = 1$, \\
then  we will show that $d_{\mathcal{SPD}}(L_4(T_1), L_4(T_2)) = 1.$ 

Without loss of generality, assume that $T_1$ can be obtained from $T_2$ by a series of collapse moves. Note that for $T_1 \in \mathcal{SPD}(G_4,\mathscr{H})$, there can be at most two edge-orbit collapses. Due to a collapse move the isomorphism types of underlying graphs $\mathbf{X_1}$ and $\mathbf{X_2}$ change. Since, the underlying graph of $\mathbf{\overline{X}_1}$ (resp., $\mathbf{\overline{X}_2}$) is isomorphic to $\mathbf{X_1}$ (resp., $\mathbf{X_2}$). So, the isomorphism types of the underlying graphs of $\mathbf{\overline{X}_1}$ and $\mathbf{\overline{X}_2}$ differ. Two different situations can happen with the non-trivial vertex groups:
\begin{enumerate}[wide, labelwidth=!, labelindent=0pt]
\item The vertex groups of $\mathbf{\overline{X}_1}$ and $\mathbf{\overline{X}_2}$ are same. Then the distance between them in $\mathcal{SPD}$ is $1$.
\item One of the vertex groups in $T^{12}$ (resp. $T^{34}$) is different between $\mathbf{\overline{X}_1}$ and $\mathbf{\overline{X}_2}$. Hence, a vertex group conjugate to $A_1$ or $A_2$ may differ (but not both) and a vertex group conjugate to $A_3$ or $A_4$ may differ (but not both). Due to uniqueness of $T^{12}$ and $T^{34}$ in $T_1$ and $T_2$ and by the construction of $\overline{T}_1$ and $\overline{T}_2$, there is a graph of group $\mathbf{\overline{X'}_2}$ equivalent to $\mathbf{\overline{X}_2}$ having same vertex group as $\mathbf{\overline{X}_1}$. So, this is similar to the previous case and the distance between $\mathbf{\overline{X}_1}$ and $\mathbf{\overline{X}_2}$ in $\mathcal{SPD}$ is $1$.
\end{enumerate}

Hence, the map $L_4$ can be linearly extended to $\mathcal{M}_4^1$ and $\mathcal{M}_4$ such that \\$d_{\mathcal{SPD}}(T_1, T_2) = 1 \implies d_{\mathcal{SPD}}(L_4(T_1), L_4(T_2)) = 1$. 
\end{proof}

\begin{corollary}\label{cor:M4_undistorted}
$\mathcal{M}_4$ is a quasi isometrically embedded sub-complex of $\mathcal{SPD}(G_4, \mathscr{H})$. Hence, $M^{12} \oplus M^{34}$ is undistorted in $\mathsf{Out}(A_1*A_2*A_3*A_4)$.
\end{corollary}

\section{Summary}\label{section:summary}

We will summarize our work together to give a summary of the proof of theorem \ref{thm:main} in this section.

\begin{proof}[Proof of theorem \ref{thm:main}]

\begin{enumerate}[wide, labelwidth=!, labelindent=0pt]

\item Finiteness of $\Gamma_2$ follows from corollary \ref{cor:G2_fin}

\item Hyperbolicity of $\Gamma_3$ follows from corollary \ref{cor:G3_hyp}

 \item $\Gamma_4$ is thick of order at most $1$. We will take the assistance of the following table to list the subgroups. The subgroups relevant to our discussions are $H^{12}$ (subgroup generated by the cells shaded in blue), $H^{34}$ (subgroup generated by the cells shaded in red), definition \ref{defn:H^ij};  $M^{12} \oplus M^{34}$, definition \ref{defn:M1234} (subgroup generated by the cells shaded in green). For a tabular representation refer to table \ref{tab:G4}. We will list the reasons whose combination make $\Gamma_4'$ thick of order at most $1$. 

 \begin{enumerate}[wide, labelwidth=!, labelindent=0pt]
 \item $ \la H^{12}, H^{34}, M^{12} \oplus M^{34}\ra \geq \Gamma_4'$.
 
 \item $H^{12}, H^{34}$ are undistorted in $\Gamma_4'$ (corollary \ref{cor:hij_undistorted}).  $M^{12} \oplus M^{34}$ is undistorted in $\Gamma_4$ (corollary \ref{cor:M4_undistorted}).
  
 \item Proposition \ref{prop:hij_sum} proves $H^{12}, H^{34}$ is at most zero thick.  Corollaries \ref{lemma:M1234thick}, \ref{cor:M4_undistorted} prove $M^{12} \oplus M^{34}$ is at most zero thick.
 
 \item Remark \ref{rem:M4_connection} proves
 $H^{12}, H^{34}, M^{12} \oplus M^{34}$ are thickly connected.

 \end{enumerate}

\item For $n > 4$, $\Gamma_n$ is thick of order at most $1$, when each $A_i$ is finite. For notations refer to definitions \ref{defn:H^ij}, \ref{defn:nij}. For a tabular representation refer to table \ref{tab:GN}. We will list the reasons whose combination make $\Gamma_n' \leq \Gamma_n$ thick of order at most $1$.
 \begin{enumerate}[wide, labelwidth=!, labelindent=0pt]
 \item When, $i \ne j$ and $i, j \in \{1,..., n\}$, then $H^{ij}$ generate $\Gamma_n'$ 
  
 \item When $i_1, i_2, i_3, i_4$ are all distinct integers, then $H^{i_1i_2}$s are undistorted in $\Gamma_n'$ (corollary \ref{cor:hij_undistorted}). $\la N^{i_1i_2}, N^{i_3i_4}\ra$ is undistorted in $\Gamma_n'$ (corollary \ref{cor:nij_undistorted})
 
 \item When $i_1, i_2, i_3, i_4$ are all distinct integers, then, proposition \ref{prop:hij_sum}, corollary \ref{cor:nij_zerothick} proves $H^{i_1i_2}, H^{i_3i_4}, \la N^{i_1i_2}, N^{i_3i_4}\ra$ are zero thick.
 
 \item When $i_1, i_2, i_3, i_4$ are all distinct integers, then corollary \ref{cor:thickly_connected} proves
 $H^{i_1i_2}, H^{i_3i_4}, \la N^{i_1i_2}, N^{i_3i_4}\ra$ are thickly connected.

 \end{enumerate}

\end{enumerate}
\end{proof}

\nocite{*}

\bibliography{mybib}

\end{document}